\documentclass{tac}

\usepackage{xy}
\usepackage{euscript}
\usepackage{tabularx}
\usepackage{booktabs}
\usepackage{tikz}
\usetikzlibrary{decorations.pathmorphing,shapes}
\usepackage{extarrows}
\usepackage{amssymb}
\usepackage{float}
\usepackage{mathalfa}
\usepackage{footnote}

\input diagxy

\usepackage[colorlinks=true]{hyperref}
\hypersetup{allcolors=[rgb]{0.1,0.1,0.4}}

\newcommand{\xupdownarrow}[1]{%
  {\left\Updownarrow\vbox to #1{}\right.\kern-\nulldelimiterspace}
}

\makeatletter

\newcounter{sarrow}
\newcommand\xrsquigarrow[1]{%
\stepcounter{sarrow}%
\mathrel{\begin{tikzpicture}[baseline= {( $ (current bounding box.south) + (0,-0.5ex) $ )}]
\node[inner sep=.5ex] (\thesarrow) {$\scriptstyle #1$};
\path[draw,<-,decorate,
  decoration={zigzag,amplitude=0.7pt,segment length=1.2mm,pre=lineto,pre length=4pt}] 
    (\thesarrow.south east) -- (\thesarrow.south west);
\end{tikzpicture}}%
}
 \usetikzlibrary{calc}

\author{Jovana Obradovi\' c}

\thanks{I would like to thank Pierre-Louis Curien and Fran\c cois Lamarche for  useful discussions.}

\address{IRIF, Universit\' e Paris Diderot - Paris 7 \\ Case 7014 
75205 PARIS Cedex 13\\
}

\title {Monoid-like definitions of cyclic operad}

\copyrightyear{2017}


\eaddress{jovana.obradovic@inria.fr}

\newtheorem{thm}{Theorem}
 
\newtheorem{lem}{Lemma}
\newtheorem{con}{Convention}
\newtheorem{cor}{Corollary}

\newtheoremrm{rem}{Remark}

\mathrmdef{Hom}
\mathbfdef{Set}

\begin{document}

\maketitle
\begin{abstract}
 Guided by the microcosm principle of Baez-Dolan and by the algebraic definitions
of operads of Kelly and Fiore, we introduce two ``monoid-like'' definitions of cyclic
operads, one for the original, ``exchangable-output'' characterisation of Getzler-Kapranov, and the other for the alternative
``entries-only'' characterisation, both within the category of Joyal's species of structures. Relying on a result of Lamarche on descent   for species,
we use these “monoid-like” definitions to prove the equivalence between the ``exchangable-output''
and ``entries-only'' points of view on cyclic operads.
\end{abstract}


\section*{Introduction}\label{sec-Introduction}

A species of structures $S$ associates to each finite set $X$ a set $S(X)$ of combinatorial structures on $X$ that are invariant under renaming the elements of $X$ in a way  consistent with composition of such renamings. The notion, introduced in combinatorics by Joyal in \cite{joyal}, has been set up to provide a
description of discrete structures that is independent from any specific format these structures could be presented in.  For example, $S(X)$ could be the set of graphs whose vertices are given by $X$,  the set of  all permutations of $X$, the set of all subsets of $X$, etc.  Categorically speaking, a species of structures is simply a functor  ${\EuScript C}:{\bf Bij}\rightarrow{\bf Set}$, wherein ${\bf Set}$ is the category of sets and functions, and ${\bf Bij}$ is the category of finite sets and bijections.
Species can be combined in various ways into new species and these ``species algebras"   provide the category of species  with different notions of  ``tensor product''. Some of these products  allow   to redefine operads internally to the category of species,   as monoids. A definition  given in this framework is usually referred to as {\it algebraic}. A definition of an operad as a collection of abstract operations of different arities
that can be suitably  composed will be called  {\em componential}  in this paper. \\
\indent  Kelly \cite{kelly} has given
an algebraic definition of a symmetric operad corresponding to the original
componential definition of May \cite{GILS}.
  This definition is  referred to as the {\em monoidal definition} of operads, since the involved product on species bears a monoidal structure.  
The second definition, which characterises operads with partial composition, has been recently established by Fiore in \cite{fiore}. The {\em pre-Lie product} of Fiore's definition  is not monoidal, but the inferred structure arises by the same kind of principle as the one reflecting a specification of a monoid  in a monoidal category (which is why we call this definition  the  monoid-{\em like} definition of operads). This is a typical example of what has been called the {\em microcosm principle} by Baez and Dolan in \cite{micro}.  The principle tells that 
{\em certain algebraic structures can be defined in any category equipped with a categorified version of the same structure},
and the instance with monoids, presented in Table 1 below, can serve as a guide when seeking the most general way to internalize different algebraic structures. 
 
\begin{center} {\small 
  \begin{tabular}{rcc}  
    \toprule
       &  \textsc{Monoidal category}  ${\bf M}$  & \textsc{Monoid} $M\in {\bf M}$ \\
    \midrule
    {\small{\textsc{product\enspace}  }}     &{\small{$\otimes:{\bf M}\times{\bf M}\rightarrow {\bf M}$}}    & {\small{$\mu: M\otimes M\rightarrow M$}}      \\[0.1cm]
    {\small{\textsc{unit\,\,\,\,\,}}}  &{\small{$1\in {\bf M}$}}    & {\small{$\eta:1\rightarrow M$}}      \\[0.1cm]
{\small\begin{tikzpicture}
\node (A)  at (0,0.5) {};
\node (F)  at (1.55,0.83) {   {\small{\textsc{ associativity}  }}  };
\node (B) at (2,0.5) {};
\node (C) at (2,0) {};
\node (D) at (0,0) {};
\node (E) at (1.2,-0.5) {};
\end{tikzpicture}  }   & {\small\begin{tikzpicture}
\node (A)  at (0,1) {};
\node (F)  at (2,0.83) {   {\small{$\alpha_{x,y,z}:(x\!\otimes\! y)\!\otimes\!z\!\rightarrow\! x\!\otimes(y\!\otimes\!z)$ }}  };
\node (B) at (2,1) {};
\node (C) at (2,0) {};
\node (D) at (0,0) {};
\node (E) at (1.2,-0.5) {};
\end{tikzpicture}  }     &{\small
\begin{tikzpicture}[scale=1.5]
\node (A)  at (0,1) {$(x\otimes x)\otimes x$};
\node (B) at (2.3,1) {$x\otimes (x\otimes x)$};
\node (C) at (2.3,0.2) {$x\otimes x$};
\node (D) at (0,0.2) {$x\otimes x$};
\node (E) at (1.15,-0.25) {$x$};
\path[->,font=\scriptsize]
(A) edge node[above]{$\alpha_{x,x,x}$} (B)
(B) edge node[right]{$\mu$} (C)
(A) edge node[left]{$\mu\otimes {\it id}$} (D)
(D) edge node[below]{$\mu$} (E)
(C) edge node[below]{$\mu$} (E);
\end{tikzpicture}  }   \\[0.1cm]
{\small\begin{tikzpicture}
\node (A)  at (0,0.5) {};
\node (F)  at (1.14,0.53) {  {\small{\textsc{left unit} }}  };
\node (C) at (1.05,-0.05) {{\small{\textsc{right unit} }}  };
\node (D) at (0,0) {};
\node (E) at (1.2,-0.5) {};
\end{tikzpicture}  }
 &{\small\begin{tikzpicture}[scale=1.5]
\node (A)  at (0,0.5) {};
\node (F)  at (1.05,0.2) { {\small{$\lambda_x:1\otimes x\rightarrow x$}}};
\node (C) at (1.05,-0.2) {{\small{$\rho_x:x\otimes 1\rightarrow x$}}  };
\node (D) at (0,0) {};
\node (E) at (1.2,-0.5) {};
\end{tikzpicture}  } & {\small\begin{tikzpicture}[scale=1.5]
\node (A)  at (0.5,1) {$1\otimes x$};
\node (B) at (2,1) {$x\otimes x$};
\node (C) at (3.5,1) {$x\otimes 1$};
\node (D) at (2,0.2) {$x$};
\path[->,font=\scriptsize]
(A) edge node[above]{\scriptsize $\eta\!\otimes\!{\it id}$} (B)
(C) edge node[above]{${\it id}\!\otimes\!\eta$} (B)
(A) edge node[left,yshift=-0.1cm]{$\lambda_x$} (D)
(C) edge node[yshift={-0.13cm},right]{$\rho_x$} (D)
(B) edge node[right]{$\mu$} (D);
\end{tikzpicture}}      \\
    \bottomrule
  \end{tabular}

\vspace{0.25cm}
Table 1. A monoid in a monoidal category}
\end{center}

\indent In this paper we follow the microcosm principle in order to give two algebraic (monoid-{\em like}) definitions of {\em cyclic operads}, introduced by   Getzler-Kapranov \cite{Getzler:1994pn}. The enrichment of the symmetric operad structure determined by the definition of a cyclic operad is provided by adding to the action of permuting the inputs of an operation, an  action of interchanging its output  with one of the inputs, in a way compatible with operadic composition. 
The fact that  operations can now be composed along  inputs that ``used to be outputs" and outputs that ``used to be  inputs"  leads to another point of view on cyclic operads, in  which an operation, instead of having inputs and an (exchangeable) output, now has ``entries", and it can be composed with another operation along any of them.  Such  an {\em entries-only} componental definition is \cite[Definition 48]{modular}.  By contrast, we  refer to  definitions based on  describing cyclic operads as symmetric operads with extra structure as {\em exchangeable-output} ones. One such definition is \cite[Proposition 42]{opsprops}.\\
\indent The algebraic definitions that we deliver correspond to these two approaches for defining  cyclic operads via components. They are moreover given in a {\em non-skeletal} version, which means that the entries/inputs of operations are labeled by arbitrary finite sets, as in  \cite[Definition 48]{modular}, as opposed to  the labeling  by natural numbers   in {\em skeletal} variants, as in \cite[Proposition 42]{opsprops}. Therefore, we first  propose a non-skeletal version of  \cite[Proposition 42]{opsprops}. We additionally give two   proofs of  the equivalence between the  entries-only and exchangeable-output approaches (which, to the author's knowledge, has been taken for granted in the literature), one by comparing the usual  definitions in components (Theorem \ref{2}), and the other one by comparing two algebraic definitions (Theorem \ref{3}). Since algebraic definitions are arguably more conceptual descriptions (that the componential ones) of what a cyclic operad is, we also  point out Theorem \ref{3} as the main result of the paper. Together with the proof of the equivalence between the componential and algebraic definitions of entries-only cyclic operads (Theorem \ref{1}), this makes a sequence of equivalences that also justifies  the algebraic definition of exchangeable-output cyclic operads. An overview of the definitions that we introduce and the correspondences that we make between them is given in Table 2 below.

 \begin{center}
{{\footnotesize \begin{tabular}{rrcccl}  
    \toprule
&    \textsc{{Entries-only}}  &&&&  \textsc{{Exchangeable-output}} \\
    \midrule
     {\textsc{\small{Componential}}} &     \small{ Definition \ref{entriesonly}}&& $\xLeftrightarrow[\enspace \mbox{\small{Theorem \ref{2}}} \enspace]{}$  & & {\small{Definition \ref{exoutput}}} \\[0.1cm]
  	& {\small{Theorem \ref{1}}} $\xupdownarrow{0.5cm}$&& &&   \\[0.1cm]
    {\textsc{\small{Algebraic}}} &   \small{Definition \ref{copeo}}  & & $\xLeftrightarrow{\enspace \mbox{\small{Theorem \ref{3}}}\enspace}$ &&  \small{Definition \ref{dddd}}    \\
    \bottomrule
  \end{tabular}

\vspace{0.25cm}
Table 2. The outline of the paper}}
\end{center}

\indent The plan of the paper is as follows.  Section 1 is a review of the basic elements of the theory of species of structures. In Section 2 we recall   the existing algebraic definitions of  operads and indicate the microcosm principle behind them. Section 3 will be devoted to the introduction of the  algebraic definitions of cyclic operads (Definition \ref{copeo} and Definition \ref{dddd}) and of the componential non-skeletal version of  \cite[Proposition 42]{opsprops} (Definition \ref{exoutput}). Here we also prove Theorem \ref{1} and Theorem \ref{2}. In Section 4, we give the proof of Theorem \ref{3}.

\paragraph{Notation and conventions.} 
This paper is about non-skeletal  cyclic operads with units, introduced  in  {\bf Set}. \\
\indent We shall use two different notions of union. In {\bf Set}, for finite sets $X$ and $Y$, $X+Y$ will denote the coproduct (disjoint union) of $X$ and $Y$ (constructed in the usual way by tagging $X$ and $Y$, by, say, $1$ and $2$) and we shall use the notation $\Sigma_{i\in I} X_i$ (resp. $\Pi_{i\in I}X_i$) for the coproduct (resp. the Cartesian product) of  the family of sets $\{X_i\,|\, i\in I\}$. In order to  avoid making distinct copies of $X$ and $Y$ before taking the union, we take the usual convention of assuming that they are already disjoint.
In {\bf Bij}, we shall denote the {\em ordinary} union of {\em already disjoint} sets $X$ and $Y$ with $X\cup Y$.  \\
\indent If $f_1:X_1\rightarrow Z_1$ and $f_2:X_2\rightarrow Z_2$ are functions  such that $X_1\cap X_2=\emptyset$ and $Z_1\cap Z_2=\emptyset$, $f_1\cup f_2:X_1\cup X_2\rightarrow Z_1\cup  Z_2$ will denote the function defined as $f_1$ on $X_1$ and as $f_2$ on $X_2$. If $Z_1=Z_2=Z$, we shall write $[f_1,f_2]:X_1\cup X_2\rightarrow Z$ for the  function defined in the same way. Accordingly, for the corresponding functions between disjoint unions, we shall  write $f_1+ f_2:X_1+ X_2\rightarrow Z_1+  Z_2$ and $[f_1,f_2]:X_1+ X_2\rightarrow Z$.\\
\indent A decomposition of a finite set $X$ is a family $\{X_i\}_{i\in I}$ of (possibly empty) pairwise disjoint subsets of $X$ such
that their (ordinary) union   is $X$.\\
\indent For a bijection $\sigma:X'\rightarrow X$ and  $Y\subseteq X$, we  denote with $\sigma|^{Y}$ the corestriction of $\sigma$ on $\sigma^{-1}(Y)$. \\
\indent We shall work to a large extent with compositions
of multiple canonical natural isomorphisms between functors. In order for such compositions not to look too
cumbersome, we shall often omit their indices.
 \section{The category of species of structures} 
The content of this section is to a great extent a review and a gathering of material coming from \cite{species}. Certain isomorphisms, whose existence has been claimed in \cite{species}, will be essential for subsequent  sections and we shall construct them explicitely.
\subsection{Definition of species of structures}
The notion of species of structures that we fix as primary corresponds to functors underlying non-skeletal cyclic operads.
\begin{definition}
A  species  (of structures) is a functor $S:{\bf Bij}^{op}\rightarrow {\bf Set}$.
\end{definition}
In the sequel, we shall refer to the functor category ${\bf Set}^{{\bf Bij}^{op}}$ as the {\em category of species} and we shall denote it with ${\bf Spec}$. For an arbitrary finite set $X$, an element $f\in S(X)$ will be referred to as an {\em S-structure}.\\[0.1cm]
\indent Notice that if $S$ is a species and $\sigma:Y\rightarrow X$ a bijection, then $S(\sigma):S(X)\rightarrow S(Y)$ is necessarily a bijection (with the inverse $S(\sigma^{-1})$). 
\begin{con}
For $f\in S(X)$ and a bijection $\sigma:Y\rightarrow X$, we say that $\sigma$ renames the variables of $X$ to (appropriate) variables of $Y$. In particular, if $\sigma:X\backslash\{x\}\cup\{y\}\rightarrow X$ is identity on $X\backslash\{x\}$ and $\sigma(y)=x$, we say that $\sigma$ renames $x$ to $y$, and if $\tau:X\rightarrow X$ is identity on $X\backslash\{x,z\}$ and   $\tau(x)=z$ and $\tau(z)=x$, we say that $\tau$ exchanges $x$ and $z$.
\end{con}
As an example of concrete species we give the following family, since it will be essential for the treatement of operadic units in the subsequent sections. The species $E_n$, where $n\geq 0$, called the {\em cardinality} $n$ {\em species},  is defined by setting
$$E_n(X) = \left\{ 
	\begin{array}{cl}
		\{X\}& \mbox{if }  X \mbox{ has $n$ elements},  \\[0.1cm]
		\emptyset & \mbox{otherwise}.  \\[0.1cm]
	\end{array}
\right.$$
\indent An  isomorphism between species is  simply   a natural isomorphism between functors. If there exists an isomorphism from $S$ to $T$, we say that they $S$ and $T$ are isomorphic and we write $S\simeq T$.
 
 \subsection{Operations on species of structures}
We now  recall   operations on species and their properties. Categorically speaking, every binary operation   is a bifunctor of the form ${\bf Spec}\times {\bf Spec}\rightarrow {\bf Spec}$ and every unary operation is a  functor of the form ${\bf Spec}\rightarrow {\bf Spec}$. Every property of an operation holds up to isomorphism of species. \\[0.1cm]
\indent We start with the analogues of the arithmetic operations of addition and multiplication.
\begin{definition}\label{sumprod} Let $S$ and $T$ be species, $X$ an arbitrary finite set and $\sigma:Y\rightarrow X$ a bijection. The  sum-species of $S$ and $T$ is the species $S+T$ defined  as   $$(S+T)(X)=S(X)+  T(X)  $$ and   $$(S+T)({\sigma})(f)=\left\{ 
	\begin{array}{cl}
		S({\sigma})(f) & \mbox{if }  f \in S(X) \\[0.1cm]
		T({\sigma})(f)& \mbox{if }  f \in T(X)\,. \\[0.1cm]
	\end{array}
\right. $$

The   product-species of $S$ and $T$ is the species $S\cdot T$ defined as  \vspace{-0.1cm}$$(S\cdot T)(X)=\sum_{(X_1,X_2)}S(X_1)\times T(X_2), \vspace{-0.1cm}$$ where the sum is taken over all binary  decompositions $(X_1,X_2)$  of $X$. The action of $S\cdot T$ on  $\sigma$ is given as $$(S\cdot T)(\sigma)(f,g)=(S(\sigma_1)(f),T(\sigma_2)(g)),$$ where $\sigma_i=\sigma|^{X_i}$, $i=1,2.$
\end{definition}

 The isomorphisms from the following lemma  are  constructed straightforwardly.
\begin{lem} The addition and multiplication of species have the following properties.
\begin{itemize}
\item[a)] The operation of addition is associative and commutative.  
\item[b)]  The product of species is associative and commutative. The cardinality $0$
species $E_0$ is  neutral element  for the product of species. Therefore, for all species $S$,  $S\cdot E_0\simeq E_0\cdot S\simeq S$.
\end{itemize}
\end{lem}
\begin{con} We extend the notation $f_1+f_2$ and $[f_1,f_2]$ (see the paragraph ``Notations and conventions" in Introduction) from functions to natural transformations. For natural transformations $\psi_i:S_i\rightarrow T_i$, $i=1,2$, $\psi_1+\psi_2:S_1+S_2\rightarrow T_1+T_2$ will denote the natural transformation determined by $(\psi_1+\psi_2)_X={\psi_1}_X+{\psi_2}_X$. For natural transformations $\kappa_i:S_i\rightarrow U$, $i=1,2$,  $[\kappa_1,\kappa_2]:S_1+S_2\rightarrow U$ will denote the natural transformation defined as $[\kappa_1,\kappa_2]_X=[{\kappa_1}_X,{\kappa_2}_X]$. With $i_l$ and $i_r$ we shall denote the insertion natural transformations $i_l:S\rightarrow S+T$ and $i_r:T\rightarrow S+T$, respectively.  \end{con}
\indent Next we recall the operation  corresponding to the operation of  substitution.
\begin{definition}
Let $S$ and $T$ be species, $X$  a finite set, $\sigma:Y\rightarrow X$ a bijection and   $D(X)$  the set of all decompositions of $X$. The substitution product of $S$ and $T$ is the species $S\circ T$ defined as \vspace{-0.1cm}$$(S\circ T)(X)=\sum_{\pi\in D(X)}\bigg(S(\pi)\times\prod_{p\in\pi}T(p)\bigg).\vspace{-0.1cm}$$
For an arbitrary $h=(\pi,f,(g_p)_{p\in \pi})\in (S\circ T)(X)$, the action of $S\circ T$ on $\sigma$ is defined by $$(S\circ T)(\sigma)(h)=(\overline{\pi},\overline{f},(\overline{g}_{\overline{p}})_{\overline{p}\in\overline{\pi}}),$$ where $\overline{\pi}$ is the decomposition of $Y$ induced by $\sigma$, the bijection $\overline{\sigma}:\overline{\pi}\rightarrow \pi$ is induced by $\sigma$, $\overline{f}=S(\overline{\sigma})(f)$,    and finally,  for each $\overline{p}\in\overline{\pi}$,  $\overline{g}_{\overline{p}}=T(\sigma|^{p})(g_p)$.\end{definition}
The basic properties of the substitution product are given in the following lemma.
\begin{lem}\label{associativityofsubstitution}
The substitution product of species is associative   and has  the cardinality 1 species  $E_1$ as  neutral element.  
\end{lem}
Next comes the analogue of the  operation of derivation.
\begin{definition}\label{derivate}
 The  derivative of $S$ is the species $\partial S$ defined as $$(\partial S)(X)=S(X\cup\{\ast_X\}),$$ where $\ast_X\not\in X$. The action of $\partial S$ on $\sigma$ is defined as $$(\partial S)(\sigma)(f)=S(\sigma^{+})(f),$$ where $\sigma^{+}:Y\cup\{\ast_Y\}\rightarrow X\cup\{\ast_X\}$ is such that $\sigma^{+}(y)=\sigma(y)$ for $y\in Y$ and $\sigma^{+}(\ast_Y)=\ast_X$. We shall refer to $\sigma^{+}$ as the $\partial$-extension of $\sigma$.
\end{definition}
\indent  We now introduce a natural isomorphism that will be used for the algebraic version of the associativity axiom for entries-only cyclic operads. Let $f\in \partial\partial S(X)$ and let \vspace{-0.1cm}$$\varepsilon_X:X\cup\{{\ast_X},{\ast_{X\cup\{\ast_X\}}}\}\rightarrow X\cup\{{\ast_X},{\ast_{X\cup\{\ast_X\}}}\}\vspace{-0.1cm}$$ be the bijection that is identity on $X$ and such that $\varepsilon_X({\ast_X})={\ast_{X\cup\{\ast_X\}}}$ (and $\varepsilon_X(\ast_{X\cup\{\ast_X\}})={\ast_X}$). We define a natural transformation ${\texttt{ex}}_S:\partial(\partial S)\rightarrow\partial(\partial S)$ as \vspace{-0.1cm}$${{{\texttt{ex}}_S}_X}(f)=S(\varepsilon_X)(f).\vspace{-0.1cm}$$ 
We shall refer to ${\texttt{ex}}_S$ as the {\em exchange isomorphism}, since its components exchange the two distinguished elements (arising from the two-fold application of the operation of derivation). \\[0.1cm]
\indent The following lemma exibits isomorphisms between species that correspond to the  rules of {\em the derivative of a sum} and {\em the derivative of a product} of the classical differential calculus.

\begin{lem}\label{leibnitz}
For arbitrary species $S$ and $T$, the following properties hold:
\begin{itemize}
\item[a)] $\partial(S+T)\simeq\partial S+ \partial T$, and  
\item[b)] $\partial(S\cdot T)\simeq(\partial S)\cdot T+S\cdot(\partial T)$. 
\end{itemize}
\end{lem}
\begin{proof} {\em a)} The isomorphism $\Delta:\partial(S+T)\rightarrow\partial S+ \partial T$ is the identity natural transformation.\\[0.1cm]
 {\em b)} We define an isomorphism $\varphi:\partial(S\cdot T)\rightarrow(\partial S)\cdot T+S\cdot(\partial T)$. For a finite set $X$ we have {\small  $$\begin{array}{rcll}
\partial(S\cdot T)(X)&=&\sum_{(X_1,X_2)}\{(f,g)\,|\, f\in S(X_1)\mbox{ and } g\in T(X_2)\},&\\[0.1cm]
(\partial S\cdot T)(X)&=&\sum_{(X'_1,X'_2)}\{(f,g)\,|\, f\in (\partial S)(X'_1)\mbox{ and } g\in T(X'_2)\}, \mbox{ and}&\\[0.1cm] 
(S\cdot\partial T)(X)&=&\sum_{(X'_1,X'_2)}\{(f,g)\,|\, f\in S(X'_1) \mbox{ and } g\in (\partial T)(X'_2)\},&
\end{array}$$} 

\noindent where $(X_1,X_2)$ is an arbitrary decomposition of the set $X\cup\{\ast_X\}$, and   $(X'_1,X'_2)$ is an arbitrary decomposition of the set $X$.\\[0.1cm]
\indent If $(f,g)\in\partial (S\cdot T)(X)$, where $f\in  S(X_1)$ and $g\in T(X_2)$, and if $\ast_X\in X_1$, then $(X'_1,X'_2)=(X_1\backslash\{\ast_X\},X_2)$ is a decomposition of the set $X$ and we set
 $$
\varphi_X(f,g)= 
		(S(\sigma)(f),g), $$ where $\sigma:X_1\backslash\{\ast_X\}\cup\{\ast_{X'_1}\}\rightarrow X_1$ renames ${\ast_X}$ to ${\ast}_{X'_1}$.
 We do analogously if $\ast_X\in X_2$.\\[0.1cm]
\indent To define the inverse of $\varphi_X$, suppose that $(f,g)\in (\partial S\cdot T)(X)$, where $f\in (\partial S)(X'_1)$ and $g\in T(X'_2)$. The pair $(X'_1\cup\{\ast_{X'_1}\},X'_2)$ is then a decomposition of the set $X\cup\{\ast_{X'_1}\}$. Let $\tau:X'_1\cup\{\ast_X\}\rightarrow X'_1\cup\{\ast_{X'_1}\}$ be the renaming of $\ast_{X'_1}$ to $\ast_X$.  The pair $(X_1,X_2)=(X'_1\cup\{\ast_{X}\},X'_2)$ is now a decomposition of the set $X\cup\{\ast_X\}$ and we set  $$\varphi^{-1}_X(f,g)=(S(\tau)(f),g)\in \partial(S\cdot T)(X). $$  We proceed analogously for  $(f,g)\in(S\cdot\partial T)(X)$.
\end{proof}
We shall also need the family of isomorphisms from the following lemma.

\begin{lem}\label{derivativen}
For all $n\geq 1$,  $\partial E_n\simeq E_{n-1}$.
\end{lem}
\begin{proof}
For a finite set $X$ we have
$$\begin{array}{ccc}
\partial E_n(X)=\left\{ 
	\begin{array}{cl}
		\{X\cup\{\ast_X\}\}& \mbox{if }  |X|=n-1,  \\[0.1cm]
		\emptyset & \mbox{otherwise},  \\[0.1cm]
	\end{array}
\right. &\mbox{and }&
E_{n-1}(X)=\left\{ 
	\begin{array}{cl}
		\{X\}& \mbox{if }  |X|=n-1,  \\[0.1cm]
		\emptyset & \mbox{otherwise}.  \\[0.1cm]
	\end{array}
\right.
\end{array}$$
The isomorphism $\epsilon_n:\partial E_n\rightarrow E_{n-1}$ is defined as ${\epsilon_n}_X(X\cup\{\ast_X\})=X,$ for $|X|=n-1$. Otherwise,   ${\epsilon_n}_X$ is the empty function.
\end{proof}
Finally, we shall also use the following {\em pointing} operation on species.
\begin{definition}
Let $S$ be a species. The species $S^{\bullet}$, spelled $S$  dot, is defined as follows $$S^{\bullet}(X)=S(X)\times X.$$ For a pair $(f,x)\in S(X)\times X$, the action of $S^{\bullet}$ on a bijection $\sigma:Y\rightarrow X$ is given by  $$S^{\bullet}(\sigma)((f,x))=(S(\sigma)(f),\sigma^{-1}(x)).$$
\end{definition}
\begin{rem}
Observe that the distinguished element of an $S^{\bullet}$-structure belongs to the underlying set $X$, as opposed to the distinguished element of a $\partial S$-struture, which is always outside of $X$.
\end{rem}
To summarise, we list  below the isomorphisms between species that we shall use in the remaining of the paper. 
\begin{center}
  \begin{tabular}{ccc}  
    \toprule
   {\small\textsc{Name}}   &  {\small\textsc{Reference}} &   {\small\textsc{Description}} \\
    \midrule
    {\footnotesize{\textsc{associativity of} $\cdot$}}     &{\footnotesize{$\alpha_{S,T,U}:(S\cdot T)\cdot U\rightarrow S\cdot(T\cdot U)$}} &    {\footnotesize{$((f,g),h)\mapsto (f,(g,h))$}}      \\[0.1cm]
    {\footnotesize{\textsc{commutativity of} $\cdot$}}  &{\footnotesize{${\tt c}_{S,T}:S\cdot T\rightarrow T\cdot S$}} &    {\footnotesize{$(f,g)\mapsto (g,f)$}}      \\[0.1cm]
 {\footnotesize{\textsc{left unitor for} $\cdot$}}&  {\footnotesize{$\lambda_{S}:E_0\cdot S\rightarrow S$}} &    {\footnotesize{$(\{\emptyset\},f)\mapsto f$}}      \\[0.1cm]
 {\footnotesize{\textsc{right unitor for} $\cdot$}}&{\footnotesize{$\rho_{S}:S\cdot E_0\rightarrow S$}} &    {\footnotesize{$(f,\{\emptyset\})\mapsto f$}}      \\[0.1cm]
{\footnotesize{\textsc{exchange}}}  &{\footnotesize{${\tt ex}_S:\partial(\partial S)\rightarrow \partial(\partial S)$}} &    {\footnotesize{$f\mapsto S(\varepsilon)(f)$}}      \\[0.1cm]
{\footnotesize{\textsc{derivative of a sum}}}      &{\footnotesize{$\Delta_{S,T}:\partial(S+T)\rightarrow\partial S+ \partial T$}} &   {\footnotesize{Lemma \ref{leibnitz} (a)}}  \\[0.1cm]
{\footnotesize{\textsc{Leibniz rule}}}      &{\footnotesize{$\varphi_{S,T}:\partial(S\cdot T)\rightarrow(\partial S)\cdot T+S\cdot(\partial T)$}} &   {\footnotesize{Lemma \ref{leibnitz} (b)}}  \\[0.1cm]
{\footnotesize{$\epsilon_n$\textsc{-isomorphism}}}      &{\footnotesize{$\epsilon_n:\partial E_n\rightarrow E_{n-1}$}} &   {\footnotesize{Lemma \ref{derivativen}}}  \\[0.1cm]
    \bottomrule
  \end{tabular}

\vspace{0.25cm}
{\small Table 3. Canonical isomorphisms}
\end{center}

\section{Symmetric operads}
This part is a reminder on several definitions of symmetric operad. Our emphasis is on the use of the
microcosm principle of Baez and Dolan, which we illustrate by reviewing Fiore's
definition  in Section 2.2 below.
\subsection{Kelly-May definition} 
Kelly's monoidal definition \cite[Section 4]{kelly} is the algebraic version of the original definition of an operad, given by May in \cite{GILS}. In the non-skeletal setting, the operadic composition  of May's definition is
given by morphisms 
\begin{equation}\label{simult}\gamma_{X,Y_1,\dots Y_n}:S(X)\times S(Y_1)\times\cdots \times S(Y_n)\rightarrow S(Y_1\cup\cdots \cup Y_n)\,,\end{equation}
defined for non-empty finite set  $X$ and  pairwise disjoint finite sets $Y_1,\dots,Y_n$, where $n=|X|$, and the unit  ${\it id}_x\in S(\{x\})$, defined for all singletons $\{x\}$, which are subject to associativity,
equivariance and unit axioms. Morphisms $\gamma_{X,Y_1,\dots ,Y_n}$ are to be thought of as simultaneous insertions of $n$ operations   into an $n$-ary one, wherefore this kind of composition is referred to as simultaneous. \\
\indent To arrive to Kelly's definition, one first observes that Lemma \ref{associativityofsubstitution} can be reinforced to a stronger claim:  \begin{center}{\em $({\bf Spec},\circ,E_1)$ is a monoidal category.}\end{center}
A monoid in this category is  a triple $(S,\mu,\eta)$, where $S$ is a species and the natural transformations  $\mu:S\circ S\rightarrow S$ and $\eta:E_1\rightarrow S$, called the {\em multiplication} and the {\em unit} of the monoid, respectively,   satisfiy the coherence conditions given by the commutation of the following two diagrams
\begin{center}
\begin{tikzpicture}[scale=1.5]
\node (A)  at (0,1) {\footnotesize $(S\circ S)\circ S$};
\node (B) at (1.9,1) {\footnotesize $S\circ (S\circ S)$};
\node (C) at (3.8,1) {\footnotesize $S\circ S$};
\node (D) at (0,-0) {\footnotesize $S\circ S$};
\node (E) at (3.8,-0) {\footnotesize $S$};
\path[->,font=\scriptsize]
(A) edge node[above]{$\alpha^{\circ}$} (B)
(B) edge node[above]{${\it id}\circ\mu$} (C)
(A) edge node[left]{$\mu\circ {\it id}$} (D)
(D) edge node[below]{$\mu$} (E)
(C) edge node[right]{$\mu$} (E);
\end{tikzpicture}
\enspace\enspace\enspace
\begin{tikzpicture}[scale=1.5]
\node (A)  at (0,1) {\footnotesize $E_1\circ S$};
\node (B) at (1.9,1) {\footnotesize $S\circ S$};
\node (C) at (3.8,1) {\footnotesize $S\circ E_1$};
\node (D) at (1.9,-0) {\footnotesize $S$};
\node (E) at (1.9,-0.211) {};
\path[->,font=\scriptsize]
(A) edge node[above]{$\eta\circ {\it id}$} (B)
(C) edge node[above]{${\it id}\circ\eta$} (B)
(A) edge node[left,yshift=-0.1cm]{$\lambda^{\circ}$} (D)
(C) edge node[right,yshift=-0.1cm]{$\rho^{\circ}$} (D)
(B) edge node[right]{$\mu$} (D);
\end{tikzpicture}
\end{center}
in which $\alpha^{\circ}, \lambda^{\circ}$ and $\rho^{\circ}$ denote the associator, left  and right unitor of $({\bf Spec},\circ,E_1)$, respectively.\\
\indent An element $(f,g_1,\dots,g_n)\in S(X)\times S(Y_1)\times\cdots \times S(Y_n)$ determines the element  $(\pi,f,(g_i)_{1\leq i\leq n})\in(S\circ S)(Y_1\cup\cdots\cup Y_n)$, where $\pi=\{Y_1,\dots, Y_n\}$. 
By defining   \eqref{simult}  as \vspace{-0.1cm}$$\gamma_{X,Y_1,\dots Y_n}(f,g_1,\dots,g_n)=\mu(\pi,f,g_1,\dots,g_n),\vspace{-0.1cm} $$ and  ${{\it id}_x}$ as $\eta_{\{x\}}(\{x\})$, the operadic axioms are easily verified by the naturality of $\mu$ and  laws of the monoid. This gives a crisp alternative to the somewhat cumbersome componential definition:\vspace{-0.1cm}
\begin{center}
{\em A symmetric operad is a monoid in the monoidal category $({\bf Spec},\circ,E_1)$.}\vspace{-0.1cm}
\end{center}

\indent The steps to derive the monoidal definition from above and, more generally, a   monoid-{\em like} definition of an arbitrary operad-like structure, starting from its componential characterisation, can be summarised   as follows. One first has to exhibit a product $\diamond$ on ${\bf Spec}$ that captures the type of operadic composition that is to be formalised (in the same way as the $(S\circ S)$-structure $(\pi,f,(g_i)_{1\leq i\leq n})$ corresponds to the configuration $(f,g_1,\dots,g_n)$ of operadic operations). One then has to examine the properties of this product, primarily by  comparing   species  $(S\diamond T)\diamond U$ and $S\diamond (T\diamond U)$, in order to exhibit an isomorphism whose commutation with the multiplication $\mu:S\diamond S\rightarrow S$  expresses axioms of the operad-like structure in question. Analogously, an appropriate isomorphism of species is needed for each of the remaining axioms of such a structure  (for example, the isomorphims $\lambda_S^{\circ}$ and $\rho_S^{\circ}$ account fot the unit axioms of an operad), except for the equivariance axiom, which holds by the naturality of   $\mu$. The operad-like structure is then introduced as  an object $S$ of ${\bf Spec}$, together with the multiplication $\mu$ (and possibly other natural transformations, like the unit $\eta$ in the previous definition) that commutes in the appropriate way with established isomorphisms.
\subsection{Fiore-Markl definition} In order to recover Fiore's
algebraic definition  \cite[page 8]{fiore}  by following the steps described above, we first recall
the componential definition based on partial composition products, cf. Markl  \cite[Proposition 42]{opsprops}.

 In the definition below, for a species ${\EuScript O}:{\bf Bij}^{op}\rightarrow {\bf Set}$, a bijection $\sigma:Y\rightarrow X$ and an element $f\in {\EuScript O}(X)$, we write $f^{\sigma}$ for ${\EuScript O}(\sigma)(f)$\footnote{This convention is unambiguous in the framework with a fixed species and can be misleading otherwise, as it does not carry the information about the species it refers to. Since it shortens significantly the notation, we will nevertheless   use it in the general setting whenever the relevant species  is clear from the context.}.
\begin{definition}[\textsc{componential}]\label{operadpartial}
A symmetric   operad  is a species ${\EuScript O}:{\bf Bij}^{op}\rightarrow {\bf Set}$, together with a distinguished element ${\it id}_x\in{\EuScript O}(\{x\})$ that exists for each singleton $\{x\}$, and a partial composition operation $$\circ_{x}:{\EuScript O}(X)\times{\EuScript O}(Y)\rightarrow {\EuScript O}(X\backslash\{x\}\cup Y),$$ defined for arbitrary non-empty finite set $X$, an arbitrary set $Y$ and   $x\in X$, such that $X\backslash\{x\}\cap Y=\emptyset$. These data satisfy the axioms given below. \\[0.15cm]
{\em Associativity.} For $f\in {\EuScript O}(X)$, $g\in {\EuScript O}(Y)$ and $h\in {\EuScript O}(Z)$,  the following two equalities hold:\\[0.15cm]
\indent {\em\texttt{[A1]}} $(f \circ_{x} g) \circ_y h =  (f \circ_{y} h) \circ_x g$, where $x,y\in X$,  and\\[0.15cm]
\indent {\em\texttt{[A2]}} $(f \circ_{x} g) \circ_y h = f \circ_{x} (g \circ_y h)$, where $x\in X$ and $y \in Y$.\\[0.15cm]
{\em Equivariance.} For bijections $\sigma_1:X'\rightarrow X$ and $\sigma_2:Y'\rightarrow Y$, and $f\in{\EuScript C}(X)$ and $g\in {\EuScript C}(Y)$, the following equality holds:\\[0.15cm] 
\indent {\em\texttt{[EQ]}} $f^{\sigma_1} \circ_{\sigma_1^{-1}(x)}\,\, g^{\sigma_2}=(f  \circ_x   g)^{\sigma}$,
where    $\sigma=\sigma_1|^{X\backslash\{x\}}\cup \sigma_2$. \\[0.15cm]
{\em Unitality.} For $f\in\EuScript{C}(X)$ and $x\in X$, the following two equalities hold:\\[0.15cm]
\indent {\em\texttt{[U1]}} ${\it id}_{y} \circ_y f =f$, and\\[0.15cm]
\indent {\em\texttt{[U2]}} $f \circ_x  {\it id}_{x}=f$.\\[0.15cm]
Moreover, the unit elements are preserved under the action of ${\EuScript C}(\sigma)$, i.e.\\[0.15cm] 
\indent {\em\texttt{[UP]}} ${id_{x}}^{\sigma}=id_{u}$, for any two singletons $\{x\}$ and $\{u\}$, and a bijection $\sigma:\{u\}\rightarrow\{x\}$.
\end{definition}
This definition is referred to as partial, since the morphisms $\circ_x$ are to be thought of as insertions of one operation into (one input of) another operation.
\begin{rem} By the axioms  { \texttt{[EQ]}} and {\texttt{[UP]}}, it can be easily shown that, for $f\in{\EuScript O}(X)$ and a renaming 
 $\sigma:X\backslash \{x\}\cup\{y\}\rightarrow X$  of  $x$ to $y$, we have $f\circ_{x} {\it id}_{y}=f^{\sigma}.$
\end{rem}
Observe that the data out of which the composition $f\circ_x g$ is obtained consists of the ordered pair $(f,g)$, together with a chosen input $x$ of $f$. This indicates that the product of species that is supposed to capture  partial composition must involve the   product  $\cdot:{\bf Spec}\times{\bf Spec}\rightarrow {\bf Spec}$ introduced in Definition \ref{sumprod}, whereby the structures arising from the left component of the tensor product should have a distinguished element among the elements of the underlying set.  Hence, a priori, there are two possible candidates for the new product: $S^{\bullet}\cdot S$ and $(\partial S)\cdot S$. However, the first one does not work: for $(f,g)\in (S^{\bullet}\cdot S)(X)$, the multiplication $(S^{\bullet}\cdot S)(X)\rightarrow S(X)$ produces an element of $S(X)$, while the composition of $f$ and $g$ along some input $x$ of $f$ is an element of $S(X\backslash\{x\})$.  On the other hand, the elements  of the   set $(\partial S\cdot S)(X)$  are pairs $(f,g)$ such that $f\in S(X_1\cup\{\ast_{X_1}\}) $ and $g\in S(X_2)$, where $(X_1,X_2)$ is a decomposition of the set $X$. From the operadic perspective, the composition of $f\circ_{\ast_{X_1}}g$  belongs to $S(X)$, which agrees with the form of the multiplication $\nu_X:(\partial S\cdot S)(X)\rightarrow S(X)$.  Therefore, as the tentative product of species we take the {\em pre-Lie product}   $S\star T$, defined as $$S\star T=\partial S\cdot T. $$  The next step is to  compare the species $(S\star T)\star U$ and $S\star(T\star U)$. Chasing the associativity fails in this case. However, there is a canonical natural
{\em pre-Lie isomorphism} $$\beta_{S,T,U}:(S\star T)\star U+ S\star(U\star T)\rightarrow S\star(T\star U)+(S\star U)\star T $$ determined by the isomorphisms 
{\small $$ \begin{array}{ll}
\beta_1:(\partial\partial S\cdot T)\cdot U\rightarrow (\partial\partial S\cdot U)\cdot T \enspace\enspace \enspace& \beta_1={\alpha}^{-1}\circ({\texttt{ex}}\cdot {\texttt{c}})\circ\alpha,\\[0.1cm]
\beta_2:(\partial S\cdot \partial T)\cdot U\rightarrow \partial S\cdot(\partial T\cdot U)\enspace\enspace \enspace & \beta_2=\alpha,\\[0.1cm]
\beta_3:\partial S\cdot(\partial U\cdot T)\rightarrow (\partial S\cdot \partial U)\cdot T\enspace\enspace \enspace & \beta_3={\alpha}^{-1}, 
\end{array}$$} \\[-0.3cm]
where $\alpha$, ${\texttt{c}}$ and ${\texttt{ex}}$ stand  for appropriate instances of  isomorphisms given in Table 2. The pre-Lie isomorphism is the ``smallest" isomorphism that captures both associativity axioms for operads ($\beta_1$ accounts for \texttt{[A1]} and $\beta_2$ for \texttt{[A2]}).\\[0.1cm]
\indent For  operadic units we shall need the isomorphisms exhibited in the following lemma.
\begin{lem}
For an arbitrary species $S$,  $E_1\star S\simeq S$ and $S\star E_1\simeq S$.
\end{lem}
\begin{proof} We get the  isomorphism $\lambda_S^{\star}:E_1\star S\rightarrow S$ via
$E_1\star S=\partial E_1\cdot S\simeq E_0\cdot S\simeq S,$ i.e. as $\lambda_S^{\star}=\lambda_S^{\cdot}\circ({\epsilon_1}\cdot {\it id}_S)$. The isomorphism $\rho_S^{\star}:S\star E_1\rightarrow S$ is defined analogously.
\end{proof}
By the microcosm principle, these data induce  the following definition.
\begin{definition}[\textsc{algebraic}]\label{fiore}
An  operad  is a triple $(S,\nu,\eta_1)$ of a species $S$, a morphism  $\nu:S\star S\rightarrow S$, called the {\em multiplication}, and a morphism $\eta_1:E_1\rightarrow S$, called the {\em unit}, such that\\[0.15cm]
\indent {\em\texttt{[OA1]}} $\nu_2\circ\beta=\nu_1$, where $\nu_1$ and $\nu_2$ are induced by $\nu$ as follows:\\[0.15cm]
\indent\phantom{aaaaka}- $\nu_1:(S\star S)\star S+S\star(S\star S)\rightarrow S$ is determined by \vspace{-0.15cm}
$${\scriptsize \begin{array}{llllll}
&&&&&\nu_{11}:(\partial\partial S\cdot S)\cdot S\xrightarrow{\enspace   i_l\cdot{\it id} \enspace}(\partial\partial S\cdot S+\partial S\cdot\partial S)\cdot \partial S\xrightarrow{\enspace\varphi^{-1}\cdot{\it id}\enspace}
\partial(\partial S\cdot S)\cdot S\xrightarrow{\enspace\partial\nu\cdot{\it id}\enspace}\partial S\cdot S\xrightarrow{\enspace\nu\enspace}S,\\
&&&&&\nu_{12}:(\partial S\cdot \partial S)\cdot S\xrightarrow{\enspace   i_r\cdot{\it id} \enspace}(\partial\partial S\cdot S+\partial S\cdot\partial S)\cdot \partial S\xrightarrow{\enspace\varphi^{-1}\cdot{\it id}\enspace}
\partial(\partial S\cdot S)\cdot S\xrightarrow{\enspace\partial\nu\cdot{\it id}\enspace}\partial S\cdot S\xrightarrow{\enspace\nu\enspace}S,\\
&&&&&\nu_{13}:\partial S\cdot(\partial S\cdot S)\xrightarrow{\enspace {\it id}\cdot\nu \enspace}\partial S\cdot S\xrightarrow{\enspace \nu \enspace}S, \mbox{ and}
\end{array}}$$
\indent\phantom{iiiaaia}- $\nu_2:S\star(S\star S)+(S\star S)\star S\rightarrow S$ is determined by $\nu_{21}=\nu_{11}$, $ \nu_{22}=\nu_{13}$ and\linebreak \indent\phantom{iiiiaaia} $\nu_{23}=\nu_{12}$,  and\\[0.05cm]
\indent {\em\texttt{[OA2]}} $\eta_1$ satisfies  coherence conditions given by the commutation of the  diagram
\begin{center}
\begin{tikzpicture}[scale=1.5]
\node (A)  at (0.2,1) {\footnotesize $E_1\star S$};
\node (B) at (2,1) {\footnotesize $S\star S$};
\node (C) at (3.8,1) {\footnotesize $S\star E_1$};
\node (D) at (2,-0.2) {\footnotesize $S$};
\node (E) at (-0.3,0) {};
\path[->,font=\scriptsize]
(A) edge node[above]{$\eta_1\star {\it id}_S$} (B)
(C) edge node[above]{${\it id}_S\star\eta_1$} (B)
(A) edge node[left,yshift=-0.1cm]{$\lambda_S^{\star}$} (D)
(C) edge node[right,yshift=-0.1cm]{$\rho_S^{\star}$} (D)
(B) edge node[right]{$\nu$} (D);
\end{tikzpicture} 
\end{center}
\end{definition}
Indeed, it can be shown that \texttt{[OA1]} acounts for  \texttt{[A1]} and \texttt{[A2]}\footnote{Actually, the equalities $\nu_{21}\circ\beta=\nu_{11}$ and $\nu_{22}\circ\beta=\nu_{12}$ are enough to prove associativity.}, the naturality of $\nu$ ensures \texttt{[EQ]}, \texttt{[OA2]} proves \texttt{[U1]} and \texttt{[U2]}, are  the naturality of $\eta$ ensures \texttt{[UP]}.
\begin{con}\label{non-unitary}
Henceforth, we shall restrict ourselves to {\em constant-free} operads, i.e. operads for which the underlying species $\EuScript O$ is such that ${\EuScript O}(\emptyset)=\emptyset$. This assumption is necessary for establishing the equivalence  between the algebraic definitions of cyclic operads, as will become clear in Section 4 (Remark 4.4). 
\end{con}
\section{Cyclic operads}  
This section contains the algebraic treatment of cyclic operads. 
\subsection{Entries-only definition of cyclic operads}
Starting from the  entries-only componential definition of cyclic operads,  we follow the steps anticipated by the microcosm principle and present its algebraic counterpart. 
\subsubsection{Componential definition}
The condition ${\EuScript O}(\emptyset)=\emptyset$ that we imposed in Convention \ref{non-unitary} intuitively means that all operations of an operad have at least one input, which, together with the output, makes at least two ``entries" from the point of view of cyclic operads. We revise below \cite[Definition 48]{modular}, by restricting to the class of {\em constant-free} cyclic operads and adding units.
 \begin{definition}[\textsc{entries-only, componential}]\label{entriesonly}
A   constant-free cyclic operad is a species  ${\EuScript C}:{\bf Bij}^{op}\rightarrow {\bf Set}$, such that ${\EuScript C}(\emptyset)={\EuScript C}(\{x\})=\emptyset$ for all singletons $\{x\}$, together with a distinguished element ${\it id}_{x,y}\in {\EuScript C}(\{x,y\})$ for each two-element set $\{x,y\}$, and a partial composition operation 
$${{_{x}\circ_{y}}}:{\EuScript C}(X)\times {\EuScript C}(Y)\rightarrow {\EuScript C}(X\backslash\{x\}\cup Y\backslash\{y\}) ,$$
defined for arbitrary non-empty finite sets $X$ and $Y$ and elements $x\in X$ and $y\in Y$, such that $X\backslash\{x\}\cup Y\backslash\{y\}=\emptyset$.
These data must satisfy the  axioms given below.  \\[0.15cm]
{\em Parallel associativity.} For $f\in {\EuScript C}(X)$, $g\in {\EuScript C}(Y)$ and $h\in {\EuScript C}(Z)$,  the following   equality holds:\\[0.15cm]
\indent {\em\texttt{(A1)}} $(f\, {_{x}\circ_{y}}\,\, g)\,\,{_{u}\circ_z}\, h =  (f {_{u}\circ_{z}}\,\, h)\,\,{_{x}\circ_{y}}\, g$, where $x,u\in X$, $y\in Y,$ $z\in Z$.\\[0.15cm]
{\em Equivariance.} For bijections $\sigma_1:X'\rightarrow X$ and $\sigma_2:Y'\rightarrow Y$, and $f\in{\EuScript C}(X)$ and $g\in {\EuScript C}(Y)$, the following equality holds:\\[0.15cm] 
\indent {\em\texttt{(EQ)}} $f^{\sigma_1}\,\,{_{{ {\sigma_1^{-1}}(x)}}\circ_{\sigma_2^{-1}(y)}}\,\, g^{\sigma_2}=(f {_x\circ_y} \,\, g)^{\sigma}$,
where    $\sigma=\sigma_1|^{X\backslash\{x\}}\cup \sigma_2|^{Y\backslash\{y\}}$. \\[0.15cm]
{\em Unitality.} For $f\in\EuScript{C}(X)$ and $x\in X$, the following  equality holds:\\[0.15cm]
\indent {\em\texttt{(U1)}} ${\it id}_{x,y}\,\,{_y\circ_x}\,\, f=f$.\\[0.15cm]
 Moreover, the unit elements are preserved under the action of ${\EuScript C}(\sigma)$, i.e.\\[0.15cm] 
\indent {\em\texttt{(UP)}} ${id_{x,y}}^{\sigma}=id_{u,v}$, for any bijection $\sigma:\{u,v\}\rightarrow\{x,y\}$.
\end{definition}
\indent The lemma below gives  basic properties of the partial composition operation.
\begin{lem}\label{j} The partial composition operation ${{_{x}\circ_{y}}}$ satisfies the laws listed below.\\[0.15cm]
{\em Commutativity.} For  $f\in{\EuScript C}(X)$, $g\in {\EuScript C}(Y)$, $x\in X$ and $y\in Y$, the following equality holds:\\[0.15cm]
\indent {\em\texttt{(CO)}} $f\, {_x\circ_y} \,\, g=g\, {_y\circ_x} \,\, f$.\\[0.15cm]
{\em Sequential associativity.} For $f\in {\EuScript C}(X)$, $g\in {\EuScript C}(Y)$ and $h\in {\EuScript C}(Z)$,  the following   equality holds:\\[0.15cm]
\indent {\em\texttt{(A2)}} $(f\, {_{x}\circ_{y}}\,\, g)\,\,{_{u}\circ_z}\, h =  f {_{x}\circ_{z}}\,\, (g\,\,{_{u}\circ_{z}}\, h)$, where $x\in X$, $y,u\in Y,$ $z\in Z$.\\[0.15cm]
{\em Unitality.}  For $f\in\EuScript{C}(X)$ and $x\in X$, the following  equality holds: \\[0.15cm]
\indent {\em\texttt{(U2)}} $f \,\,{_x\circ_y}\,\, {\it id}_{y,x}=f$.\\[0.15cm]
{\em Equivariance.}
Let $f\in{\EuScript C}(X), g\in{\EuScript C}(Y)$, $x\in X$, $y\in Y$, and let $\sigma :Z\rightarrow X\backslash\{x\}\cup Y\backslash\{y\}$ be an arbitrary bijection. If $\tau_1:X'\rightarrow X,\tau_2:Y'\rightarrow Y$ and $\tau:Z\rightarrow X'\backslash\{\tau_1^{-1}(x)\}\cup Y'\backslash\{\tau_2^{-1}(y)\}$ are bijections such that $\sigma=(\tau_1|^{X\backslash\{x\}}\cup \tau_2|^{Y\backslash\{y\}})\circ \tau$,
then \\[0.15cm]
\indent {\em\texttt{(EQ)'}}   $(f {_x\circ_y} \,\, g)^{\sigma}=(f^{\tau_1} {_x\circ_y} \,\, g^{\tau_2})^{\tau}$.
\end{lem}
\begin{proof} The commutativity law \texttt{(CO)} holds thanks to \texttt{(U1)} and \texttt{(A1)}, as follows: \vspace{-0.1cm}$$f\, {_{x}\circ_{y}}\,\, g=({\it id}_{x,y}\,\,{_y\circ_x}\,\, f)\, {_{x}\circ_{y}}\,\, g=({\it id}_{x,y}\,\,{_x\circ_y}\,\, g)\, {_{y}\circ_{x}}\,\, f=g\, {_{y}\circ_{x}}\,\, f\,.\vspace{-0.1cm}$$ We then use \texttt{(CO)} together with   \texttt{(A1)} to derive  the sequential associativity law  \texttt{(A2)}\footnote{Since parallel and sequential associativity are both essential for the notion of cyclic operad, for cyclic operads without units the commutativity law has to be specified as an axiom (in order to make \texttt{(A2)} derivable).}: \vspace{-0.1cm}$$(f\, {_{x}\circ_{y}}\,\, g)\,\,{_{u}\circ_z}\, h =(g\, {_{y}\circ_{x}}\,\, f)\,\,{_{u}\circ_z}\, h= (g\, {_{u}\circ_{z}}\,\, h)\,\,{_{y}\circ_x}\, f=f \,\,{_{x}\circ_y}\,  (g\, {_{u}\circ_{z}}\,\, h)\,.\vspace{-0.1cm}$$
The unit law \texttt{(U2)} follows from \texttt{(U1)} by applying \texttt{(CO)} on its left side, and the variation \texttt{(EQ')} of the equivariance axiom follows easily by \texttt{(EQ)}.
\end{proof}
Related to the  implication \texttt{(A1)}$+$\texttt{(CO)}$\Rightarrow$\texttt{(A2)}  proved above, it is also true (and easily checked) that \texttt{(A2)}$+$\texttt{(CO)}$\Rightarrow$\texttt{(A1)}. Therefore, \texttt{(A2)}, \texttt{(CO)}, \texttt{(EQ)}, \texttt{(U1)}, and \texttt{(UP)}  provide an equivalent (but not minimal) axiomatisation for ${_{x}\circ_{y}}$.
\subsubsection{Algebraic definition}
Applying the microcosm principle relative to Definition \ref{entriesonly}  begins with the observation that the data out of which $f {_x\circ_y} \,\, g$ is obtained consists of the pair $(f,g)$ and chosen inputs $x$ and $y$ of $f$ and $g$, respectively. The discussion we had on page 10 for operads makes it easy to guess  what combination of the product and derivative of species is the ``right'' one in this case. 
\begin{definition}
Let $S$ and $T$ be species. The   triangle product  (or, shorter, the $\blacktriangle$-product) of $S$ and $T$ is the species $S\blacktriangle T$ defined as  $$S\blacktriangle T=\partial S\cdot\partial T. $$
Therefore, for a finite set $X$,
 $$(S\blacktriangle T)(X)=  \sum_{(X_1,X_2)}\{(f,g)\,|\,f\in S(X_1\cup\{\ast_{X_1}\}),g\in T(X_2\cup\{\ast_{X_2}\})\},  $$
and, for $(f,g)\in (S\blacktriangle T)(X)$ and a bijection $\sigma:Y\rightarrow X$, 
 $$(S\blacktriangle T)(\sigma)(f,g)=(S(\sigma_1^{+})(f),T(\sigma_2^{+})(g)), $$
where  $\sigma_1=\sigma|^{X_1}$, $\sigma_2=\sigma|^{X_2}$, and $\sigma_i^{+}$ are the $\partial$-extensions of $\sigma_i$,  $i=1,2$.

\end{definition}
\begin{rem} The  isomorphism  {\em\texttt c}$_{\partial S,\partial T}:\partial S\cdot\partial T\rightarrow \partial T\cdot\partial S$ witnesses that the $\blacktriangle$-product is commutative.
\end{rem}
\indent  The next step is  to exhibit an isomorphism that equates various ways to derive a $\blacktriangle$-product of three species. Intuitively, in the language of species, the associativity axiom \texttt{(A1)}  can be stated   as the existence of an isomorphism of the form  $(\partial\partial S\cdot\partial T)\cdot \partial U\rightarrow (\partial\partial S\cdot\partial U)\cdot \partial T.$
It turns out that the ``minimal" isomorphism  that compares $(S\blacktriangle T)\blacktriangle U$ and $S\blacktriangle (T\blacktriangle U)$ and that includes the above isomorphism is  $$\gamma_{S,T,U}:(S\blacktriangle T)\blacktriangle U+T\blacktriangle(S\blacktriangle U)+(T\blacktriangle U)\blacktriangle S\rightarrow S\blacktriangle (T\blacktriangle U)+(S\blacktriangle U)\blacktriangle T+U\blacktriangle(S\blacktriangle T), $$ whose explict description is as follows. By unfolding the definition of $\blacktriangle$, we see that $\gamma_{S,T,U}$  connects a sum of $6$ species on the left with a sum of $6$ species on 
the right.
Here is the list of the $6$   constituents of $\gamma$, together with their explicit definitions:
 $${\small\begin{array}{ll}
\gamma_1:(\partial\partial S\cdot\partial T)\cdot \partial U\rightarrow (\partial\partial S\cdot\partial U)\cdot \partial T\enspace \enspace  \enspace  & \gamma_1={\alpha}^{-1}\circ({\tt ex}\cdot {\tt c})\circ\alpha,\\[0.1cm]
\gamma_2:(\partial S\cdot\partial\partial T)\cdot\partial U\rightarrow\partial S\cdot(\partial\partial T\cdot\partial U)\enspace \enspace  \enspace & \gamma_2={\tt c}\circ\gamma_1\circ({\tt c}\cdot{\it id}),\\[0.1cm]
\gamma_3:\partial T\cdot(\partial\partial S\cdot\partial U)\rightarrow\partial U\cdot(\partial\partial S\cdot\partial T)\enspace  \enspace  \enspace & \gamma_3={\tt c}\circ\gamma_1\circ {\tt c},\\[0.1cm]
\gamma_4:\partial T\cdot(\partial S\cdot\partial\partial U)\rightarrow\partial S\cdot(\partial T\cdot\partial\partial U)\enspace \enspace  \enspace  & \gamma_4=({\it id}\cdot {\tt c})\circ{\tt c}\circ\gamma_1\circ({\tt c}\cdot{\it id})\circ  {\tt c},\\[0.1cm]
\gamma_5:(\partial\partial T\cdot\partial U)\cdot\partial S\rightarrow \partial U\cdot(\partial S\cdot\partial\partial T)\enspace \enspace   \enspace & \gamma_5=({\it id}\cdot {\tt c})\circ {\tt c}\circ\gamma_1, \mbox{ and}\\[0.1cm]
\gamma_6:(\partial T\cdot\partial\partial U)\cdot\partial S\rightarrow(\partial S\cdot\partial\partial U)\cdot\partial T\enspace \enspace  \enspace & \gamma_6={\tt c}\circ({\it id}\cdot {\tt c})\circ{\tt c}\circ\gamma_1\circ({\tt c}\cdot{\it id}).\end{array}}$$
Notice that, having  fixed $\gamma_1$, the pairing given by $\gamma_3$ is also predetermined, but there are other ways to pair the remaining $4$ summands from the left with the $4$ summands from the right.  We made this choice in order for all $\gamma_i$ to represent ``parallel associativity modulo commutativity", but a different pairing could have been chosen as well. \\[0.1cm]
\indent What remains is to exhibit the structure on species that will account for operadic units. The following lemma is essential.

\begin{lem}\label{2unit}
For an arbitrary species $S$,   $E_2\blacktriangle S\simeq S^{\bullet}$ and $S\blacktriangle E_2\simeq S^{\bullet}$.
\end{lem}
\begin{proof} 
By the definition of the product $\blacktriangle$ and of the species $E_2$, we have
\begin{equation}\label{e2}(E_2\blacktriangle S)(X)=\sum_{x\in X}\{(\{x,\ast_{\{x\}}\},f)\,|\, , f\in S(X\backslash\{x\}\cup\{\ast_{X\backslash\{x\}}\})\}.\end{equation}

\indent We define $\lambda_{S}^{\blacktriangle}:E_2\blacktriangle S\rightarrow S^{\bullet}$ as \vspace{-0.1cm}$$
{\lambda_S^{\blacktriangle}}_X:(\{x,\ast_{\{x\}}\},f)\mapsto (S(\sigma)(f),x)\vspace{-0.1cm},$$ where $\sigma:X\rightarrow X\backslash\{x\}\cup\{\ast_{X\backslash\{x\}}\}$ renames  $\ast_{X\backslash\{x\}}$ to $x$.  For $X=\emptyset$, ${\lambda_S^{\blacktriangle}}_X$ is the empty function.\\[0.1cm]
\indent The isomorphism $\kappa_S^{\blacktriangle}:S\blacktriangle E_2\rightarrow S^{\bullet}$ is defined as $\kappa_S^{\blacktriangle}=\lambda_S^{\blacktriangle}\circ {\tt c}$.
\end{proof}
\begin{rem}\label{partial_pointed}
Since $\partial E_2\simeq E_{1}$, we also have that $E_1\cdot \partial S\simeq S^{\bullet}$.  The isomorphism $\delta_{S}:S^{\bullet}\rightarrow E_1\cdot \partial S$ is defined as $(\epsilon_2^{-1}\cdot{\it id}_{\partial S})^{-1}\circ (\lambda_S^{\blacktriangle})^{-1}$, i.e. for $f\in S(X)$ and $x\in X$, we have \vspace{-0.1cm}$${\delta_S}_{X}(f,x)=(\{x\},S(\sigma)(f)),\vspace{-0.1cm}$$ where $\sigma: X\backslash\{x\}\cup\{\ast_{X\backslash\{x\}}\}\rightarrow X$ renames $x$ to $\ast_{X\backslash\{x\}}$.  
\end{rem}
These data are  assembled by the microcosm principle as follows.
\begin{definition}[\textsc{entries-only, algebraic}]\label{copeo}
A constant-free cyclic operad  is a triple $(S,\rho,\eta_2)$ of a species $S$, such that $S(\emptyset)=S(\{x\})=\emptyset$ for all singletons $\{x\}$, a morphism $\rho: S\blacktriangle S\rightarrow S$, called the {\em multiplication}, and a morphism $\eta_{2}:E_2\rightarrow S$, called the {\em unit}, such that\\[0.15cm]
\indent {\em\texttt{(CA1)}} $\rho_2\circ \gamma=\rho_1$, where $\rho_1$ and $\rho_2$ are induced from $\rho$ as follows:\\[0.2cm]
\indent\phantom{aaaaka}- $\rho_1:(S\blacktriangle S)\blacktriangle S+S\blacktriangle(S\blacktriangle S)+(S\blacktriangle S)\blacktriangle S\rightarrow S$ is determined by \vspace{-0.1cm}
$${\scriptsize \begin{array}{llllll} &&&&&\rho_{11}:(\partial\partial S\cdot\partial S)\cdot \partial S\xrightarrow{i\cdot{\it id}}(\partial\partial S\cdot\partial S+\partial S\cdot\partial\partial S)\cdot\partial S\xrightarrow{\varphi^{-1}\cdot{\it id}}\partial(\partial S\cdot\partial S)\cdot\partial S\xrightarrow{\partial\rho\cdot{\it id}}\partial S\cdot\partial S\xrightarrow{\rho}S\\ &&&&&\rho_{12}:(\partial S\cdot\partial\partial S)\cdot\partial S\xrightarrow{i\cdot{\it id}}(\partial\partial S\cdot\partial S+\partial S\cdot\partial\partial S)\cdot\partial S\xrightarrow{\varphi^{-1}\cdot{\it id}}\partial(\partial S\cdot\partial S)\cdot\partial S\xrightarrow{\partial\rho\cdot{\it id}}\partial S\cdot\partial S\xrightarrow{\rho}S\\ &&&&&\rho_{13}:\partial S\cdot(\partial\partial S\cdot\partial S)\xrightarrow{{\it id}\cdot i}\partial S\cdot (\partial\partial S\cdot\partial S+\partial S\cdot\partial\partial S)\xrightarrow{{\it id}\cdot \varphi^{-1}}\partial S\cdot\partial(\partial S\cdot\partial S)\xrightarrow{{\it id}\cdot\partial\rho}\partial S\cdot\partial S\xrightarrow{\rho}S\\ &&&&&\rho_{14}:\partial S\cdot(\partial S\cdot\partial\partial S)\xrightarrow{{\it id}\cdot i}\partial S\cdot (\partial\partial S\cdot\partial S+\partial S\cdot\partial\partial S)\xrightarrow{{\it id}\cdot \varphi^{-1}}\partial S\cdot\partial(\partial S\cdot\partial S)\xrightarrow{{\it id}\cdot\partial\rho}\partial S\cdot\partial S\xrightarrow{\rho}S\\[0.06cm]
&&&&&\rho_{15}=\rho_{11}\mbox{ \enspace\enspace and \enspace\enspace}\rho_{16}=\rho_{12}\end{array}}$$
\indent\phantom{aaaaka}- $\rho_2:S\blacktriangle (S\blacktriangle S)+(S\blacktriangle S)\blacktriangle S+S\blacktriangle(S\blacktriangle S)\rightarrow S$ is determined by $\rho_{21}=\rho_{11}$,\\
\indent\phantom{aafaaka} $\rho_{22}=\rho_{23}=\rho_{13}$, $\rho_{24}=\rho_{25}=\rho_{14}$ and $\rho_{26}=\rho_{12}$, and\\[0.15cm]
\indent {\em\texttt{(CA2)}} $\eta_2$ satisfies the coherence condition given by the commutation of the  diagram 
\begin{center}
\begin{tikzpicture}[scale=1.5]
\node (A)  at (0,1) {\footnotesize $E_2\blacktriangle S$};
\node (B) at (1.8,1) {\footnotesize $S\blacktriangle S$};
\node (E) at (0,0) {\footnotesize $S^{\bullet}$};
\node (D) at (1.8,0) {\footnotesize $S$};
\path[->,font=\scriptsize]
(A) edge node[above]{$\eta_2\blacktriangle{\it id}_S$} (B)
(A) edge node[left]{$\lambda_S^{\blacktriangle}$} (E)
(B) edge node[right]{$\rho$} (D)
(E) edge node[below]{${\pi_1}_{S}$} (D);
\end{tikzpicture}
\vspace{-0.2cm}
\end{center}
where $\lambda_S^{\blacktriangle}$ and $\kappa_S^{\blacktriangle}$ are the isomorphisms from Lemma \ref{2unit}, and ${{\pi_1}_S}_X$ is the first projection.
\end{definition}
\begin{con}\label{non-unitary-cyclic}
In the remaining of the paper we shall work only with constant-free cyclic operads. To make things shorter, we shall refer to them simply as cyclic operads. 
\end{con}
 In the following lemma we prove the equality that represents the algebraic analogue of the commutativity law \texttt{(CO)}, which would follow anyhow after we prove the equivalence between Definition 3.2 and Definition 3.8. Nevertheless, we do this directly since it shortens significantly the proof of that equivalence.
\begin{lem}\label{rhoc}
For an arbitrary algebraic entries-only cyclic operad $(S,\rho,\eta_2)$, $\rho\circ{\tt c}=\rho$.
\end{lem}
\begin{proof}
In Diagram 1 below,  $D_l$ and $D_r$ commute  by {\texttt{(CA2)}},  $D_t$ and $D_m$ commute by the naturality of $\gamma_1$, and $D_b$ commutes as it represents the equality $\rho_{21}\circ\gamma_1=\rho_{11}$.\\
\indent For a finite set $X$, let $(f,g)\in(\partial S\cdot\partial S)(X)$, where $f\in\partial S(X_1)$, $y\in\partial S(X_2)$ and $(X_1,X_2)$ is an arbitrary decomposition of $X$. Starting with $(f,g)$, we chase Diagram 1 from the top left $\partial S\cdot\partial S$ to the bottom left $\partial S\cdot\partial S$, by going through the left ``'border'' of the diagram, i.e. by applying the composition  \begin{equation}\label{thecomp}(\partial\pi_1\circ\partial\lambda^{\blacktriangle}\circ\varphi^{-1}\circ{\it i_l}\circ(\partial\epsilon_2^{-1}\cdot{\it id})\circ{{\lambda^{\star}}^{-1}})\cdot{\it id}.\end{equation}
\vspace{-0.85cm}
\begin{center}
\begin{tikzpicture}[scale=1.5]
\node (J)  at (0.5,3) {\scriptsize $(\partial\partial E_2\cdot\partial S)\cdot \partial S$};
\node (K)  at (3.5,3) {\scriptsize $(\partial\partial E_2\cdot\partial S)\cdot \partial S$};
\node (A)  at (0.5,2) {\scriptsize $(\partial\partial S\cdot\partial S)\cdot \partial S$};
\node (D) at (3.5,2) {\scriptsize $(\partial\partial S\cdot\partial S)\cdot \partial S$};
\node (B)  at (-2.5,3) {\scriptsize $\partial(\partial E_2\cdot\partial S)\cdot\partial S$};
\node (B')  at (6.5,3) {\scriptsize $\partial(\partial E_2\cdot\partial S)\cdot\partial S$};
\node (C)  at (-2.5,1) {\scriptsize $\partial(S^{\bullet})\cdot\partial S$};
\node (C')  at (6.5,1) {\scriptsize $\partial(S^{\bullet})\cdot\partial S$};
\node (G)  at (0.5,1) {\scriptsize $\partial S\cdot\partial S$};
\node (G')  at (3.5,1) {\scriptsize $\partial S\cdot\partial S$};
\node (S)  at (2,1) {\scriptsize $S$};
\node (Sa)  at (0.5,4) {\scriptsize $\partial S\cdot\partial S$};
\node (Sb)  at (3.5,4) {\scriptsize $\partial S\cdot\partial S$};
\node (L)  at (-1.3,2) {\small $D_l$};
\node (R)  at (5.3,2) {\small $D_r$};
\node (T)  at (2,2.5) {\small $D_m$};
\node (St)  at (2,3.5) {\small $D_t$};
\node (M)  at (2,1.5) {\small $D_b$};
\path[->,font=\scriptsize]
(J) edge node[above]{$\gamma_1$} (K)
(Sa) edge node[above]{${\tt c}$} (Sb)
(Sa) edge node[left]{\tiny $((\partial\epsilon_2^{-1}\cdot{\it id})\circ{\lambda^{\star}}^{-1})\cdot\!{\it id}$} (J)
(Sb) edge node[right]{\tiny $((\partial\epsilon_2^{-1}\cdot{\it id})\circ{\lambda^{\star}}^{-1})\cdot\!{\it id}$} (K)
(J) edge node[above]{\tiny $(\varphi^{-1}\!\circ i_l)\!\cdot\!{\it id}$} (B)
(K) edge node[above]{\tiny $(\varphi^{{\small{-1}}}\!\circ i_l)\!\cdot\!{\it id}$} (B')
(J) edge node[left]{$(\partial\partial\eta_2\cdot{\it id})\cdot{\it id}$} (A)
(B) edge node[left]{$\partial\lambda^{\blacktriangle}\cdot{\it id}$} (C)
(B') edge node[right]{$\partial\lambda^{\blacktriangle}\cdot{\it id}$} (C')
(C) edge node[below]{$\partial\pi_1\cdot{\it id}$} (G)
(C') edge node[below]{$\partial\pi_1\cdot{\it id}$} (G')
(K) edge node[right]{$(\partial\partial\eta_2\cdot{\it id})\cdot{\it id}$} (D)
(A) edge node[above]{$\gamma_1$} (D)
(A) edge node[left]{$(\partial\rho\circ\varphi^{-1}\circ i_l)\cdot{\it id}$} (G)
(G) edge node[below]{$\rho$} (S)
(G') edge node[below]{$\rho$} (S)
(D) edge node[right]{$(\partial\rho\circ\varphi^{-1}\circ i_l)\cdot{\it id}$} (G');
\end{tikzpicture}\\[-0.2cm]
{\small Diagram 1.}
\end{center}
We get the sequence \vspace{-0.1cm}
$$(f,g)\mapsto ((\{\ast_{\emptyset}\},f),g)\mapsto ((\{\ast_{\emptyset},\ast_{\ast_{\emptyset}}\},f),g)\mapsto ((\{\ast_{\emptyset},\ast_{X_1}\},f),g)\mapsto ((f,\ast_{X_1}),g)\mapsto (f,g)\,.\vspace{-0.1cm}$$
Hence, \eqref{thecomp} is the identity on $\partial S\cdot\partial S$. By the equalities behind the commutations of $D_l,D_b,D_m,D_t$ and $D_r$, we now get the   sequence of equalities
{\footnotesize $$\begin{array}{rrrrrrcll}
&&&&&\rho&=&\rho\circ{\it id}& \\[0.1cm]
&&&&&&=& \rho\circ ((\partial\pi_1\circ\partial\lambda^{\blacktriangle}\circ\varphi^{-1}\circ{\it i_l}\circ(\partial\epsilon_2^{-1}\cdot{\it id})\circ{{\lambda^{\star}}^{-1}})\cdot{\it id})& \\[0.1cm]
&&&&&&=&\rho\circ ((\partial\rho\circ\varphi^{-1}\circ{\it i_l}\circ(\partial\partial\eta_2\cdot{\it id})\circ(\partial\epsilon_2^{-1}\cdot{\it id})\circ{{\lambda^{\star}}^{-1}})\cdot{\it id})& \enspace\enspace\enspace  {{ \mbox{ \small{$(D_l)$}}}}\\[0.1cm]
&&&&&&=&\rho\circ ((\partial\rho\circ\varphi^{-1}\circ{\it i_l})\cdot{\it id})\circ{\gamma_1}\circ(((\partial\partial\eta_2\cdot{\it id})\circ(\partial\epsilon_2^{-1}\cdot{\it id})\circ{{\lambda^{\star}}^{-1}})\cdot{\it id})&\enspace\enspace\enspace {{ \mbox{ \small{$(D_b)$}}}}\\[0.1cm]
&&&&&&=&\rho\circ ((\partial\rho\circ\varphi^{-1}\circ{\it i_l}\circ(\partial\partial\eta_2\cdot{\it id}))\cdot{\it id})\circ{\gamma_1}\circ(((\partial\epsilon_2^{-1}\cdot{\it id})\circ{{\lambda^{\star}}^{-1}})\cdot{\it id})& \enspace\enspace\enspace  {{ \mbox{ \small{$(D_m)$}}}}\\[0.1cm]
&&&&&&=&\rho\circ ((\partial\rho\circ\varphi^{-1}\circ{\it i_l}\circ(\partial\partial\eta_2\cdot{\it id})\circ(\partial\epsilon_2^{-1}\cdot{\it id})\circ{{\lambda^{\star}}^{-1}})\cdot{\it id}) \circ{\tt c}& \enspace\enspace\enspace {{ \mbox{ \small{$(D_t)$}}}}\\[0.1cm]
&&&&&&=& \rho\circ ((\partial\pi_1\circ\partial\lambda^{\blacktriangle}\circ\varphi^{-1}\circ{\it i_l}\circ(\partial\epsilon_2^{-1}\cdot{\it id})\circ{{\lambda^{\star}}^{-1}})\cdot{\it id})\circ{\tt c}& \enspace\enspace\enspace  {{ \mbox{ \small{$(D_r)$}}}}\\[0.1cm]
&&&&&&=&\rho\circ {\tt c}\,,
\end{array}\vspace{-0.2cm}$$}

\noindent which proves the claim.
\end{proof}
As a consequence of Lemma \ref{rhoc}, the verification of the axiom \texttt{(CA1)} comes down to  the verification of its instance $\rho_{21}\circ\gamma_1=\rho_{11}$.
\begin{cor}\label{oh}
The equality $\rho_2\circ\gamma=\rho_1$ holds if and only if the equality $\rho_{21}\circ\gamma_1=\rho_{11}$ holds.
\end{cor}
Together with the fact that $\kappa_S^{\blacktriangle}=\lambda_S^{\blacktriangle}\circ {\tt c}$, Lemma \ref{rhoc} is also used to prove  the algebraic analogue of the right unitality law \texttt{(U2)}.
\begin{cor}\label{run}
The morphism $\eta_2:E_2\rightarrow S$ satisfies the coherence condition given by the commutation of the diagram
\vspace{-0.45cm}
\begin{center}
\begin{tikzpicture}[scale=1.5]
\node (A)  at (1.8,1) {\footnotesize $S\blacktriangle E_2$};
\node (B) at (0,1) {\footnotesize $S\blacktriangle S$};
\node (E) at (1.8,0) {\footnotesize $S^{\bullet}$};
\node (D) at (0,0) {\footnotesize $S$};
\path[->,font=\scriptsize]
(A) edge node[above]{${\it id}_S\blacktriangle \eta_2$} (B)
(A) edge node[right]{$\kappa_S^{\blacktriangle}$} (E)
(B) edge node[left]{$\rho$} (D)
(E) edge node[below]{${\pi_1}_{S}$} (D);
\end{tikzpicture}
\vspace{-0.2cm}
\end{center}
\end{cor}
The following theorem ensures that Definition \ref{copeo} does the job. In order to make its statement concise (as well as the statements of Theorem \ref{2} and Theorem \ref{3} later), we adopt the following convention.
\begin{con}\label{deq}
We say that two definitions are {\em equivalent} if, given a  structure as specified by the first definition, one can construct a structure as specified by the second definition, and vice-versa, in such a way that going from one structure to the other one, and back, leads to a structure isomorphic   to the starting one. If the latter transformations  results exactly  in the initial structure (rather than an isomorphic one), we say that  the corresponding definitions are {\em strongly equivalent}.
\end{con} 
\begin{thm}\label{1}
Definition \ref{entriesonly} (entries-only, componential) and  Definition \ref{copeo} (entries-only, algebraic)   are strongly equivalent definitions of cyclic operads.
\end{thm}
\begin{proof} 
We define   transformations in both directions  and  show that going from one structure to the other one, and back, leads to the same structure.  \\[0.15cm]
\textsc{[componential $\Rightarrow$ algebraic]} Let ${\EuScript C}:{\bf Bij}^{\it op}\rightarrow {\bf Set}$ be an entries-only cyclic operad defined in components.  For a finite set $X$, a  decomposition $(X_1,X_2)$  of $X$, $f\in \partial {\EuScript C}(X_1)$ and $g\in\partial {\EuScript C}(X_2)$, $\rho_X:(\partial {\EuScript C}\cdot \partial {\EuScript C})(X)\rightarrow {\EuScript C}(X)$ is defined as  $$\rho_X(f,g)=f\, {_{\ast_{X_1}}\!\!\circ_{\ast_{X_2}}} \,\, g \,.$$ For a two-element set, say $\{x,y\}$, the morphism  $\eta:E_2\rightarrow {\EuScript C}$ is defined as   $\eta_{\{x,y\}}:\{x,y\}\mapsto {\it id}_{x,y}$. Otherwise, $\eta_X$ is the empty function. We now verify the axioms.\\[0.1cm]
\indent \texttt{(CA1)} We prove the equality $\rho_{21}\circ\gamma_1\!=\!\rho_{11}$ by chasing Diagram 2, obtained by unfolding the definitions of the three morphisms involved. The axiom \texttt{(CA1)} will then follow by Corollary 3.11.
\begin{center}
\begin{tikzpicture}[scale=1.5]
\node (A)  at (0,2) {\footnotesize $(\partial\partial S\cdot\partial S)\cdot \partial S$};
\node (B) at (2.5,2) {\footnotesize $\partial\partial S\cdot(\partial S\cdot \partial S)$};
\node (C) at (5,2) {\footnotesize $\partial\partial S\cdot(\partial S\cdot \partial S)$};
\node (D) at (7.5,2) {\footnotesize $(\partial\partial S\cdot\partial S)\cdot \partial S$};
\node (F)  at (0,1) {\footnotesize $\partial(\partial S\cdot\partial S)\cdot\partial S$};
\node (F')  at (7.5,1) {\footnotesize $\partial(\partial S\cdot\partial S)\cdot\partial S$};
\node (G)  at (2.4,1) {\footnotesize $\partial S\cdot\partial S$};
\node (G')  at (5.1,1) {\footnotesize $\partial S\cdot\partial S$};
\node (S)  at (3.75,1) {\footnotesize $S$};
\path[->,font=\scriptsize]
(A) edge node[above]{$\alpha$} (B)
(B) edge node[above]{${\tt ex}\cdot {\tt c}$} (C)
(C) edge node[above]{${\alpha}^{-1}$} (D)
(A) edge node[left]{$(\varphi^{-1}\circ i_l)\cdot{\it id}$} (F)
(F) edge node[below]{$\partial\rho\cdot{\it id}$} (G)
(G) edge node[below]{$\rho$} (S)
(G') edge node[below]{$\rho$} (S)
(D) edge node[right]{$(\varphi^{-1}\circ i_l)\cdot {\it id}$} (F')
(F') edge node[below]{$\partial\rho\cdot{\it id}$} (G');
\end{tikzpicture}\\
\vspace{-0.2cm}
{\small Diagram 2.}
\end{center}

\indent Let $((f,g),h)\in ((\partial\partial S\cdot\partial S)\cdot \partial S)(X)$, and suppose that 
$$f\in \partial\partial S(X'_{1}), \enspace  g\in \partial S(X''_1), \enspace h\in\partial S(X_2), \enspace (f,g)\in (\partial\partial S\cdot\partial S)(X_1), $$ where $(X'_1,X''_1)$ is a  decomposition of $X_1$ and $(X_1,X_2)$ is a  decomposition of $X$.\\[0.1cm]
\indent By chasing the diagram to the down-right, we get the following sequence:  \vspace{-0.1cm}{\small $$\begin{array}{lclllll} ((f,g),h)\!&\!\mapsto \!& \! ((f^{\tau_1^{+}},g),h)\!& \!\mapsto \!&\! (f^{\tau_1^{+}}\, {_{\ast_{X'_1\cup\{\ast_{X_1}\}}}\!\circ_{\ast_{X''_1}}} \,\, g , h)\!&\!\mapsto \!(f^{\tau_1^{+}}\, {_{\ast_{X'_1\cup\{\ast_{X_1}\}}}\circ_{\ast_{X''_1}}} \,\, g)\, {_{\ast_{X_1}}\!\circ_{\ast_{X_2}}} \,\, h\,.\\
\end{array} \vspace{-0.5cm}$$}

\noindent The first step here corresponds to the application of $(\varphi^{-1}\circ i_l)\cdot{\it id}$ on $((f,g),h)$, and, therefore, it involves the renaming $\tau_1: X'_1\cup\{\ast_{X_1}\}\rightarrow X'_1\cup\{\ast_{X'_1}\}$ of $\ast_{X'_1}$ to $\ast_{X_1}$, i.e. the action of $\partial S(\tau_1)=S(\tau_1^{+})$ on $f\in \partial S(X'_1\cup\{\ast_{X'_1}\})=S(X'_1\cup\{\ast_{X'_1}\}\cup\{\ast_{X'_1\cup\{\ast_{X'_1}\}}\})$, where $\tau_1^{+}$ is the $\partial$-extension of $\tau_1$.
 Therefore,  $$f^{\tau_1^{+}}\in\partial S(X'_1\cup\{\ast_{X_1}\})=S(X'_1\cup\{\ast_{X_1}\}\cup\{\ast_{X'_1\cup\{\ast_{X_1}\}}\})\,,$$ and $$(f^{\tau_1^{+}},g)\in\partial(\partial S\cdot\partial S)(X_1)=(\partial S\cdot\partial S)(X_1\cup\{\ast_{X_1}\}). $$    The  action of $\partial\rho$ on $(f^{\tau_1^{+}},g)$  composes   $f^{\tau_1^{+}}$ and $g$ along the corresponding distinguished entries $\ast_{X'_1\cup\{\ast_{X_1}\}}$ and $\ast_{X''_1}$ (while carrying over the distinguished element $\ast_{X_1}$ from the pair to the   composite of its components), and, finally, the action of $\rho$ on $(f^{\tau_1^{+}}\, {_{\ast_{X'_1\cup\{\ast_{X_1}\}}}\!\circ_{\ast_{X''_1}}} \, g, h)$  results  in the partial composition of the two components along  $\ast_{X_1}$ and $\ast_{X_2}$.\\[0.1cm]
\indent The sequence on the right-down side consists of the sequence   
 $$\begin{array}{lclllll}
((f,g),h) &\mapsto & (f,(g,h))&\mapsto & (f^{\varepsilon},(h,g))   &\mapsto & ((f^{\varepsilon},h),g),\\
\end{array} $$
arising from the action of $\gamma_1={\alpha}^{-1}\circ({\tt ex}\cdot {\tt c})\circ {\alpha}$, where \vspace{-0.1cm}$$\varepsilon:X'_1\cup\{\ast_{X'_1},\ast_{X'_1\cup\{\ast_{X'_1}\}}\}\rightarrow X'_1\cup\{\ast_{X'_1},\ast_{X'_1\cup\{\ast_{X'_1}\}}\}\vspace{-0.1cm}$$ exchanges $\ast_{X'_1}$ and $\ast_{X'_1\cup\{\ast_{X'_1}\}}$,   followed by the sequence
\vspace{-0.1cm}{\small$$\begin{array}{lclll}
 ((f^{\varepsilon},h),g) \!&\! \mapsto\! &\! (((f^{\varepsilon})^{\tau_2^{+}},h),g)\! & & \\[0.1cm]
& \!\mapsto\! &\!(((f^{\varepsilon})^{\tau_2^{+}}\, {_{\ast_{X'_1\cup\{\ast_{X'_1\cup X_2}\}}}\!\!\circ_{\ast_{X_2}}} \,\, h) , g) \!&\! \mapsto \!& \!((f^{\varepsilon})^{\tau_2^{+}}\, {_{\ast_{X'_1\cup\{\ast_{X'_1\cup X_2}\}}}\!\!\circ_{\ast_{X_2}}} \,\, h)\, {_{\ast_{X'_1\cup X_2}}\!\!\circ_{\ast_{X''_1}}} \,\, g\,, \end{array}\vspace{-0.2cm}$$}

\noindent corresponding to the action of $\rho_{21}$. Similarly as before, the action of $(\varphi^{-1}\circ i_l)\cdot{\it id}$ on $((f^{\varepsilon},h),g)$ involves the renaming $\tau_2:X'_1\cup \{\ast_{X'_1\cup X_2}\}\rightarrow X'_1\cup\{\ast_{X'_1}\}$ of $\ast_{X'_1}$ to $\ast_{X'_1\cup X_2}$, i.e. the action of $\partial S(\tau_2)=S(\tau_2^{+})$ on $f^{\varepsilon}\in \partial S(X'_1\cup\{\ast_{X'_1}\})=S(X'_1\cup\{\ast_{X'_1}\}\cup \{\ast_{X'_1\cup \{\ast_{X'_1}\}}\})$, where $\tau_2^{+}$ is the $\partial$-extension of $\tau_2$. 
This results in  $$(f^{\varepsilon})^{\tau_2^{+}}\in  \partial S(X'_1\cup\{\ast_{X'_1\cup X_2}\})=S(X'_1\cup\{\ast_{X'_1\cup X_2}\}\cup \{\ast_{X'_1\cup\{\ast_{X'_1\cup X_2}\}}\}),$$ i.e.  $$((f^{\varepsilon})^{\tau_2^{+}},h)\in \partial (\partial S\cdot\partial S)(X'_1\cup X_2)=(\partial S\cdot\partial S)(X'_1\cup X_2\cup \{\ast_{X'_1\cup X_2}\}). $$    The application of $\partial\rho$ on $((f^{\varepsilon})^{\tau_2^{+}},h)$  composes $(f^{\varepsilon})^{\tau_2^{+}}$ and $h$ along $\ast_{X'_1\cup \{\ast_{X'_1\cup X_2}\}}$ and $\ast_{X_2}$ (and carries over the distinguished element $\ast_{X'_1\cup X_2}$ to the composition). Finally, the application of $\rho$ on $((f^{\varepsilon})^{\tau_2^{+}}\, {_{\ast_{X'_1\cup \{\ast_{X'_1\cup X_2}\}}}\circ_{\ast_{X_2}}} \,h , g) $ composes the two components along  $\ast_{X'_1\cup X_2}$ and $\ast_{X''_1}$.\\[0.1cm]
\indent The claim follows by the sequence of equalities 
\vspace{-0.1cm}$$\begin{array}{lclr}
(f^{\tau_1^{+}}\, {_{\ast_{X'_1\cup\{\ast_{X_1}\}}}\circ_{\ast_{X''_1}}} \,\, g)\, {_{\ast_{X_1}}\circ_{\ast_{X_2}}} \,\, h &=& (f^{\tau_1^{+}}\,  {_{\ast_{X_1}}\circ_{\ast_{X_2}}} \,\, h)\, {_{\ast_{X'_1\cup\{\ast_{X_1}\}}}\circ_{\ast_{X''_1}}} \,\, g & \mbox{{\small{\texttt{(A1)}}}} \\
&=& (((f^{\tau_1^{+}})^{\sigma}\,  {_{\ast_{X'_1\cup\{\ast_{X'_1\cup X_2}\}}}\circ_{\ast_{X_2}}} \,\, h)\, {_{\ast_{X'_1\cup X_2}}\circ_{\ast_{X''_1}}} \,\, g & \mbox{{\small{\texttt{(EQ)}}}}\\[-0.1cm]
&=&  ((f^{\varepsilon})^{\tau_2^{+}}\, {_{\ast_{X'_1\cup \{\ast_{X'_1\cup X_2}\}}}\circ_{\ast_{X_2}}} \,\, h)\, {_{\ast_{X'_1\cup X_2}}\circ_{\ast_{X''_1}}} \,\, g,
\end{array}\vspace{-0.1cm}$$
where $\sigma:X'_1\cup \{\ast_{X'_1\cup X_2},\ast_{X'_1\cup \{\ast_{X'_1\cup X_2}\}}\}\rightarrow X'_1\cup \{\ast_{X_1},\ast_{X'_1\cup \{\ast_{X_1}\}}\}$ renames $\ast_{X_1}$ to $\ast_{X'_1\cup \{\ast_{X'_1\cup X_2}\}}$ and $\ast_{X'_1\cup \{\ast_{X_1}\}}$ to   $\ast_{X'_1\cup X_2}$. The last equality in the sequence holds by the equality  $\tau_1^{+}\circ\sigma=\tau_2^{+}\circ\varepsilon$.\\[0.1cm]
\indent \texttt{(CA2)} The commutation of the diagram
\begin{center}
\begin{tikzpicture}[scale=1.5]
\node (A)  at (0,1) {\footnotesize $(E_2\blacktriangle S)(X)$};
\node (B) at (2.5,1) {\footnotesize $(S\blacktriangle S)(X)$};
\node (C) at (2.5,0) {\footnotesize $S(X)$};
\node (D) at (0,-0) {\footnotesize $S^{\bullet}(X)$};
\path[->,font=\scriptsize]
(A) edge node[above]{$(\eta\blacktriangle{\it id}_S)_X$} (B)
(A) edge node[left]{${\lambda^{\blacktriangle}_S}_X$} (D)
(B) edge node[right]{$\rho_X$} (C)
(D) edge node[below]{${{\pi_1}_S}_X$} (C);
\end{tikzpicture}
\end{center}
for $X=\emptyset$ follows since $(E_2\blacktriangle S)(\emptyset)=\emptyset $ and
 since there is a unique empty function with codomain $S(X)$. If $X\neq\emptyset$, then for $(\{x,\ast_{\{x\}}\},f)\in (E_2\blacktriangle S)(X)$ (see \eqref{e2}), by chasing the down-right side of the diagram,  we get \vspace{-0.1cm}$$(\{x,\ast_{\{x\}}\},f)\mapsto (f^{\sigma},x)\mapsto f^{\sigma},\vspace{-0.1cm}$$ where $\sigma$ renames $\ast_{X\backslash\{x\}}$ to $x$. By going to the right-down, we get \vspace{-0.1cm}$$(\{x,\ast_{\{x\}}\},f)\mapsto ({\it id}_{x,\ast_{\{x\}}},f)\mapsto {\it id}_{x,\ast_{\{x\}}}\, {_{\ast_{\{x\}}}\!\circ_{\ast_{X\backslash\{x\}}}} \,\, f \,.\vspace{-0.1cm}$$ The equality $f^{\sigma}={\it id}_{x,\ast_{\{x\}}}\, {_{\ast_{\{x\}}}\!\circ_{\ast_{X\backslash\{x\}}}} \,\, f$ follows easily by {\texttt{(U1)}} and {\texttt{(EQ)}}. \\[0.18cm]
\textsc{[algebraic $\Rightarrow$ componential]}  Suppose that $S:{\bf Bij}^{\it op}\rightarrow {\bf Set}$ is an algebraic cyclic operad,   let $X$ and $Y$ be non-empty finite sets, let $x\in X$ and $y\in Y$ be such that $X\backslash\{x\}\cap Y\backslash\{y\}=\emptyset$, and let $f\in S(X)$ and $g\in S(Y)$. Then $(f,x)\in S^{\bullet}(X)$ and $(g,y)\in S^{\bullet}(Y)$. Therefore, for  $$(\{x\},f^{\sigma_1})=\delta(f,x)\in E_1\cdot\partial S(X)\mbox{ \enspace and \enspace} (\{y\},g^{\sigma_2})=\delta(g,y)\in E_1\cdot\partial S(Y), $$ where $\delta$ is the isomorphism from Remark \ref{partial_pointed}, we have that  $f^{\sigma_1}\in \partial S(X\backslash\{x\})$ and $g^{\sigma_2}\in \partial S(Y\backslash\{y\})$. We define ${{_{x}\circ_{y}}}:{\EuScript C}(X)\times {\EuScript C}(Y)\rightarrow {\EuScript C}(X\backslash\{x\}\cup Y\backslash\{y\})$ as \vspace{-0.1cm}$$f\, {_{x}\circ_{y}} \,\, g =\rho(f^{\sigma_1},g^{\sigma_2}).\vspace{-0.1cm}$$
For a two-element set, say $\{x,y\}$, the distinguished element ${\it id}_{x,y}\in S(\{x,y\})$ is $\eta_{\{x,y\}}(\{x,y\})$. We move on to the verification of the axioms.\\[0.1cm]
\indent {\texttt{(A1)}}  Let $f$ and $g$ be as above, and let $u\in X\backslash\{x\}$, $h\in S(Z)$ and $z\in Z$.
We use the naturality of $\rho$ and the commutation  of Diagram 2 to prove the equality $(f\, {_{x}\circ_{y}}\,\, g)\,\,{_{u}\circ_z}\, h =  (f {_{u}\circ_{z}}\,\, h)\,\,{_{x}\circ_{y}}\, g$. Since it is not evident by which element we should start the diagram chasing in order to reach $(f\, {_{x}\circ_{y}}\,\, g)\,\,{_{u}\circ_z}\, h$, we shall first express this composition via  $\rho$, and then reshape the expression we obtained towards an equal one, ``accepted" by the diagram.\\
\indent Firstly, we have that $$(f\, {_{x}\circ_{y}} \,\, g , u)=(\rho(f^{\sigma_1},g^{\sigma_2}),u)\in S^{\bullet}(X\backslash\{x\}\cup Y\backslash\{y\})  \mbox{ \enspace\enspace and \enspace\enspace}  (h,z)\in S^{\bullet}(Z). $$ By applying the isomorphism $\delta$ from Remark \ref{partial_pointed} on these two elements, we get $$(\{u\},(f\, {_{x}\circ_{y}} \,\, g)^{\kappa_1})=\delta(f\, {_{x}\circ_{y}} \,\, g , u) \mbox{ \enspace\enspace and \enspace\enspace}  (\{z\},h^{\kappa_2})=\delta(h,z),$$ 
where $\kappa_1:X\backslash\{x,u\}\cup Y\backslash\{y\}\cup \{\ast_{X\backslash\{x,u\}\cup Y\backslash\{y\}}\}\rightarrow X\backslash\{x\}\cup Y\backslash\{y\}$  renames $u$ to $\ast_{X\backslash\{x,u\}\cup Y\backslash\{y\}}$ and $\kappa_2:Z\backslash\{z\}\cup \{\ast_{Z\backslash\{z\}}\}\rightarrow Z$ renames $z$ to $\ast_{Z\backslash\{z\}}$. Therefore, 
 $$(f\, {_{x}\circ_{y}} \,\, g)^{\kappa_1}\in \partial S(X\backslash\{x,u\}\cup Y\backslash\{y\}) \mbox{ \enspace\enspace and \enspace\enspace} h^{\kappa_2}\in\partial S(Z\backslash\{z\}), $$ and  the left hand side of the equality  \texttt{(A1)} becomes
 $$(f\, {_{x}\circ_{y}} \,\, g)\, {_{u}\circ_{z}} \,h=\rho(\rho(f^{\sigma_1}, g^{\sigma_2})^{\kappa_1},h^{\kappa_2}).$$
 
\indent Next, notice that the shape of  $\rho(f^{\sigma_1}, g^{\sigma_2})^{\kappa_1}\in S(X\backslash\{x,u\}\cup Y\backslash\{y\}\cup \{\ast_{X\backslash\{x,u\}\cup Y\backslash\{y\}}\}),$ makes the element $\rho(\rho(f^{\sigma_1}, g^{\sigma_2})^{\kappa_1},h^{\kappa_2})$  not explicitely reachable in Diagram 2. However, $\rho(f^{\sigma_1}, g^{\sigma_2})^{\kappa_1}$ is the result of chasing to the down-left the diagram below
\begin{center}
\begin{tikzpicture}[scale=1.5]
\node (A)  at (0,1) {\footnotesize $(S\blacktriangle S)(X\backslash\{x\}\cup Y\backslash\{y\})$};
\node (B) at (5,1) {\footnotesize $(S\blacktriangle S)(X\backslash\{x\}\cup Y'\backslash\{y\})$};
\node (C) at (5,0) {\footnotesize $S(X'\backslash\{x\}\cup Y\backslash\{y\})$};
\node (D) at (0,0) {\footnotesize $S(X'\backslash\{x\}\cup Y\backslash\{y\})$};
\path[->,font=\scriptsize]
(A) edge node[above]{$(S\blacktriangle S)(\kappa_1)$} (B)
(A) edge node[left]{$\rho_{X\backslash\{x\}\cup Y\backslash\{y\}}$} (D)
(B) edge node[right]{$\rho_{X'\backslash\{x\}\cup Y\backslash\{y\}}$} (C)
(D) edge node[below]{$S(\kappa_1)$} (C);
\end{tikzpicture}
\end{center}
\noindent where $X'=X\backslash\{u\}\cup \{\ast_{X\backslash\{x,u\}\cup Y\backslash\{y\}}\}$, starting with $(f^{\sigma_1},g^{\sigma_2})$. This diagram commutes as an instance of the naturality condition for $\rho$. Let us chase it to the right. Firstly, we have that 
 $$
(S\blacktriangle S)(\kappa_1)(f^{\sigma_1},g^{\sigma_2})=((\partial S)(\nu_1)(f^{\sigma_1}),(\partial S)(\nu_2)(g^{\sigma_2}))= (S(\nu_1^{+})(f^{\sigma_1}),S(\nu_2^{+})(g^{\sigma_2})) \,, $$
where $\nu_1^{+}$ and $\nu_2^{+}$ are the $\partial$-extensions of $\nu_1:X\backslash\{x,u\}\!\cup \!\{\ast_{X\backslash\{x,u\}\cup Y\backslash\{y\}}\}\rightarrow X\backslash\{x\}$, which renames $u$ to $\ast_{X\backslash\{x,u\}\cup Y\backslash\{y\}}$, and  the identity $\nu_2:Y\backslash\{y\}\rightarrow Y\backslash\{y\}$, respectively.
Therefore,  the result of chasing the diagram on the right is $\rho((f^{\sigma_1})^{\nu_1^{+}},g^{\sigma_2})$. Consequently, we have that  $$(f\, {_{x}\circ_{y}} \,\, g)\, {_{u}\circ_{z}} \,h=\rho(\rho((f^{\sigma_1})^{\nu_1^{+}},g^{\sigma_2}),h^{\kappa_2}). $$

\indent On the other hand, chasing Diagram 2 in order to reach $(f\, {_{x}\circ_{y}} \,\, g)\, {_{u}\circ_{z}} \,h$ will certainly include considering the element  $(\{u\},(f^{\sigma_1})^{\alpha^{+}})=\delta(f^{\sigma_1},u)$, where $\alpha^{+}$ is the $\partial$-extension of  $\alpha:X\backslash\{x,u\}\cup \{\ast_{X\backslash\{x,u\}}\}\rightarrow X\backslash\{x\}$, which renames $u$ to $\ast_{X\backslash\{x,u\}}$. Therefore,  $$(f^{\sigma_1})^{\alpha^{+}}\in\partial\partial S(X\backslash\{x,u\}), $$ and, consequently,  $$((f^{\sigma_1})^{\alpha^{+}},g^{\sigma_2})\in (\partial S\cdot\partial\partial S)(X\backslash\{x,u\}\cup Y\backslash\{y\}). $$
Furthermore, this chasing will include the element \vspace{-0.1cm}$$\varphi^{-1}((f^{\sigma_1})^{\alpha^{+}},g^{\sigma_2})= (((f^{\sigma_1})^{\alpha^{+}})^{\tau^{+}},g^{\sigma_2})\in \partial(\partial S\cdot\partial S),\vspace{-0.1cm}$$ where $\tau^{+}$ is the $\partial$-extension of $\tau:X\backslash\{x,u\}\cup \{\ast_{X\backslash\{x,u\}\cup Y\backslash\{y\}}\}\rightarrow X\backslash\{x,u\}\cup \{\ast_{X\backslash\{x,u\}}\}$ that renames $\ast_{X\backslash\{x,u\}}$ to $\ast_{X\backslash\{x,u\}\cup Y\backslash\{y\}}$.\\
\indent As a consequence of the equality $\nu_1^{+}=\alpha^{+}\circ\tau^{+}$, we get the equality
 $$(f\, {_{x}\circ_{y}} \,\, g)\, {_{u}\circ_{z}} \,h=\rho(\rho(((f^{\sigma_1})^{\alpha^{+}})^{\tau^{+}},g^{\sigma_2}),h^{\kappa_2}), $$ in which the right hand side is the result of chasing Diagram 2 to the down-left, starting with   $$(((f^{\sigma_1})^{\alpha^{+}},g^{\sigma_2}),h^{\kappa_2})\in ((\partial\partial S\cdot \partial S)\cdot \partial S)(X\backslash\{x,u\}\cup Y\backslash\{y\}\cup Z\backslash\{z\}). $$ 
\indent The remaining of the proof of \texttt{(A1)}  now unfolds easily:
the sequence obtained by chasing Diagram 2 to the right-down, starting with $(((f^{\sigma_1})^{\alpha^{+}},g^{\sigma_2}),h^{\kappa_2})$, consists of {\small $$(((f^{\sigma_1})^{\alpha^{+}},g^{\sigma_2}),h^{\kappa_2})\mapsto ((f^{\sigma_1})^{\alpha^{+}},(g^{\sigma_2},h^{\kappa_2}))\mapsto (((f^{\sigma_1})^{\alpha^{+}})^{\varepsilon},(h^{\kappa_2},g^{\sigma_2}))\mapsto ((((f^{\sigma_1})^{\alpha^{+}})^{\varepsilon},h^{\kappa_2}),g^{\sigma_2}),\vspace{-0.1cm}$$}

\noindent arising from the action of $\gamma_1={\alpha}^{-1}\circ({\tt ex}\cdot {\tt c})\circ {\alpha}$, where \vspace{-0.1cm}$$\varepsilon: X\backslash\{x,u\}\cup \{\ast_{X\backslash\{x,u\}},\ast_{X\backslash\{x,u\}\cup \{\ast_{X\backslash\{x,u\}}\}}\}\rightarrow X\backslash\{y,u\}\cup \{\ast_{X\backslash\{x,u\}},\ast_{X\backslash\{x,u\}\cup \{\ast_{X\backslash\{x,u\}}\}}\}\vspace{-0.1cm}$$ exchanges $\ast_{X\backslash\{x,u\}}$ and $\ast_{X\backslash\{x,u\}\cup \{\ast_{X\backslash\{x,u\}}\}}$,
 followed by 
{\small $$\begin{array}{lclllll}
((((f^{\sigma_1})^{\alpha^{+}})^{\varepsilon},h^{\kappa_2}),g^{\sigma_2})&\mapsto &(((((f^{\sigma_1})^{\alpha^{+}})^{\varepsilon})^{\omega^{+}},h^{\kappa_2}),g^{\sigma_2}) & &&\\[0.1cm]
&\mapsto & (\rho((((f^{\sigma_1})^{\alpha^{+}})^{\varepsilon})^{\omega^{+}},h^{\kappa_2}),g^{\sigma_2})&\mapsto  & \rho(\rho((((f^{\sigma_1})^{\alpha^{+}})^{\varepsilon})^{\omega^{+}},h^{\kappa_2}),g^{\sigma_2})\, ,  \end{array}\vspace{-0.1cm}$$}

\noindent where $\omega^{+}$ is the $\partial$-extension of the renaming
$\omega:X\backslash\{x,u\}\cup \{\ast_{X\backslash\{x,u\}\cup Z\backslash\{z\}}\}\rightarrow X\backslash\{x,u\}\cup \{\ast_{X\backslash\{x,u\}}\}$ of $\ast_{X\backslash\{x,u\}}$ to $\ast_{X\backslash\{x,u\}\cup Z\backslash\{z\}}$.
 Similarly as we did for the left side of {{\texttt{(A1)}}}, it can be shown that $\rho(\rho((((f^{\sigma_1})^{\alpha^{+}})^{\varepsilon})^{\omega^{+}},h^{\kappa_2}),g^{\sigma_2})=f\, {_{x}\circ_{y}} \,( g\, {_{u}\circ_{z}} \,h),$ which completes the proof of {\texttt{(A1)}}.\\[0.1cm]
\indent \texttt{(EQ)} Let $f$ and $g$ be as above and suppose that $\sigma_1,\sigma_2$ and $\sigma$ are as in \texttt{(EQ)}. Then $(f^{\sigma_1},\sigma_1^{-1}(x))\in S^{\bullet}(X')$ and $(g^{\sigma_2},\sigma_2^{-1}(y))\in S^{\bullet}(Y')$, and we have  $$f^{\sigma_1} \, {_{\sigma_1^{-1}(x)}\circ_{\sigma_2^{-1}(y)}} \,\, g^{\sigma_2}=\rho((f^{\sigma_1})^{\tau_1},(g^{\sigma_2})^{\tau_2}),  $$ where $\tau_1$ renames $\sigma_1^{-1}(x)$ to $\ast_{X'\backslash\{\sigma_1^{-1}(x)\}}$ and $\tau_2$ renames $\sigma_2^{-1}(y)$ to $\ast_{Y'\backslash\{\sigma_2^{-1}(y)\}}$. On the other hand, we have that $$(f\, {_{x}\circ_{y}} \,\,g)^{\sigma}=\rho(f^{\kappa_1},g^{\kappa_2})^{\sigma},$$ where $\kappa_1$ renames $x$ to $\ast_{X\backslash\{x\}}$ and $\kappa_2$ renames $y$ to $\ast_{Y\backslash\{y\}}$. The equality $\rho((f^{\sigma_1})^{\tau_1},(g^{\sigma_2})^{\tau_2})=\rho(f^{\kappa_1},g^{\kappa_2})^{\sigma}$ follows easily by chasing the diagram 
\begin{center}
\begin{tikzpicture}[scale=1.5]
\node (E) at (-2,0) {};
\node (A)  at (0,1) {\footnotesize $(S\blacktriangle S)(X\backslash\{x\}\cup Y\backslash\{y\})$};
\node (B) at (5,1) {\footnotesize $(S\blacktriangle S)(X'\backslash\{\nu_{1}^{-1}(x)\}\cup Y'\backslash\{\nu_{2}^{-1}(y)\})$};
\node (C) at (5,-0) {\footnotesize $S(X'\backslash\{\nu_{1}^{-1}(x)\}\cup Y'\backslash\{\nu_{2}^{-1}(y)\})$};
\node (D) at (0,-0) {\footnotesize $S(X\backslash\{x\}\cup Y\backslash\{y\})$};
\path[->,font=\scriptsize]
(A) edge node[above]{$(S\blacktriangle S)(\sigma)$} (B)
(A) edge node[left]{$\rho_{X\backslash\{x\}\cup Y\backslash\{y\}}$} (D)
(B) edge node[right]{$\rho_{X'\backslash\{\nu_{1}^{-1}(x)\}\cup Y'\backslash\{\nu_{2}^{-1}(y)\}}$} (C)
(D) edge node[below]{$S(\sigma)$} (C);
\end{tikzpicture}
\end{center}
which is an instance of the naturality of $\rho$,  starting with $(f^{\kappa_1},g^{\kappa_2})$.
\\[0.1cm]
\indent {\texttt{(U1)}} For $f\in S(X)$ and $x,y\in X$ we have $${\it id}_{x,y}\,\,{_y\circ_x}\,\, f=\rho({\it id}_{x,y}^{\tau_1},f^{\tau_2})=\rho({\it id}_{\ast_{\{z\}},z},f^{\tau_2}),  $$ where $\tau_1$ renames $y$ to $\ast_{\{x\}}$ and $\tau_2$ renames $x$ to $\ast_{X\backslash \{x\}}$. The right hand side of the previous equality is the result of chasing to the right-down the diagram
\begin{center}
\begin{tikzpicture}[scale=1.5]
\node (A)  at (0,1) {\footnotesize $(E_2\blacktriangle S)(X)$};
\node (B) at (2.5,1) {\footnotesize $(S\blacktriangle S)(X)$};
\node (C) at (2.5,-0) {\footnotesize $S(X)$};
\node (D) at (0,-0) {\footnotesize $S^{\bullet}(X)$};
\path[->,font=\scriptsize]
(A) edge node[above]{$(\eta_2\blacktriangle{\it id}_S)_{X}$} (B)
(A) edge node[left]{${\lambda^{\blacktriangle}_S}_{X}$} (D)
(B) edge node[right]{$\rho_{X}$} (C)
(D) edge node[below]{${{\pi_1}_S}_{X}$} (C);
\end{tikzpicture}
\end{center}
which commutes by \texttt{(CA2)}, starting with $(\{x,\ast_{\{x\}}\},f^{\tau_2})$. 
By chasing it to the down-right we get exactly $f$,  which completes the proof of \texttt{(U1)}.\\[0.1cm]
\indent {\texttt{(UP)}} The preservation of units follows directly by the naturality of $\eta_2$.\\[0.15cm]
\textsc{[componential $\Rightarrow$ algebraic $\Rightarrow$ componential]}    That the transition  \[ 
f\, {_{x}\circ_{y}} \, g \xrsquigarrow{\mbox{\tiny{Def. 3.2}}\,\rightarrow\, \mbox{\tiny{Def. 3.8}}}\rho(f^{\sigma_1},g^{\sigma_2})\xrsquigarrow{\mbox{\tiny{Def. $3.8$}}\,\rightarrow\, \mbox{\tiny{Def. $3.2$}}} f^{\sigma_1} \, {_{\ast_{X\backslash\{x\}}}\circ_{\ast_{Y\backslash\{y\}}}} \, g^{\sigma_2},\vspace{-0.1cm}
\]
from the composition morphism $ {_{x}\circ_{y}}$ of an entries-only cyclic operad ${\EuScript C}$ defined in components, to the composition by means of the multiplication $\rho$,  and back, leads to the same composition operation follows by the axiom \texttt{(EQ)} of ${\EuScript C}$.\\[0.1cm]
\textsc{[algebraic $\Rightarrow$ componential $\Rightarrow$ algebraic]} For  the transition
 $$\rho_X(f,g)\xrsquigarrow{\mbox{\tiny{Def. 3.8}}\,\rightarrow\, \mbox{\tiny{Def. 3.2}}} f \, {_{\ast_{X_1}}\circ_{\ast_{X_2}}} \, g\xrsquigarrow{\mbox{\tiny{Def. 3.2}}\,\rightarrow\, \mbox{\tiny{Def. 3.8}}} \rho_X(f^{\tau_1},g^{\tau_2}),$$ we have that $\rho_X(f,g)=\rho_X(f^{\tau_1},g^{\tau_2})$,
 since $\tau_1$ and $\tau_2$ are identities. It is also easily seen that both transitions preserve units, which makes the proof complete. 
\end{proof}
\subsection{Exchangeable-output definition of cyclic operads}
In this part, we first transfer   Markl's skeletal exchangeable-output definition \cite[Proposition 42]{opsprops} to the non-skeletal setting, by introducing a non-skeletal version of the cycle $\tau_n=(0,1,\dots,n)$ that enriches the operad structure to the structure of cyclic operads. We then deliver the algebraic counterpart of the obtained non-skeletal definition.

\subsubsection{Componential definition}
The  symmetric group $\mathbb{S}_n$, whose action (in the skeletal operad structure)  formalizes  the permutations of the inputs of an $n$-ary operation, together with the action of $\tau_n$,   generates all possible permutations of the set $\{0,1,\dots,n\}$. Hence, they constitute the action of ${\mathbb S}_{n+1}$, which  involves the action of exchanging the output of an operation (now denoted with $0$) with one of the inputs. 
Observe that   ${\mathbb S}_{n+1}$ can equivalently be generated by   extending the action ${\mathbb S}_n$ with transpositions of the form $(i\,\, 0)$, for $1\leq i\leq n$. In the non-skeletal setting, where the inputs of an operation should  be labeled with arbitrary letters, rather than with natural numbers, we mimick these transpositions with actions of the form $D^{\EuScript O}_x:{\EuScript O}(X)\rightarrow {\EuScript O}(X)$, where  $x\in X$ denotes the input of an operation chosen to be exchanged with the output. Here is the resulting definition.
\begin{definition}[\textsc{exchangeable-output, componential}]\label{exoutput}
A   cyclic operad  is a  (componential) symmetric operad ${\EuScript O}$, enriched with actions  $$D^{\EuScript O}_{x}:{\EuScript O}(X)\rightarrow {\EuScript O}(X),$$ defined for all  $x\in X$ and subject to the  axioms given below, wherein, for each of the axioms, we assume that $f\in {\EuScript O}(X)$.\\[0.15cm]
{\em Preservation of units.} \\[0.15cm]
\indent {\em\texttt{(DID)}} $D^{\EuScript O}_{x}({\it id}_x)={\it id}_x$.\\[0.15cm]
{\em Inverse.} For $x\in X$, \\[0.15cm]
\indent {\em\texttt{(DIN)}} $D^{\EuScript O}_{x}(D^{\EuScript O}_{x}(f))=f$.\\[0.15cm]
{\em Equivariance.} For $x\in X$ and an arbitrary bijection $\sigma:Y\rightarrow X$, \\[0.15cm]
\indent {\em\texttt{(DEQ)}} $ D_x(f)^{\sigma} = D_{\sigma^{-1}(x)}(f^{\sigma})$.\\[0.15cm]
{\em Exchange.}
For $x,y\in X$ and a bijection $\sigma:X\rightarrow X$ that renames $x$ to $y$ and $y$ to $x$,\\[0.15cm]
\indent {\em\texttt{(DEX)}}  $D^{\EuScript O}_{x}(f)^{\sigma}=D^{\EuScript O}_x(D^{\EuScript O}_y(f)).$\\[0.15cm]
{\em Compatibility with operadic compositions.} For $g\in {\EuScript O}(Y)$, the following equality holds: \\[0.15cm]
\indent {\em\texttt{(DC1)}}  $D^{\EuScript O}_{y}(f\circ_{x}g)=D^{\EuScript O}_{y}(f)\circ_{x}g$, where $y\in X\backslash\{x\}$, and  \\[0.15cm] 
\indent {\em\texttt{(DC2)}}  $D^{\EuScript O}_{y}(f\circ_{x}g)=D^{\EuScript O}_{y}(g)^{\sigma_1}\circ_{v}D^{\EuScript O}_{x}(f)^{\sigma_2}$, where $y\in Y$, $\sigma_1:Y\backslash\{y\}\cup\{v\}\rightarrow Y$  is a\linebreak 
\phantom{\indent \texttt{(DC2)}} bijection that renames $y$ to  $v$ and $\sigma_2:X\backslash\{x\}\cup\{y\}\rightarrow X$ is a bijection that \linebreak \phantom{\indent \texttt{(DC2)}}  renames $x$ to $y$.
\end{definition}
\begin{con}
For $f\in {\EuScript O}(X)$, $x\in X$ and $y\not\in X\backslash\{x\}$, we write $D^{\EuScript O}_{xy}(f)$ for $D^{\EuScript O}_x(f)^{\sigma}$, where $\sigma:X\backslash\{x\}\cup\{y\}\rightarrow X$ renames $x$ to $y$. 
\end{con}
Notice that $D^{\EuScript O}_{xx}(f)=D^{\EuScript O}_{x}(f)$.  Other simple properties of actions $D_{xy}$ are given in the following two lemmas.
\begin{lem}
For $f\in{\EuScript O}(X)$ and $x\in X$, the following properties hold.\\[0.1cm]
a) For   $y\not\in X\backslash\{x\}$, \\[0.1cm]
\indent  {\em\texttt{(DID)'}} $D^{\EuScript O}_{xy}({\it id}_x)={\it id}_y$, and\\[0.1cm]
\indent  {\em\texttt{(DIN)'}} $D^{\EuScript O}_{yx}(D^{\EuScript O}_{xy}(f))=f$.\\[0.1cm]
b)  For the renaming   $\sigma:X\backslash\{x\}\cup\{y\}\rightarrow X$ of $x$ to $y$,   \\[0.1cm]
\indent {\em\texttt{(DEQ)'}} $D^{\EuScript O}_{y x}(f^{\sigma})=D^{\EuScript O}_{xy}(f)^{\sigma^{-1}}.$  \\[0.1cm]
 c) For $g\in {\EuScript O}(Y)$,\\[0.1cm]
\indent {\em\texttt{(DC1)'}}  $D^{\EuScript O}_{yu}(f\circ_{x}g)=D^{\EuScript O}_{yu}(f)\circ_{x}g$, where $y\in X$ and $u\not\in X\backslash\{y\}$.
\end{lem}
\begin{lem} \label{comp} {\em Composition.}
For $x,y\in X$ and $z\not\in X\backslash\{x,y\}$, \\[0.15cm]
\indent {\em\texttt{(DCO)}} $D^{\EuScript O}_{xy}(D^{\EuScript O}_{yz}(f))=D^{\EuScript O}_{xz}(f).$
\end{lem}
\begin{proof}
Since $D^{\EuScript O}_{xz}(f)=D_x(f)^{\sigma}=(D_x(f)^{\sigma_1})^{\sigma_2},$ where  $\sigma:X\backslash\{x\}\cup\{z\}\rightarrow X$ renames $x$ to $z$, $\sigma_1:X\rightarrow X$ renames $x$ to $y$ and $y$ to $x$, and $\sigma_2:X\backslash\{x\}\cup\{z\}\rightarrow X$ renames $y$ to $z$ and $x$ to $y$, by \texttt{(DEX)} we have  $$D^{\EuScript O}_{xz}(f)=(D^{\EuScript O}_x(D^{\EuScript O}_y(f)))^{\sigma_2}.$$

\indent For the left-hand side we have $D^{\EuScript O}_{xy}(D^{\EuScript O}_{yz}(f))=(D^{\EuScript O}_x(D^{\EuScript O}_y(f)^{\tau_1}))^{\tau_2},$
where $\tau_1:X\backslash\{y\}\cup\{z\}\rightarrow X$ renames $y$ to $z$ and $\tau_2:X\backslash\{x\}\cup\{z\}\rightarrow X\backslash\{y\}\cup\{z\}$ renames $x$ to $y$.
Therefore, by \texttt{(DEQ)},  $$D^{\EuScript O}_{xy}(D^{\EuScript O}_{yz}(f))=(D^{\EuScript O}_x(D^{\EuScript O}_y(f))^{\tau_1})^{\tau_2}=(D^{\EuScript O}_x(D^{\EuScript O}_y(f)))^{\tau_1\circ\tau_2}. $$ The conclusion follows from the equality $\sigma_2=\tau_1\circ\tau_2$.
\end{proof}
\indent We  make some preparations for the proof that Definition \ref{exoutput}  is equivalent to Definition \ref{entriesonly} (see Convention \ref{deq}). For an exchangeable-output cyclic operad ${\EuScript O}$ and a finite set $X$, we  introduce an equivalence relation $\approx$ on the set $\sum_{x\in X}{\EuScript O}(X\backslash\{x\})$ of (ordered) pairs $(x,f)$, where $x\in X$ and $f\in {\EuScript O}(X\backslash\{x\})$: $\approx$ is the reflexive closure of the familly of equalities \vspace{-0.1cm}\begin{equation}\label{d}(x,f)\approx (y,D_{yx}(f)),\vspace{-0.1cm}\end{equation} where $y\in X\backslash\{x\}$ is arbitrary.
\begin{rem}
By {\texttt{(DIN)'}} and {\texttt{(DCO)}} it follows that, for each $x\in X$, an equivalence class  $$[(x,f)]_{\approx}\in \sum_{x\in X}{\EuScript O}(X\backslash\{x\})/_{\approx}$$ has a unique representative of the form $(x,\rule{0.45em}{0.4pt}\,)$, i.e. if $(x,f)\approx(x,g)$, then $f=g$.
\end{rem}
In the next remark we exhibit a property of $\approx$ that we shall also need for the proof of the equivalence.
\begin{rem}\label{cong}
By  { \texttt{(DC1')}} and {\texttt{(DCO)}}, we have that  $(y, D_{yx}(f)\circ_x g)\approx (z,D_{zx}(f)\circ_x g). $  
\end{rem} 
\begin{thm}\label{2}
Definition \ref{exoutput} (exchangeable-output, componential) and   Definition \ref{entriesonly} (entries-only, componential)  are equivalent definitions of cyclic-operads.
\end{thm}
\begin{proof} The proof steps are the same as  in the proof of Theorem 3.14, except that we now show that the transition from one structure to the other one, and back, leads to an isomorphic structure.\\[0.15cm]
\textsc{[entries-only $\Rightarrow$ exchangeable-output]} Let ${\EuScript C}:{\bf Bij}^{op}\rightarrow {\bf Set}$  be an entries-only cyclic operad. For a finite set $X$ and a bijection $\sigma:Y\rightarrow X$, the species ${\EuScript O}_{\EuScript C}:{\bf Bij}^{op}\rightarrow {\bf Set}$,  underlying the corresponding exchangeable-output cyclic operad, is defined as $${\EuScript O}_{\EuScript C}(X)=\partial {\EuScript C}(X) \mbox{ \enspace \enspace and  \enspace\enspace} {\EuScript O}_{\EuScript C}(\sigma)=\partial {\EuScript C}(\sigma)={\EuScript C}(\sigma^{+}).$$ 
For $f\in {\EuScript O}_{\EuScript C}(X)$ and $g\in {\EuScript O}_{\EuScript C}(Y)$, the partial composition operation $\circ_{x}:{\EuScript O}_{\EuScript C}(X)\times {\EuScript O}_{\EuScript C}(Y)\rightarrow {\EuScript O}_{\EuScript C}(X\backslash\{x\}\cup Y)$ is defined by setting  \begin{equation}\label{aa} f\circ_{x} g= f^{\sigma}\, {{_{x}\circ_{\ast_{Y}}}}\,\, g, \end{equation} where $\sigma:X\cup \{\ast_{X\backslash\{x\}\cup Y}\}\rightarrow X\cup \{\ast_X\}$ renames $\ast_X$ to $\ast_{X\backslash\{x\}\cup Y}$. The distinguished element ${\it id}_x\in {\EuScript O}_{\EuScript C}(\{x\})$ is defined as ${\it id}_{x,\ast_{\{x\}}}$. Finally, for $f\in {\EuScript O}_{\EuScript C}(X)$ and $x\in X\cup\{\ast_{X}\}$, the action $D^{{\EuScript O}_{\EuScript C}}_{x}:{\EuScript O}_{\EuScript C}(X)\rightarrow {\EuScript O}_{\EuScript C}(X)$ is defined by setting  $$D^{{\EuScript O}_{\EuScript C}}_{x}(f)={\EuScript C}(\sigma)(f), $$ where $\sigma:X\cup \{\ast_{X}\}\rightarrow X\cup \{\ast_X\}$ exchanges $x$ and $\ast_{X}$. We verify the axioms.\\[0.1cm]
\indent \texttt{[A1]} Let $f$ and $g$ be as above and let $y\in X$  and $h\in {\EuScript O}_{\EuScript C}(Z)$. The sequence of equalities
{\small $$\begin{array}{rrrrrrrrrrrrrrclrr}
&&&&&&&&&&&&&(f \circ_{x} g) \circ_y h &=& (f^{\sigma}\, {{_{x}\circ_{\ast_{Y}}}}\,\, g)^{\tau} \, {{_{y}\circ_{\ast_{Z}}}}\,\, h&& \\[0.1cm]
&&&&&&&&&&&& && = & ((f^{\sigma})^{\tau_1}\, {{_{x}\circ_{\ast_{Y}}}}\,\, g^{\tau_2}) \, {{_{y}\circ_{\ast_{Z}}}}\,\, h &&  \enspace\enspace\enspace \enspace\enspace\enspace \enspace\enspace \enspace\enspace\enspace \enspace\enspace \enspace{{ \mbox{ \small\texttt{(EQ)}}}}\\[0.1cm]
&&&&&&&&&&&&& & = & ((f^{\sigma})^{\tau_1}\, {{_{y}\circ_{\ast_{Z}}}}\,\, h)  \, {{_{x}\circ_{\ast_{Y}}}}\,\, g && {{ \mbox{ \small\texttt{(A1)}}}}\\[0.1cm]
&&&&&&&&&&&&& & = & (f^{\kappa} \, {{_{y}\circ_{\ast_{Z}}}}\,\,   h) \, {{_{x}\circ_{\ast_{Y}}}}\,\, g & &  {{ \mbox{ \small\texttt{(EQ)}}}} \\[0.1cm]
&&&&&&&&&&&&& & = & (f \circ_{y} h) \circ_x g  &&
\end{array}$$}
where\\[-0.55cm]
{\small\begin{itemize}
\item $\sigma:X\cup\{\ast_{X\backslash\{x\}\cup Y}\}\rightarrow X\cup \{\ast_X\}$ renames $\ast_X$ to $\ast_{X\backslash\{x\}\cup Y}$, \\[-0.55cm]
\item $\tau:X\backslash\{x\}\cup Y\cup \{\ast_{X\backslash\{x,y\}\cup Y\cup Z}\} \rightarrow X\backslash\{x\}\cup Y\cup \{\ast_{X\backslash\{x\}\cup Y}\}$ renames $\ast_{X\backslash\{x\}\cup Y}$ to $\ast_{X\backslash\{x,y\}\cup Y\cup Z}$,\\[-0.55cm]
\item $\tau_1=\tau|^{X\backslash\{x\}\cup \{\ast_{X\backslash\{x\}\cup Y}\}}\cup id_{\{x\}}$, \\[-0.55cm]
\item $\tau_2=\tau|^{Y}={\it id}_{Y\cup \{\ast_{Y}\}}$, and\\[-0.55cm]
\item $\kappa: X\cup \{\ast_{X\backslash\{x\}\cup Z}\}\rightarrow X\cup \{\ast_{X}\}$ renames $\ast_{X}$ to $\ast_{X\backslash\{x\}\cup Z}$. 
\end{itemize}}
\noindent verifies  \texttt{[A1]} for ${\EuScript O}_{\EuScript C}$. The axiom \texttt{[A2]} follows similarily.\\[0.1cm]
\indent \texttt{[EQ]} For arbitrary bijections $\sigma_1:X'\rightarrow X$ and $\sigma_2:Y'\rightarrow Y$ we have {\small
$$\begin{array}{rrrrrrrrrrclrl}
&&&&&&&&& f^{\sigma_1^{+}}\circ_{\sigma^{-1}_1(x)} g^{\sigma^{+}_2} & = & (f^{\sigma_1^{+}})^{\tau }\, {{_{{{\sigma}_1^{+}}^{-1}(x)}\circ_{\ast_{Y'}}}}\,\,  g^{\sigma^{+}_2} & & \\[0.1cm]
&&&&&&&&&& = & (f^{\sigma_1^{+}})^{\tau } \, {{_{\tau^{-1}({\sigma_1^{+}}^{-1}(x))}\circ_{\ast_{Y'}}}}\,\,  g^{\sigma^{+}_2}  && \\[0.1cm]
&&&&&&&&&& = & (f\, {{_{x}\circ_{\ast_{Y}}}} g)^{\kappa}  && \enspace \enspace\enspace \enspace\enspace \enspace\enspace \enspace\enspace \enspace\enspace \enspace\enspace \enspace\,\,\,\,{{ \mbox{ \small\texttt{(EQ)}}}}\\[0.1cm]
&&&&&&&&&& = & (f^{\nu}\, {{_{x}\circ_{\ast_{Y}}}}\,\, g)^{\sigma^{+}} && \enspace\enspace\enspace\enspace\enspace\enspace\enspace\enspace\enspace\,\enspace\enspace\enspace\enspace \,\,\,\enspace{{ \mbox{ \small\texttt{(EQ)'}}}} \\[0.1cm]
&&&&&&&&&& = & (f\circ_{x} g)^{\sigma^{+}}&&  
\end{array}$$}
where\\[-0.55cm]
{\small\begin{itemize}
\item $\tau: X'\cup \{\ast_{X'\backslash\{\sigma_1^{-1}(x)\}\cup Y'}\} \rightarrow X'\cup \{\ast_{X'}\}$ renames $\ast_{X'}$ to ${\ast_{X'\backslash\{\sigma_1^{-1}(x)\}\cup Y'}}$,\\[-0.55cm]
\item $\nu: X\cup \{\ast_{X\backslash\{x\}\cup Y}\}\rightarrow X\cup \{\ast_X\}$ renames $\ast_X$ to $\ast_{X\backslash\{x\}\cup Y}$, \\[-0.55cm]
\item $\kappa=(\sigma_1^{+}\circ\tau)|^{X\backslash\{x\}\cup \{\ast_{X}\}}\cup \sigma_2^{+}|^{Y}$, and\\[-0.55cm]
\item $\sigma=\sigma_1|^{X\backslash\{x\}}\cup \sigma_2$. 
\end{itemize}}
\noindent Observe that $\kappa=(\nu|^{X\backslash\{x\}\cup\{\ast_{X}\}}\cup {\it id}_Y)\circ\sigma^{+},$  which justifies the application of \texttt{(EQ)'} to get the equality $(f\, {{_{x}\circ_{\ast_{Y}}}} g)^{\kappa} =(f^{\nu}\, {{_{x}\circ_{\ast_{Y}}}}\,\, g)^{\sigma^{+}} $. \\[0.1cm]
\indent \texttt{[U1]} By the axioms \texttt{(UP)} and \texttt{(U1)} for ${\EuScript C}$, for $f\in {\EuScript O}_{\EuScript C}(X)$ we have  $${\it id}_{y}\circ_y f  ={\it id}^{\sigma}_{y,\ast_{\{y\}}}\, {{_{y}\circ_{\ast_{X}}}}\,\, f={\it id}_{y,\ast_{X}}\,\, {{_{y}\circ_{\ast_{X}}}}\,\, f=f, $$
where $\sigma:\{y,\ast_{X}\}\rightarrow \{y,\ast_{\{y\}}\}$ renames $\ast_{\{y\}}$ to $\ast_{X}$.
 Analogously, the axioms \texttt{[U2]} and \texttt{[UP]}  for ${\EuScript O}_{\EuScript C}$ follow thanks to the corresponding laws of ${\EuScript C}$ (see Lemma \ref{j} and the axiom \texttt{(UP)}). \\[0.1cm]
\indent Concerning the axioms of the actions $D^{{\EuScript O}_{\EuScript C}}_{x}(f)$, \texttt{(DID)}, \texttt{(DIN)}, \texttt{(DEQ)}, and \texttt{(DEX)} follow easily by   functoriality of ${\EuScript C}$. The axioms \texttt{(DC1)} and \texttt{(DC2)}   additionally require the axiom \texttt{(EQ)} of ${\EuScript C}$. \\[0.1cm]
\indent \texttt{(DC1)} Let  $f\in {\EuScript O}_{\EuScript C}(X), g\in {\EuScript O}_{\EuScript C}(Y)$, $x\in X$ and $y\in X\backslash\{x\}$. We  need the following bijections: \\[-0.55cm]
{\small\begin{itemize}
\item $\sigma:X\backslash\{x\}\cup Y\cup \{\ast_{X\backslash\{x\}\cup Y}\}\rightarrow X\backslash\{x\}\cup Y\cup \{\ast_{X\backslash\{x\}\cup Y}\}$, which exchanges $y$ and $\ast_{X\backslash\{x\}\cup Y}$, \\[-0.55cm]
\item $\nu:X\cup \{\ast_{X\backslash\{x\}\cup Y}\}\rightarrow X\cup \{\ast_X\}$, which renames $\ast_X$ to $\ast_{X\backslash\{x\}\cup Y}$,\\[-0.55cm]
\item $\sigma':X\cup \{\ast_{X\backslash\{x\}\cup Y }\}\rightarrow X\cup \{\ast_{X\backslash\{x\}\cup Y}\}$, which exchanges $y$ and $\ast_{X\backslash\{x\}\cup Y}$, and\\[-0.55cm]
\item $\tau : X\cup \{\ast_{X}\}\rightarrow X\cup \{\ast_X\}$, which exchanges $y$ and $\ast_{X}$.\\[-0.55cm]
\end{itemize}}
\noindent Observe that $\tau\circ\nu=\nu\circ\sigma'$ and $\sigma=\sigma'|^{X\backslash\{x\}\cup \{\ast_{X\backslash\{x\}\cup Y}\}}\cup{\it id}_{Y}.$
We now have  $$ D^{\EuScript O}_{y}(f\circ_{x}g)=(f^{\nu} \, {{_{x}\circ_{\ast_{Y}}}}\,\, g)^{\sigma}= (f^{\nu})^{\sigma'}\, {{_{x}\circ_{\ast_{Y}}}}\,\,  g = (f^{\tau})^{\nu} \, {{_{x}\circ_{\ast_{Y}}}}\,\, g =D^{\EuScript O}_{yu}(f)\circ_{x}g\,.  $$

\indent  \texttt{(DC2)} Let $f, g$ and $x$ be like above and let $y\in Y$ instead. Let $\sigma_1$ and $\sigma_2$ be as in Definition \ref{exoutput}. 
We shall need the following bijections: \\[-0.55cm]
{\small\begin{itemize}
\item $\tau_1:Y\backslash\{y\}\cup \{v,\ast_{Y\backslash\{y\}\cup \{v\}}\}\rightarrow Y\cup \{\ast_Y\}$ that renames $y$ to $\ast_{Y\backslash\{y\}\cup \{v\}}$ and $\ast_Y$ to $v$, \\[-0.55cm]
\item $\tau_2:X\backslash\{x\}\cup \{y,\ast_{X\backslash\{x\}\cup \{y\}}\}\rightarrow X\cup\{\ast_X\}$ that renames $x$ to $\ast_{X\backslash\{x\}\cup \{y\}}$ and $\ast_X$ to $y$, \\[-0.55cm]
\item $\tau:Y\backslash\{y\}\cup \{v,\ast_{X\backslash\{x\}\cup Y}\}\rightarrow Y\backslash\{y\}\cup \{v,\ast_{Y\backslash\{y\}\cup \{v\}}\}$ that renames $\ast_{Y\backslash\{y\}\cup \{v\}}$ to $\ast_{X\backslash\{x\}\cup Y}$, \\[-0.55cm]
\item $\kappa_1:X\backslash\{x\}\cup \{y,\ast_{X\backslash\{x\}\cup \{y\}}\} \rightarrow X\cup \{\ast_{X\backslash\{x\}\cup Y}\}$ that renames $x$ to $\ast_{X\backslash\{x\}\cup \{y\}}$ and $\ast_{X\backslash\{x\}\cup Y}$ to $y$, \\[-0.55cm]
\item $\kappa_2:Y\backslash\{y\}\cup \{v,\ast_{X\backslash\{x\}\cup Y}\}\rightarrow Y\cup \{\ast_Y\}$ that renames $y$ to $\ast_{X\backslash\{x\}\cup Y}$ and $\ast_Y$ to $v$, \\[-0.55cm]
\item $\sigma:X\backslash\{x\}\cup Y\cup\{\ast_{X\backslash\{x\}\cup Y}\}\rightarrow X\backslash\{x\}\cup Y\cup\{\ast_{X\backslash\{x\}\cup  Y}\}$ that exchanges $y$ and $\ast_{X\backslash\{x\}\cup Y}$, and \\[-0.75cm]
\item $\nu:X\cup \{\ast_{X\backslash\{x\}\cup Y}\}\rightarrow X\cup \{\ast_X\}$ that renames $\ast_X$ to $\ast_{X\backslash\{x\}\cup Y}$. \\[-0.55cm]
\end{itemize}} 
\noindent Observe that $\tau_2=\nu\circ\kappa_1$ and $\kappa_2=\tau_1\circ\tau ,$
and  $\sigma=\kappa_1|^{X\backslash\{x\}\cup \{\ast_{X\backslash\{x\}\cup Y}\}}\cup \kappa_2|^{Y}.$ This gives us
{\small $$\begin{array}{rrrrrrrrrrrclr}
&&&&&&&&&&D^{\EuScript O}_{y}(g)^{\sigma_1}\circ_{v}D^{\EuScript O}_{x}(f)^{\sigma_2}&=& g^{\tau_1}\circ_v f^{\tau_2}&\\[0.1cm]
&&&&&&&&&&&=& (g^{\tau_1})^{\tau} \, {{_{v}\circ_{\ast_{X\backslash\{x\}\cup\{y\}}}}}\,\, f^{\tau_2}&\\[0.1cm]
&&&&&&&&&&&=& f^{\tau_2} \, {{_{\ast_{X\backslash\{x\}\cup \{y\}}}\circ_{v}}}\,\, (g^{\tau_1})^{\tau}&\enspace\enspace\enspace\enspace\enspace\enspace\enspace\enspace\enspace\enspace\enspace\enspace\enspace\enspace\enspace\enspace\small{\texttt{(CO)}}\\[0.1cm]
&&&&&&&&&&&=& (f^{\nu})^{\kappa_1} \, {{_{\ast_{X\backslash\{x\}\cup \{y\}}}\circ_{v}}}\,\, g^{\kappa_2}& \\[0.1cm]
&&&&&&&&&&&=& (f^{\nu}  \, {{_{x}\circ_{\ast_{Y}}}}\,\, g  )^{\sigma}&\small{\texttt{(EQ)}}\\[0.1cm]
&&&&&&&&&&&=& D^{\EuScript O}_{y}(f\circ_{x}g).&
\end{array}$$}

\noindent \textsc{[exchangeable-output $\Rightarrow$ entries-only]} Suppose that ${\EuScript O}:{\bf Bij}^{\it op}\rightarrow {\bf Set}$ is an exchangeable-output cyclic operad. The species ${\EuScript C}_{\EuScript O}:{\bf Bij}^{\it op}\rightarrow {\bf Set}$, underlying the cyclic operad in the entries-only fashion, is defined as  $${\EuScript C}_{\EuScript O}(X)={\sum_{x\in X}{\EuScript O}(X\backslash\{x\})}/_{\approx}. $$ Accordingly, for $[(x,f)]_{\approx}\in {\EuScript C}_{\EuScript O}(X)$ and a bijection $\sigma:Y\rightarrow X$, $${\EuScript C}_{\EuScript O}(\sigma)([(x,f)]_{\approx})=[(\sigma^{-1}(x),{\EuScript O}(\sigma|^{X\backslash\{x\}})(f))]_{\approx}.$$
\indent For $[(u,f)]_{\approx}\in{\EuScript C}_{\EuScript O}(X)$ and $[(v,g)]_{\approx}\in{\EuScript C}_{\EuScript O}(Y)$,  the partial composition operation ${{_{x}\circ_{y}}}:{\EuScript C}_{\EuScript O}(X)\times {\EuScript C}_{\EuScript O}(Y)\rightarrow {\EuScript C}_{\EuScript O}(X\backslash\{x\}\cup Y\backslash\{y\})$ is defined as follows: 
\[   
[(u,f)]_{\approx} {{_{x}\circ_{y}}} [(v,g)]_{\approx}  = 
     \begin{cases}
       [(z,D_{zx}(f)\circ_x g)]_{\approx}, &\quad\text{if } u=x \text{ and  } v=y, \\
       [(z, D_{zx}(f)\circ_x D_{yv}(g))]_{\approx}, &\quad\text{if } u=x \text{ and  } v\neq y, \\
       [(u,f\circ_{x} g)]_{\approx}, &\quad\text{if } u\neq x \text{ and  } v=y,\\
       [(u,f\circ_x D_{yv}(g))]_{\approx}, &\quad\text{if } u\neq x \text{ and  } v\neq y, \ 
     \end{cases} 
\]
where $z\in X\backslash\{x\}$ is arbitrary.\\ 
\indent For a two-element set, say $\{x,y\}$, the distinguished element ${\it id}_{x,y}\in {\EuScript C}_{\EuScript O}(\{x,y\})$ will be the equivalence class   $[(x,{\it id}_y)]_{\approx}$ (notice that, by \texttt{(DID)},  $(x,{\it id}_y)\approx (y,{\it id}_x)$).\\[0.1cm]
\indent Verifying that ${{_{x}\circ_{y}}}$ is well-defined requires checking that different  representatives of the classes that are to be composed lead to the same resulting class. Suppose, say, that  $(u,f)$ and $(v,g)$ are such that $u=x$ and  $v\neq y$ and let $s\in X\backslash\{x\}$ and $w\in Y\backslash\{v\}$ be arbitrary. Then, if, say, $w=y$, we have  
$$
[(s,D_{sx}(f))]_{\approx} \,\,{{_{x}\circ_{y}}}\,\, [(y,D_{yv}(g))]_{\approx} = [(s,D_{sx}(f)\circ_x D_{yv}(g))]_{\approx},\vspace{-0.1cm}$$ and
$(z, D_{zx}(f)\circ_x D_{yv}(g))\approx (s,D_{sx}(f)\circ_x  D_{yv}(g))$ by  Remark \ref{cong}. If $w\neq y$, then 
$$
[(s,D_{sx}(f))]_{\approx} \,\,{{_{x}\circ_{y}}}\,\, [(w,D_{wv}(g))]_{\approx} = [(s,D_{sx}(f)\circ_x D_{yw}(D_{wv}(g)))]_{\approx},  $$ and, by \texttt{(DCO)} and Remark \ref{cong}, we have   $$(s,D_{sx}(f)\circ_x D_{yw}(D_{wv}(g)))=(s,D_{sx}(f)\circ_x D_{yv}(g))\approx (z,D_{zx}(f)\circ_x D_{yv}(g)). $$ From  Remark \ref{cong} it also follows that   different choices of $z\in X\backslash\{x\}$ from the first two cases in the definition of $[(u,f)]_{\approx} {{_{x}\circ_{y}}} [(v,g)]_{\approx}$ lead to the same result. 
In the remaining of the proof, we shall assume that  $(x,f)$ and $(v,g)$  satisfy the conditions $u\neq x \text{ and } v=y$.   We check the axioms.\\[0.1cm]
\indent \texttt{(A1)} Let $[(u,f)]_{\approx}\in {\EuScript C}_{\EuScript O}(X)$, $[(y,g)]_{\approx}\in{\EuScript C}_{\EuScript O}(Y),$  $[(w,h)]_{\approx}\in{\EuScript C}_{\EuScript O}(Z)$, $x\in X$, $y\in Y$ and $w\in Z$. We prove the instance of associativity that requires the use of  \texttt{(DC2)} and \texttt{(DID)}, namely  $$([(u,f)]_{\approx} \,\,{_x\circ_y}\,\, [(y,g)]_{\approx}) \,\,{_u\circ_w}\,\, [(w,h)]_{\approx}=([(u,f)]_{\approx} \,\,{_u\circ_w}\,\, [(w,h)]_{\approx})\,\,{_x\circ_y}\,\, [(y,g)]_{\approx} . $$ Since ${\EuScript O}(\emptyset)=\emptyset$ and $g\in {\EuScript O}(Y\backslash\{y\})$ (resp. $h\in{\EuScript O}(Z\backslash\{w\}))$, we have that $Y\backslash\{y\}\neq\emptyset$ (resp. $Z\backslash\{w\}\neq\emptyset$). Suppose that $X\backslash\{x,u\}=\emptyset$.
For the expression on the left side of the equality we then have 
 $$ 
([(u,f)]_{\approx} \,\,{_x\circ_y}\,\, [(y,g)]_{\approx}) \,\,{_u\circ_w}\,\, [(w,h)]_{\approx}= [(z,D_{zu}(f\circ_x g)\circ_u h)]_{\approx},
 $$
where we chose $z\in Y\backslash\{y\}$. On the other hand, we have
 $$ 
([(u,f)]_{\approx} \,\,{_u\circ_w}\,\, [(w,h)]_{\approx})\,\,{_x\circ_y}\,\, [(y,g)]_{\approx}=[(v,D_{vx}(D_{xu}(f)\circ_u h)\circ_x g)]_{\approx} ,
 $$
where we chose $v\in Z\backslash\{w\}$. The associativity follows if we prove that $$D_{vz}(D_{zu}(f\circ_x g)\circ_u h)=D_{vx}(D_{xu}(f)\circ_u h)\circ_x g. $$ For this we use \texttt{(DC2)}, followed by \texttt{(DID)}, on both sides of the equality. We get  $$D_v(h)\circ_v D_{uz}(D_{zu}(f\circ_x g))=D_v(h)\circ_v (f\circ_x g) $$ on the left side and  $$(D_v(h)\circ_v D_{ux}(D_{xu}(f)))\circ_x g=(D_v(h)\circ_v f)\circ_x g $$ on the right side, and the conclusion follows by the axiom \texttt{[A1]} for $\EuScript O$.  If $X\backslash\{x,u\}\neq\emptyset$ and  $z\in X\backslash\{x,u\}$, the associativity   follows more directly by \texttt{(DC1)},  by choosing $v=z$.  \\[0.1cm]
\indent \texttt{(EQ)} Let $[(u,f)]_{\approx}\in {\EuScript C}_{\EuScript O}(X)$, $[(y,g)]_{\approx}\in{\EuScript C}_{\EuScript O}(Y),$ and let $\sigma_1:X'\rightarrow X$ and $\sigma_2:Y'\rightarrow Y$ be bijections. Let $\sigma^{-1}_1(x)=x'$, $\sigma_1^{-1}(u)=u'$ and $\sigma_2^{-1}(y)=y'$. 
We prove that $${\EuScript C}_{\EuScript O}(\sigma_1)([(u,f)]_{\approx})\,\,{_{{ x'}}\circ_{y'}}\,\, {\EuScript C}_{\EuScript O}(\sigma_2)( [(y,g)]_{\approx})={\EuScript C}_{\EuScript O}(\sigma)([(u,f)]_{\approx} {_x\circ_y} \,\, [(y,g)]_{\approx}),$$ where $\sigma=\sigma_1|^{X\backslash\{x\}}\cup \sigma_2|^{Y\backslash\{y\}}$. Let $\sigma'=\sigma|^{X\backslash\{x,u\}\cup Y\backslash\{y\}}$. We have 
{\small $$\begin{array}{llllllllrlrr}
&&&&&&&([(u,f)]_{\approx} {_x\circ_y} \,\, [(y,g)]_{\approx})^{\sigma}&=&[(u,f\circ_x g)]_{\approx}^{\sigma}&\\[0.15cm]
&&&&&&&&=&[(u',(f\circ_x g)^{\sigma'})]_{\approx}&\\[0.15cm]
&&&&&&&&=&[(u',f^{\tau_1}\circ_{x'}  g^{\tau_2})]_{\approx}& \enspace\enspace\enspace\enspace\enspace{\small{\texttt{[EQ]}}}\\[0.15cm]
&&&&&&&&=& [(u',f^{\sigma_1|^{X\backslash\{u\}}}\circ_{x'} g^{\sigma_2|^{Y\backslash\{y\}}})]&\\[0.15cm]
&&&&&&&&=&[(u', f^{\sigma_1|^{X\backslash\{u\}}})]_{\approx} \,\,{_{ {x'}}\circ_{y'}}\,\,  [(y',  g^{\sigma_2|^{Y\backslash\{y\}}})]_{\approx}&\\[0.15cm]
&&&&&&&&=&[(u,f)]_{\approx}^{\sigma_1}\,\,{_{ {x'}}\circ_{y'}}\,\,  [(y,g)]^{\sigma_2}_{\approx},&
\end{array}$$}

\noindent where  $\tau_1=\sigma|^{X\backslash\{u\}}\cup \sigma_1|^{x}$ and $\tau_2=\sigma|^{Y\backslash\{y\}}=\sigma_2|^{Y\backslash\{y\}}$. \\[0.1cm]
\indent \texttt{(U1)} For $[(u,f)]_{\approx}\in {\EuScript C}_{\EuScript O}(X)$, by \texttt{[U1]}, we have  {\small $$ [(y,{\it id}_x)]_{\approx}\,\,{_y\circ_u}\,\,[(u,f)]_{\approx}=[(x,D_{xy}({\it id}_x)\circ_{y}f)]_{\approx}=[(x,{\it id}_y\circ_y f)]_{\approx}=[(x, D_{xu}(f))]_{\approx}=[(u,f)]_{\approx}.$$}
\indent \texttt{(UP)} For $[(y,{\it id}_x)]_{\approx}\in {\EuScript C}_{\EuScript O}(\{x,y\})$, and a renaming $\sigma:\{u,v\}\rightarrow \{x,y\}$ of $x$ to $u$ and $y$ to $v$, $${\EuScript C}_{\EuScript O}(\sigma)({\it id}_{x,y})={\EuScript C}_{\EuScript O}(\sigma)([(y,{\it id}_x)]_{\approx})=[(\sigma^{-1}(y),{\EuScript O}(\sigma|^{\{x\}})({\it id}_x))]_{\approx}=[(v,{\it id}_u)]_{\approx}={\it id}_{u,v}.$$ 
\textsc{[the isomorphism of cyclic operads ${\EuScript C}$ and ${\EuScript C}_{{\EuScript O}_{\EuScript C}}$ (and ${\EuScript O}$ and ${\EuScript O}_{{\EuScript C}_{\EuScript O}}$)]} To complete the proof, it remains to show that 
 species ${\EuScript C}$ and ${\EuScript C}_{{\EuScript O}_{\EuScript C}}$ (resp. ${\EuScript O}$ and ${\EuScript O}_{{\EuScript C}_{\EuScript O}}$) are isomorphic, and that the exhibited isomorphism transfers the  partial composition $f\,\,{_{ {x}}\circ_{y}}\,\, g$ (resp. $f\circ_x g$)  to the partial composition of the images of $f$ and $g$ (under the same isomorphism), as well as that it preserves units. Categorically speaking, we are proving the isomorphism of cyclic operads (whose precise definition we give in Section 4).
The  isomorphism-of-species part, which is the same as in the proof of the equivalence of algebraic definitions, will be formally established in  Section 4.2 (Lemma \ref{fff}), as a consequence of the categorical equivalence given by Lamarche in \cite{l15} (which will be recalled in Section 4.1).  The components ${\phi_{\EuScript C}}_X:{\EuScript C}_{{\EuScript O}_{\EuScript C}}(X)\rightarrow {\EuScript C}(X)$ and ${\psi_{\EuScript O}}_X:{\EuScript O}(X)\rightarrow {\EuScript O}_{{\EuScript C}_{\EuScript O}}(X)$ of the isomorphisms ${\phi_{\EuScript C}}:{\EuScript C}_{{\EuScript O}_{\EuScript C}}\rightarrow {\EuScript C}$ and ${\psi_{\EuScript O}}:{\EuScript O}\rightarrow {\EuScript O}_{{\EuScript C}_{\EuScript O}}$, respectively, are defined as follows:\\[-0.55cm]
\begin{itemize}
\item ${\phi_{\EuScript C}}_X([(u,f)]_{\approx})=f^{\kappa}$, where $\kappa: X\rightarrow X\backslash\{u\}\cup \{\ast_{X\backslash\{u\}}\}$ renames $\ast_{X\backslash\{u\}}$ to $u$, and, \\[-0.55cm]
\item ${\psi_{\EuScript O}}_X(f)=[(\ast_X,f)]_{\approx}.$\\[-0.55cm]
\end{itemize}

\indent As for the corresponding partial composition translations, for $[(u,f)]_{\approx}\in {\EuScript C}_{{\EuScript O}_{\EuScript C}}(X)$,  $[(y,g)]_{\approx}\in {\EuScript C}_{{\EuScript O}_{\EuScript C}}(Y)$ and $x\in X\backslash\{u\}$ we have 
  $$[(u,f)]_{\approx}\,\,{_{ {x}}\circ_{y}}\,\, [(y,g)]_{\approx}=[(u,f\circ_x g)]_{\approx}=[(u,f^{\sigma}\,{_{ {x}}\circ_{\ast_{Y\backslash\{y\}}}}\,\, g)]_{\approx}=[(u,f^{\sigma}\,\,{_{ {x}}\circ_{y}}\,\, g^{\tau_2})]_{\approx}, $$
where $\sigma:X\backslash\{u\}\cup \{\ast_{X\backslash\{u,x\}\cup Y\backslash\{y\}}\}\rightarrow X\backslash\{u\}\cup \{\ast_{X\backslash\{u\}}\}$ renames  $\ast_{X\backslash\{u\}}$ to $\ast_{X\backslash\{u,x\}\cup Y\backslash\{y\}}$ and $\tau_2:Y\rightarrow Y\backslash\{y\}\cup \{\ast_{Y\backslash\{y\}}\}$ renames $\ast_{Y\backslash\{y\}}$ to $y$. Notice that for the last equality above we use the axiom \texttt{(EQ)} of both ${\EuScript C}_{{\EuScript O}_{\EuScript C}}$ and ${\EuScript C}$.
 The claim follows since  $${\phi_{\EuScript C}}_X([(u,f^{\sigma}\,\,{_{ {x}}\circ_{y}}\,\, g^{\tau_2})]_{\approx})=(f^{\sigma}\,\,{_{ {x}}\circ_{y}}\,\, g^{\tau_2})^{\kappa}=f^{\tau_1}\,\,{_{ {x}}\circ_{y}}\,\, g^{\tau_2}, $$
where $\kappa :X\backslash \{x\}\cup Y\backslash\{y\}\rightarrow X\backslash\{x,u\}\cup Y\backslash\{y\}\cup \{\ast_{X\backslash\{x,u\}\cup Y\backslash\{y\}}\}$ renames $\ast_{X\backslash\{x,u\}\cup Y\backslash\{y\}}$ to $u$ and $\tau_1:X\rightarrow X\backslash\{u\}\cup\{\ast_{X\backslash\{u\}}\}$ renames $u$ to $\ast_{X\backslash\{u\}}$, wherein the last equality above holds by the axiom \texttt{(EQ)} of ${\EuScript C}$.

For $f\in{\EuScript O}(X)$ and $g\in {\EuScript O}(Y)$, we have
 $$\begin{array}{rcl}
{\psi_{\EuScript O}}_X(f)\circ_x {\psi_{\EuScript O}}_Y(g)&=&[(\ast_X,f)]_{\approx}\circ_x [(\ast_Y,g)]_{\approx}\\[0.1cm]
&=&[(\ast_X,f)]^{\sigma} \,\,{_{ {x}}\circ_{\ast_Y}}\,\, [(\ast_Y,g)]\\[0.1cm]
&=&[(\ast_{X\backslash\{x\}\cup Y},f)]\,\,{_{ {x}}\circ_{\ast_Y}}\,\, [(\ast_Y,g)]_{\approx}\\[0.1cm]
&=&[(\ast_{X\backslash\{x\}\cup Y},f\circ_x g)]_{\approx}\\[0.1cm]
&=&{\psi_{\EuScript O}}_X(f\circ_x g),
\end{array} $$
where $\sigma:X\cup \{\ast_{X\backslash\{x\}\cup Y}\}\rightarrow X\cup \{\ast_X\}$ renames $\ast_{X}$ to $\ast_{X\backslash\{x\}\cup Y}$.\\
\indent For the unit elements,  we have  ${\phi_{\EuScript C}}_{\{x,y\}}([(x,{\it id}_{y,\ast_{\{y\}}})]_{\approx})={\it id}_{y,\ast_{\{y\}}}^{\kappa}={\it id}_{x,y},$ where $\kappa:\{x,y\}\rightarrow \{y,\ast_{\{y\}}\}$ renames $\ast_{\{y\}}$ to $x$,
and ${\psi_{\EuScript O}}_{\{x\}}({\it id}_{x})=[(\ast_{\{x\}},{\it id}_{x})]_{\approx}=[(x,{\it id}_{\ast_{\{x\}}})]_{\approx},$ which completes the proof of the theorem.
\end{proof}
\subsubsection{Algebraic definition}

If we think about the algebraic variant of  Definition \ref{exoutput}, it is clear that its cornerstone  should be an ordinary operad, i.e. a triple $(S,\nu,\eta_1)$ specified by Definition \ref{fiore}, and that the goal is to enrich this structure by a natural transformation which ``glues together'' the actions $D_{x}:S(X)\rightarrow S(X)$ and encompasses the coherence conditions these actions satisfy. We give the  definition below.

\begin{definition}[\textsc{exchangeable-output, algebraic}]\label{dddd}
A  cyclic operad is a  quadruple $(S,\nu,\eta_1, D)$, such that $(S,\nu,\eta_1)$ is an (algebraic) operad, and the natural transformation $D:\partial S\rightarrow \partial S$ satisfies the following laws:\\[0.15cm]
\indent{\em\texttt{(D0)}} $D\circ {\eta}^{\partial S}={\eta}^{\partial S}$,\\[0.1cm]
\indent{\em\texttt{(D1)}}  $D^{2}={\it id}_{\partial S}$,\\[0.1cm]
\indent{\em\texttt{(D2)}} $(\partial D\circ {\tt ex})^{3}={\it id}_{\partial\partial S},$\\[0.1cm] 
as well as the laws {\em\texttt{(D3)}} and {\em\texttt{(D4)}} given by  commutative  diagrams
\vspace{-0.1cm}\begin{center}
\begin{tikzpicture}[scale=1.5]
\node (A)  at (0,1) {\footnotesize $\partial(\partial S)\cdot S$};
\node (B) at (3.2,1) {\footnotesize $\partial(\partial S)\cdot S$};
\node (C) at (0,0.02) {\footnotesize $\partial S$};
\node (D) at (3.2,0.02) {\footnotesize $\partial S$};
\path[->,font=\footnotesize]
(A) edge node[above]{$({\tt ex}\circ\partial D\circ{\tt ex})\cdot {\it id}$} (B)
(A) edge node[left]{$\nu_3$} (C)
(B) edge node[right]{$\nu_3$} (D)
(C) edge node[below]{$D$} (D);
\end{tikzpicture}
\enspace\enspace 
\begin{tikzpicture}[scale=1.5]
\node (A)  at (0,0.75) {and};
\node (A)  at (0,0) {};
\end{tikzpicture}
 \enspace\enspace
\begin{tikzpicture}[scale=1.5]
\node (A)  at (0,1) {\footnotesize $\partial S\cdot\partial S$};
\node (B) at (2.6,1) {\footnotesize $\partial S\cdot\partial S$};
\node (C) at (0,0) {\footnotesize $\partial S$};
\node (D) at (2.6,0) {\footnotesize $\partial S\cdot\partial S$};
\node (E) at (1.3,0) {\footnotesize $\partial S$};
\path[->,font=\footnotesize]
(A) edge node[above]{$D\cdot D$} (B)
(A) edge node[left]{$\nu_4$} (C)
(B) edge node[right]{${\tt c}$} (D)
(C) edge node[below]{$D$} (E)
(D) edge node[below]{$\nu_4$} (E);
\end{tikzpicture}
\vspace{-0.3cm}\end{center}
respectively, 
where $\nu_3$ and $\nu_4$  are induced from $\nu$ as follows:
\vspace{-0.1cm} $${\small\begin{array}{cl}
\nu_3: \partial\partial S\cdot S\xrightarrow{\enspace i_l\enspace} \partial\partial S\cdot S+\partial S\cdot\partial\partial S \xrightarrow{\enspace \varphi^{-1}\enspace} \partial(\partial S\cdot\partial S)\xrightarrow{\enspace\partial\nu\enspace}\partial S, & \mbox{and}\\[0.1cm]
\nu_4: \partial S\cdot \partial S\xrightarrow{\enspace i_r\enspace} \partial\partial S\cdot S+\partial S\cdot\partial\partial S \xrightarrow{\enspace \varphi^{-1}\enspace} \partial(\partial S\cdot\partial S)\xrightarrow{\enspace\partial\nu\enspace}\partial S .&
\end{array}}$$
\end{definition}
That Definition \ref{dddd} is equivalent to Definition \ref{exoutput} will follow after the proof of the equivalence between Definition \ref{dddd} and Definition \ref{copeo} in the next section (see Table 2). As for a direct evidence, we content ourselves by showing  the correspondence between $D$ and the $D_x$'s, which we shall use in Section 4.1. Given $D:\partial S\rightarrow \partial S$, one defines $D_{x}:S(X)\rightarrow S(X)$ as \begin{equation}\label{dx}D_x=S(\sigma^{-1})\circ D_{X\backslash\{x\}}\circ S(\sigma),  \end{equation} where $\sigma:X\backslash\{x\}\cup \{\ast_{X\backslash\{x\}}\}\rightarrow X$ renames $x$ to $\ast_{X\backslash\{x\}}$. In the opposite direction, we define $D_X:\partial S(X)\rightarrow\partial S(X)$ via $D_x$ as  $D_X=D_{\ast_X}$. The correspondence between the axioms of $D$ and the ones of $D_{x}$ is  given in Table 4 below. In particular,  $(\partial D\circ {\tt ex})^{3}={\it id}_{\partial\partial S}$ corresponds exactly to the law \texttt{(DCO)} (that holds thanks to  \texttt{(DEQ)} and \texttt{(DEX)}).
\begin{center}
{\small  \begin{tabular}{lc}  
    \toprule
    $D$   &  $D_{x}$  \\
    \midrule
$D\circ {\eta}^{\partial S}={\eta}^{\partial S}$ & \texttt{(DID)}  \\[0.1cm]
    $D^{2}={\it id}_{\partial S}$    & \texttt{(DIN)} \\[0.1cm]
  $(\partial D\circ {\tt ex})^{3}={\it id}_{\partial\partial S}$ &   \texttt{(DEQ)}, \texttt{(DEX)}    \\[0.1cm]
 {\footnotesize{\textsc{commutative square} }} &  \texttt{(DC1)}      \\[0.1cm]
{\footnotesize{\textsc{commutative pentagone}}} & \texttt{(DC2)}      \\[0.1cm]
 \bottomrule
  \end{tabular}\\
\vspace{0.2cm}
Table 4}
\end{center}

\begin{rem}
Notice that, since ${\tt ex}\circ\partial D\circ{\tt ex}=\partial D\circ {\tt ex}\circ \partial D$, the diagram obtained by replacing $({\tt ex}\circ\partial D\circ{\tt ex})\cdot {\it id}$ with $(\partial D\circ {\tt ex}\circ \partial D)\cdot {\it id}$ in {\texttt{(D3)}} also commutes.
\end{rem}
 \section{The equivalence established}
This section deals with the proof of the equivalence between the two algebraic definitions of cyclic operads, Definition \ref{copeo} and Definition \ref{dddd}. Based on the equivalence between the category of species which are {\em empty on the empty set} and the category of {\em species with descent data}, established by Lamarche in \cite{l15}, this equivalence holds for {\em constant-free} cyclic operads, i.e. cyclic operads for which the underlying species $S$ is such that $S(\emptyset)=S(\{x\})=\emptyset$ (in the entries-only characterisation) and $S(\emptyset)=\emptyset$ (in the exchangeable-output characterisation), as we indicated in Section 2.2.
\subsection{Descent theory for species}
The equivalence of Lamarche comes from the background of descent theory. In the case of species, one starts  with the question
\begin{center}{\em Can we ``reconstruct'' a species $T$, given $\partial T$?} \end{center}
Intuitively, given the morphism ${\partial}^{+}:{\bf Bij}^{op}\rightarrow {\bf Bij}^{op}$ in ${\bf Cat}$, defined as ${\partial}^{+}(X)=X\cup \{\ast_X\}$, the idea is to  recover a morphism $T:{\bf Bij}^{\it op}\rightarrow {\bf Set}$ from $S=\partial T$ by ``descending" along ${\partial}^{+}$.
\vspace{-0.1cm}\begin{center}
\begin{tikzpicture}[scale=1.5]
\node (A)  at (0,1) {\footnotesize ${\bf Bij}^{\it op}$};
\node (B) at (2,1) {\footnotesize ${\bf Set}$};
\node (C) at (0,0) {\footnotesize ${\bf Bij}^{\it op}$};
\node (D) at (2,0.) {\footnotesize ${\bf Set}$};
\path[->,font=\footnotesize]
(A) edge node[above]{$\partial T$} (B)
(A) edge node[left]{$\partial^{+}$} (C)
(B) edge node[right]{${\it id}_{\bf Set}$} (D)
(C) edge node[below]{$T$} (D);
\end{tikzpicture}\vspace{-0.2cm}
\end{center}
Such a reconstruction is clearly not possible without an additional data, called the {\em descent data}, that compensates  the  loss  of  information caused by the action of the  functor $\partial : {\bf Spec}\rightarrow {\bf Spec}$. \\
\indent  Lamarche \cite{l15} defines a descent data as a pair $(S,D)$ of a species $S$ and a natural transformation $D:\partial S\rightarrow \partial S$, such that  $D^{2}={\it id}_{\partial S}$, and $(\partial D\circ{\tt ex}_S)^{3}={\it id}_{\partial\partial S}$, and he proves that {\em the assignment $\partial: {\bf Spec}/_{\emptyset} \rightarrow  {\bf Spec}^{+}$, defined as
$$ 
  T \mapsto  (\partial T, {\tt ex}_{T}),$$
is an equivalence of categories\footnote{Lamarche
proves this equivalence  in a skeletal setting, by considering functors of the form   $S:{\bf Fin}^{\it op}\rightarrow {\bf Set}$, where ${\bf Fin}$  denotes the
category of finite cardinals and permutations. The non-skeletal version that we present is an easy adaptation of his result.}.}  Here ${\bf Spec}/_{\emptyset}$ denotes the category of species $S$ such that $S(\emptyset)=\emptyset$ and ${\bf Spec}^{+}$ denotes the category of descent data. For $(S,D)\in {\bf Spec}^{+}$, the inverse functor $\int:{\bf Spec}^{+}\rightarrow{\bf Spec}/_{\emptyset}$ is defined as  $$\int (S,D)(X)=\sum_{x\in X}S(X\backslash\{x\})/_{\approx}, $$ where $\approx$ is defined  as in \eqref{d}, whereby the actions $D_{x}$ are defined via $D$ as in \eqref{dx}.
\subsection{The main theorem}
Let  ${\bf Spec}/_{\emptyset,\{\ast\}}$ be the subcategory of ${\bf Spec}/_{\emptyset}$ of species $S$ such that $S(\emptyset)=S(\{x\})=\emptyset$, for all singletons $\{x\}$, and let  ${\bf Spec}^{+}/_{\emptyset}$ be the subcategory of  ${\bf Spec}^{+}$  of descent data with species $S$  such that $S(\emptyset)=\emptyset$. The next result is a direct consequence of the equivalence of Lamarche from Section 4.1.
\begin{lem}\label{fff}
The assignment $\partial:{\bf Spec}/_{\emptyset,\{\ast\}}\rightarrow {\bf Spec}/_{\emptyset}^{+}$, defined in the same way as in 4.1, is an equivalence of categories.
\end{lem}
Let ${\bf CO_{en}(Spec}/_{\emptyset,\{\ast\}})$ be the category of entries-only  cyclic operads $(S,\rho,\eta_2)$ such that $S$ is an object of ${\bf Spec}/_{\emptyset,\{\ast\}}$, and let ${\bf CO_{ex}(Spec}/_{\emptyset}^{+})$ be the category of exchangeable-output cyclic operads $(S,\nu,\eta_1,D)$ such that $(S,D)$ is an object of ${\bf Spec}^{+}/_{\emptyset}$. 
In both of these categories, the (iso)morphisms are natural transformations (natural isomorphisms)  between underlying species which preserve the cyclic-operad structure. \\[0.1cm]
\indent The main result of this work is the proof that the equivalence of Lamarche   carries over, via Lemma \ref{fff}, to an equivalence between the  two algebraic definitions of cyclic operads, formally given as  the categorical equivalence between the two categories introduced above. 
\begin{rem}\label{rejs}
The reason for restricting the equivalence of Lamarche to the one of Lemma \ref{fff} (and, therefore, to the equivalence of constant-free cyclic operads) lies in the fact that, given a species $S$ from ${\bf Spec}/_{\emptyset}$, the constraint $S(\emptyset)=\emptyset$ makes the component $\rho_{\emptyset}:(\partial S\cdot \partial S)(\emptyset)\rightarrow S(\emptyset)$ of the multiplication $\rho:\partial S\cdot \partial S\rightarrow S$  the empty function, in which case  the condition $S(\{\ast_{\emptyset}\})=\emptyset$  is needed in order for the domain of $\rho_{\emptyset}$ to also be the empty set. Therefore, in the context of cyclic operads, we have to consider ${\bf Spec}/_{\emptyset,\{\ast\}}$ instead of ${\bf Spec}/_{\emptyset}$, and, consequently, ${\bf Spec}/_{\emptyset}^{+}$ instead of ${\bf Spec}^{+}$.
\end{rem}
\begin{thm}\label{3}
The categories ${\bf CO_{en}(Spec}/_{\emptyset,\{\ast\}})$ and ${\bf CO_{ex}(Spec}/_{\emptyset}^{+})$ are equivalent.
\end{thm}
\begin{proof} We follow the same steps as we did for the previous two theorems. The precise definitions of the  functors  and  natural trasformations  that constitute  the equivalence are cumbersome, but easy to derive from the transitions we make below.\\[0.15cm]
{\textsc{[exchangeable-output $\Rightarrow$ entries-only]}} Given a cyclic operad ${\EuScript O}=(T,\nu^{_T},{\eta_1}^{_T},D^{_T})$  from  ${\bf CO_{ex}(Spec}/_{\emptyset}^{+})$,  by  Lemma \ref{fff} we know that $(T,D^{_T})\simeq (\partial S,{\tt ex}_S)$, for some species $S$ from $\bf {Spec}/_{\emptyset,\{\ast\}}$. Together with the definitions of $\nu^{_T}$ and $\eta^{_T}$, this equivalence gives rise to an operad $(\partial S,\nu^{_{\partial S}},{\eta_1}^{_{\partial S}},{\tt ex}_S)$ such that 
${\EuScript O}\simeq(\partial S,\nu^{_{\partial S}},{\eta_1}^{_{\partial S}},{\tt ex}_S)$.  Since $\int(\partial S,{\tt ex}_S)\simeq S$,  defining a cyclic operad over the species $\int T$  amounts to defining a cyclic operad ${\EuScript C}_{\EuScript O}=(S,{\rho_{\nu}}^{_S},{\eta_2}^{_S})$ over the species $S$. We define ${\EuScript C}_{\EuScript O}$ below, whereby we shall write  $\rho$    for ${\rho_{\nu}}^{_S}$ and $\eta_2$ for ${\eta_2}^{_S}$.\\
\indent For $X=\emptyset$, $\rho_X:(\partial S\cdot \partial S)(X)\rightarrow S(X)$  is the empty function. For $X\neq \emptyset$, defining $\rho_X$ amounts to defining $\rho'_X:\partial(\partial S\cdot\partial S)(X)\rightarrow \partial S(X).$ We set $\rho'=[\rho'_1,\rho'_2]\circ\varphi$, where $\rho'_1:\partial\partial S\cdot\partial S\rightarrow\partial S$   and $\rho'_2:\partial S\cdot \partial\partial S\rightarrow \partial S$ are determined as 
 $$\rho'_1:\partial\partial S\cdot\partial S\xrightarrow{\enspace{\tt ex}\cdot {\it id}\enspace}\partial\partial S\cdot\partial S\xrightarrow{\enspace\nu\enspace}\partial S \mbox{ \enspace\enspace\enspace and \enspace\enspace\enspace} \rho'_2=\rho'_1\circ {\tt c}. $$
 Defining ${\eta_{2}}$ amounts to defining $\partial{\eta_{2}}:\partial E_2\rightarrow \partial S$, for which we set $\partial{\eta_{2}}={\eta_{1}}^{_{\partial S}}\circ\epsilon_{2}$.  We verify the axioms.\\[0.1cm]
\indent \texttt{(CA1)} By Corollary \ref{oh}, the axiom  \texttt{(CA1)} for  ${\EuScript C}_{\EuScript O}$ comes down to   the equality $\rho_{21}\circ\gamma_1=\rho_{11}$, whereas  $\rho_{21}\circ\gamma_1=\rho_{11}$ follows from   $\partial \rho_{21}\circ\partial\gamma_1=\partial \rho_{11}$. We prove the latter equality.

\indent In Diagram 4, the triangle $T$ is the diagram whose commutation we  aim to prove and the diagrams $L$ and $R$ are obtained by unfolding the definitions of $\partial\rho_{11}$ and $\partial\rho_{21}$.  We then express  $\partial(\rho'\cdot{\it id})$ in $L$ and $R$  as $\partial([\rho'_1,\rho'_2]\cdot{\it id})\circ \partial (\varphi\cdot{\it id})$. Snce $\partial (\varphi\cdot{\it id})\circ \partial (\varphi^{-1}\cdot{\it id})={\it id}$ and  $\partial([\rho'_1,\rho'_2]\cdot{\it id})\circ\partial(i\cdot{\it id})=\partial(\rho'_1\cdot{\it id})$, we transformed  Diagram 4 into  Diagram 5, in which     $L'$ and $R'$  commute.
\begin{center}
\begin{tikzpicture}[scale=1.5]
\node (U5)  at (1.375,-1.8) {\tiny $\partial( \partial S\cdot \partial S)$};
\node (U6)  at (1.375,-2.9) {\tiny $ \partial ((\partial\partial S\cdot\partial S+\partial S\cdot\partial\partial S)\cdot \partial S)$};
\node (L1)  at (-1.75,0) {\tiny $\partial ((\partial \partial S\cdot \partial S)\cdot \partial S)$};
\node (M1)  at (4.5,0) {\tiny $\partial ((\partial \partial S\cdot \partial S)\cdot \partial S)$};
\node (M0)  at (5,-1) {\tiny $\partial ((\partial \partial S\cdot \partial S+\partial S\cdot \partial\partial S)\cdot \partial S)$};
\node (M3)  at (-2.25,-1) {\tiny $\partial ((\partial \partial S\cdot \partial S+\partial S\cdot \partial\partial S)\cdot \partial S)$};
\node (M4)  at (-2.25,-1.8) {\tiny $\partial (\partial(\partial S\cdot\partial S)\cdot \partial S)$};
\node (K2)  at (5,-1.8) {\tiny $\partial (\partial(\partial S\cdot\partial S)\cdot \partial S)$};
\node (M2)  at (1.375,-1) {\tiny $\partial S$};
\node (M9)  at (1.375,-0.5) {\large $T$};
\node (M10)  at (3.25,-1.2) {\large $R$};
\node (M11)  at (-0.5,-1.2) {\large $L$};
\path[->,font=\footnotesize]
(L1) edge node[above]{\tiny $\partial\gamma_1$} (M1)
(L1) edge node[below]{\tiny $\partial\rho_{11}$} (M2)
(M1) edge node[below]{\tiny $\partial\rho_{21}$} (M2)
(M1) edge node[right]{\tiny $\partial (i\cdot{\it id})$} (M0)
(L1) edge node[left]{\tiny $\partial (i\cdot{\it id})$} (M3)
(M3) edge node[left]{\tiny $\partial (\varphi^{-1}\cdot{\it id})$} (M4)
(M0) edge node[right]{\tiny $\partial (\varphi^{-1}\cdot{\it id})$} (K2)
(K2) edge node[above]{\tiny $\partial (\rho'\cdot{\it id})$} (U5)
(M4) edge node[above]{\tiny $\partial (\rho'\cdot{\it id})$} (U5)
(K2) edge node[below,yshift=-0.1cm]{\tiny $\partial (\varphi\cdot{\it id})$} (U6)
(M4) edge node[below,yshift=-0.1cm]{\tiny $\partial (\varphi\cdot{\it id})$} (U6)
(U5) edge node[right]{\tiny $\rho'$} (M2)
(U6) edge node[right]{\tiny $\partial([\rho'_1,\rho'_2]\cdot{\it id})$} (U5);
\end{tikzpicture}\\
\vspace{0.1cm}
{\small Diagram 4.}
\end{center}
\begin{center}
\begin{tikzpicture}[scale=1.5]
\node (L1)  at (-1.75,0) {\tiny $\partial ((\partial \partial S\cdot \partial S)\cdot \partial S)$};
\node (M1)  at (4.5,0) {\tiny $\partial ((\partial \partial S\cdot \partial S)\cdot \partial S)$};
\node (M2)  at (1.375,-0.8) {\tiny $\partial S$};
\node (U5)  at (1.375,-1.6) {\tiny $\partial( \partial S\cdot \partial S)$};
\node (L')  at (0.7,-1) {$L'$};
\node (R')  at (2.05,-1) {$R'$};
\node (M9)  at (1.375,-0.4) {\large $T$};
\path[->,font=\footnotesize]
(L1) edge node[above]{\tiny $\partial\gamma_1$} (M1)
(L1) edge node[left,yshift=-0.1cm]{\tiny $\partial(\rho'_1\cdot{\it id})$} (U5)
(M1) edge node[right,yshift=-0.1cm]{\tiny $\partial(\rho'_1\cdot{\it id})$} (U5)
(L1) edge node[below]{\tiny $\partial\rho_{11}$} (M2)
(M1) edge node[below]{\tiny $\partial\rho_{21}$} (M2)
(U5) edge node[right]{\tiny $\rho'$} (M2);
\end{tikzpicture}\\
\vspace{0.1cm}
{\small Diagram 5.}
\end{center}
Therefore, the equality that needs to be proven is $A\circ\partial\gamma_1=A$, where $A=\rho'\circ \partial(\rho'_1\cdot{\it id}).$
By implementing $\rho'=[\rho'_1,\rho'_2]\circ\varphi$ in $A$  and then by using the  equalities $[\partial \rho'_1\cdot {\it id},\rho'_1\cdot\partial {\it id}]\circ\varphi=\varphi\circ\partial(\rho'_1\cdot {\it id})$ and $\partial{\it id}={\it id},$  $A$ is reformulated as $A=[B,C]\circ\varphi,$ where $B=\rho'_1\circ(\partial\rho'_1\cdot{\it id})$ and $C=\rho'_2\circ (\rho'_1\cdot{\it id})$. By setting $\Gamma=\varphi\circ\partial\gamma_1\circ\varphi^{-1}$, the equality $A\circ\partial\gamma_1=A$ reformulates as $[B,C]=[B,C]\circ\Gamma\,.$ Since $\Gamma$ can be expressed by the commutation of Diagram 6\footnote{	In the top horizontal arrow of Diagram 6, the first ${\tt ex}\cdot{\tt c}$ maps the second summand on the left to the third one on the right, and the second ${\tt ex}\cdot{\tt c}$ maps the third summand on the left to the second one on the right.}, by setting  $D=B\circ (\varphi^{-1}\cdot{\it id})\circ(i_l\cdot{\it id})$ and $E=B\circ (\varphi^{-1}\cdot{\it id})\circ(i_r\cdot{\it id}),$ the equality
$[B,C]=[B,C]\circ\Gamma$ is proven if the following three equalities hold:
 $$ D=D\circ (\alpha^{-1}\circ(\partial{\tt ex}\cdot {\tt c})\circ{\alpha}), \enspace \enspace \enspace E=C\circ(\alpha^{-1}\circ({\tt ex}\cdot {\tt c})\circ\alpha)\enspace\enspace  \enspace  \mbox{ and }\enspace \enspace \enspace C=E\circ(\alpha^{-1}\circ({\tt ex}\cdot {\tt c})\circ\alpha).
 $$
\vspace{-1cm}
\begin{center}
\begin{tikzpicture}[scale=1.5]
\node (U3)  at (4.75,1) {\tiny $(\partial\partial \partial S\cdot \partial S+\partial\partial S\cdot\partial\partial S)\cdot\partial S+(\partial\partial S\cdot\partial S)\cdot\partial\partial S$};
\node (B2)  at (4.75,2) {\tiny $(\partial\partial \partial S\cdot \partial S)\cdot \partial S+(\partial\partial S\cdot\partial\partial S)\cdot\partial S+(\partial\partial S\cdot\partial S)\cdot\partial\partial S$};
\node (U2)  at (-1.7,1) {\tiny $(\partial\partial \partial S\cdot \partial S+\partial\partial S\cdot\partial\partial S)\cdot\partial S+(\partial\partial S\cdot\partial S)\cdot\partial\partial S$};
\node (B1)  at (-1.7,2) {\tiny $(\partial\partial \partial S\cdot \partial S)\cdot\partial S+(\partial\partial S\cdot\partial\partial S)\cdot\partial S+(\partial\partial S\cdot\partial S)\cdot\partial\partial S$};
\node (B0)  at (-1.7,3) {\tiny $\partial\partial \partial S\cdot (\partial S\cdot\partial S)+\partial\partial S\cdot(\partial\partial S\cdot\partial S)+\partial\partial S\cdot(\partial S\cdot\partial\partial S)$};
\node (B5)  at (4.75,3) {\tiny $\partial\partial \partial S\cdot (\partial S\cdot\partial S)+\partial\partial S\cdot(\partial\partial S\cdot\partial S)+\partial\partial S\cdot(\partial S\cdot\partial\partial S)$};
\node (U4)  at (4.75,0) {\tiny $\partial (\partial \partial S\cdot \partial S)\cdot \partial S+(\partial\partial S\cdot\partial S)\cdot\partial\partial S$};
\node (U1)  at (-1.7,0) {\tiny $\partial (\partial \partial S\cdot \partial S)\cdot \partial S+(\partial\partial S\cdot\partial S)\cdot\partial\partial S$};
\path[->,font=\footnotesize]
(U2) edge node[left]{\tiny $\Delta+{\it id}$} (B1)
(B1) edge node[left]{\tiny $\alpha+\alpha+\alpha$} (B0)
(B0) edge node[above]{\tiny $\partial{\tt ex}\cdot{\tt c}+{\tt ex}\cdot{\tt c}+{\tt ex}\cdot{\tt c}$} (B5)
(B2) edge node[right]{\tiny $\Delta^{-1}+{\it id}$} (U3)
(B5) edge node[right]{\tiny $\alpha^{-1}+\alpha^{-1}+\alpha^{-1}$} (B2)
(U1) edge node[above]{\tiny $\Gamma$} (U4)
(U1) edge node[left]{\tiny $\varphi\cdot{\it id}+{\it id}$} (U2)
(U3) edge node[right]{\tiny $\varphi^{-1}\cdot{\it id}+{\it id}$} (U4);
\end{tikzpicture}\\
\vspace{0.15cm}
{\small Diagram 6.}
\end{center}

\indent Therefore, the first equality that needs to be proven is $$\rho'_1\circ(\partial\rho'_1\cdot{\it id})\circ(\varphi^{-1}\cdot{\it id})\circ(i_l\cdot{\it id})=\rho'_1\circ(\partial\rho'_1\cdot{\it id})\circ(\varphi^{-1}\cdot{\it id})\circ(i_l\cdot{\it id})\circ(\alpha^{-1}\circ(\partial{\tt ex}\cdot {\tt c})\circ{\alpha}),  $$
and the outer part of Diagram 7 corresponds exactly to this equality once the definition of $\rho'$ (via $\nu$) is unfolded.
The rest of the arrows  show that the outer part indeed commutes. Notice  that \\[-0.7cm]
\begin{itemize}
\item $J_l$ and $J_r$ commute since they are   the commuting squares from Remark 3.24 and Definition \ref{dddd}, respectively, once the definition of $\nu_3$ is unfolded and $D$ is set to be ${\tt ex}$,\\[-0.7cm]
\item $K$ commutes as it represents the equality $\nu_{21}\circ\beta_1=\nu_{11}$  (see Definition \ref{fiore}), and \\[-0.7cm]
\item  $I_l, I_r, L$ and $M$ commute as they represent  naturality conditions for $\varphi$ and $\alpha$.
\end{itemize}
\begin{center}
\begin{tikzpicture}[scale=1.5]
\node (L1)  at (0,0) {\footnotesize ($\partial \partial \partial S\cdot \partial S)\cdot \partial S$};
\node (I_l)  at (1.2,-1.4) {$I_l$};
\node (I_r)  at (7.1,-1.4) {$I_r$};
\node (L2)  at (0,-1) {\footnotesize $\partial (\partial \partial S\cdot \partial S)\cdot \partial S$};
\node (L3)  at (0,-2) {\footnotesize $\partial (\partial \partial S\cdot \partial S)\cdot \partial S$};
\node (L4)  at (0,-3) {\footnotesize $\partial \partial S\cdot \partial S$};
\node (L5)  at (0,-4) {\footnotesize $\partial \partial S\cdot \partial S$};
\node (M1)  at (2.75,0) {\footnotesize $\partial \partial \partial S\cdot (\partial S\cdot \partial S)$};
\node (M2)  at (5.5,0) {\footnotesize $\partial \partial \partial S\cdot (\partial S\cdot \partial S)$};
\node (R1) at (8.25,0) {\footnotesize ($\partial \partial \partial S\cdot \partial S)\cdot \partial S$};
\node (R2)  at (8.25,-1) {\footnotesize $\partial (\partial \partial S\cdot \partial S)\cdot \partial S$};
\node (R3)  at (8.25,-2) {\footnotesize $\partial (\partial \partial S\cdot \partial S)\cdot \partial S$};
\node (R4)  at (8.25,-3) {\footnotesize $\partial \partial S\cdot \partial S$};
\node (R5)  at (8.25,-4) {\footnotesize $\partial \partial S\cdot \partial S$};
\node (BM)  at (4.125,-5) {\footnotesize $\partial S$};
\node (BJ)  at (4.125,-4) { $K$};
\node (BC)  at (4.125,-2.5) { $L$};
\node (BC)  at (4.125,-1) { $M$};
\node (BJL)  at (1.2,-3) { $J_l$};
\node (BJR)  at (7.1,-3) { $J_r$};
\node (IL1)  at (2.75,-2) {\tiny ($\partial \partial \partial S\cdot \partial S)\cdot \partial S$};
\node (IL2)  at (2.75,-3) {\tiny ($\partial \partial \partial S\cdot \partial S)\cdot \partial S$};
\node (IL3)  at (2.75,-4) {\tiny $\partial (\partial \partial S\cdot \partial S)\cdot \partial S$};
\node (IR1)  at (5.5,-2) {\tiny ($\partial \partial \partial S\cdot \partial S)\cdot \partial S$};
\node (IR2)  at (5.5,-3) {\tiny ($\partial \partial \partial S\cdot \partial S)\cdot \partial S$};
\node (IR3)  at (5.5,-4) {\tiny $\partial (\partial \partial S\cdot \partial S)\cdot \partial S$};
\path[->,font=\footnotesize]
(L1) edge node[above]{$\scriptsize\alpha$} (M1)
(R1) edge node[sloped,above]{\tiny $(\partial{\tt ex}\cdot{\it id})\cdot{\it id}$} (IR1)
(L1) edge node[left]{\scriptsize $(\varphi^{-1}\circ i_l)\!\cdot\!{\it id}$} (L2)
(L2) edge node[left]{\scriptsize $\partial ({\tt ex}\cdot id)\cdot{\it id}$} (L3)
(L3) edge node[left]{\scriptsize $\partial\nu \cdot {\it id}$} (L4)
(L4) edge node[left]{\scriptsize ${\tt ex}\cdot{\it id}$} (L5)
(L5) edge node[below]{\scriptsize $\nu$} (BM)
(R1) edge node[right]{\scriptsize $(\varphi^{-1}\circ i_l)\!\cdot\!{\it id}$} (R2)
(R2) edge node[right]{\scriptsize $\partial ({\tt ex}\cdot id)\cdot{\it id}$} (R3)
(R3) edge node[right]{\scriptsize $\partial\nu \cdot {\it id}$} (R4)
(R4) edge node[right]{\scriptsize ${\tt ex}\cdot{\it id}$} (R5)
(R5) edge node[below]{\scriptsize $\nu$} (BM)
(M1) edge node[above]{\scriptsize $\partial{\tt ex}\cdot {\tt c}$} (M2)
(M2) edge node[above]{\scriptsize ${\alpha}^{-1}$} (R1)
(IL1) edge node[above]{\tiny $ (\varphi^{-1}\circ i_l)\!\cdot\!{\it id}$} (L3)
(IR1) edge node[above]{\tiny $ (\varphi^{-1}\circ i_l)\!\cdot\!{\it id}$} (R3)
(IL1) edge node[left]{\tiny $(\partial {\tt ex}\circ {\tt ex}\circ \partial{\tt ex})\cdot {\it id}$} (IL2)
(IL2) edge node[left]{\tiny $({\varphi}^{-1}\circ i_l)\cdot{\it id}$} (IL3)
(IL1) edge node[above]{\tiny ${\alpha^{\cdot}}^{\scalebox{0.75}{-1}}\!\circ\! (\partial{\tt ex}\cdot c)\!\circ\!\alpha^{\cdot}$} (IR1)
(IL3) edge node[above]{\tiny $\partial \nu\cdot {\it id}$} (L5)
(IR1) edge node[right]{\tiny $({\tt ex}\circ \partial{\tt ex}\circ {\tt ex})\cdot {\it id}$} (IR2)
(IR2) edge node[right]{\tiny $({\varphi}^{-1}\circ i_l)\cdot{\it id}$} (IR3)
(IR3) edge node[above]{\tiny $\partial \nu\cdot {\it id}$} (R5)
(IL2) edge node[above]{\tiny ${\alpha^{\cdot}}^{\scalebox{0.75}{-1}}\!\circ\! ({\tt ex}\cdot c)\!\circ\!\alpha^{\cdot}$} (IR2)
(L1) edge node[sloped,above]{\tiny $(\partial{\tt ex}\cdot{\it id})\cdot{\it id}$} (IL1);
\end{tikzpicture}\\
\vspace{0.1cm}
{\small Diagram 7.}
\end{center}
\vspace{-0.7cm}
\indent The second equality is $\rho'_1\circ(\partial\rho'_1\cdot{\it id})\circ(\varphi^{-1}\cdot{\it id})\circ(i_r\cdot{\it id})=\rho'_2\circ (\rho'_1\cdot{\it id})\circ\alpha^{-1}\circ({\tt ex}\cdot {\tt c})\circ{\alpha},$ and the corresponding diagram is (the outer part of) Diagram 8.
This diagram commutes because \\[-0.7cm]
\begin{itemize}
\item $I$ is the commuting pentagon from Definition \ref{dddd}, once the definition of $\nu_4$ is unfolded and $D$ is set to be ${\tt ex}$, and \\[-0.7cm]
\item $J$ commutes as it corresponds  to the equality $\nu_{22}\circ\beta_2=\nu_{12}$ (see Definition \ref{fiore}). 
\end{itemize} 
\begin{center}
\begin{tikzpicture}[scale=1.5]
\node (L1)  at (0,0) {\footnotesize ($\partial \partial S\cdot \partial\partial S)\cdot \partial S$};
\node (L2)  at (0,-1) {\footnotesize $\partial (\partial\partial S\cdot \partial S)\cdot \partial S$};
\node (L3)  at (0,-2) {\footnotesize $\partial (\partial \partial S\cdot \partial S)\cdot \partial S$};
\node (L4)  at (0,-3) {\footnotesize $\partial \partial S\cdot \partial S$};
\node (L5)  at (0,-4) {\footnotesize $\partial \partial S\cdot \partial S$};
\node (M1)  at (2.75,0) {\footnotesize $\partial \partial  S\cdot ( \partial\partial S\cdot \partial S)$};
\node (M2)  at (5.5,0) {\footnotesize $\partial \partial  S\cdot (\partial S\cdot \partial\partial S)$};
\node (R1) at (8.25,0) {\footnotesize ($\partial \partial S\cdot \partial S)\cdot \partial\partial S$};
\node (R2)  at (8.25,-1) {\footnotesize ($\partial \partial S\cdot \partial S)\cdot \partial\partial S$};
\node (R3)  at (8.25,-2) {\footnotesize $\partial S\cdot \partial\partial S$};
\node (R4)  at (8.25,-3) {\footnotesize $\partial \partial S\cdot \partial S$};
\node (R5)  at (8.25,-4) {\footnotesize $\partial \partial S\cdot \partial S$};
\node (BM)  at (4.125,-5) {\footnotesize $\partial S$};
\node (IL1)  at (2.75,-2) {\tiny ($\partial \partial S\cdot \partial\partial S)\cdot \partial S$};
\node (I)  at (1.375,-3) { $I$};
\node (J)  at (4.125,-4) { $J$};
\node (IL2)  at (2.75,-2.67) {\tiny ($\partial \partial S\cdot \partial\partial S)\cdot \partial S$};
\node (IR3)  at (5.5,-3.34) {\tiny $\partial \partial S\cdot (\partial\partial S\cdot \partial S)$};
\node (IL3)  at (2.75,-3.34) {\tiny ($\partial \partial S\cdot \partial\partial S)\cdot \partial S$};
\node (IL4)  at (2.75,-4) {\tiny $\partial (\partial \partial S\cdot \partial S)\cdot \partial S$};
\node (IL0)  at (2.75,-1) {\tiny $\partial \partial  S\cdot ( \partial\partial S\cdot \partial S)$};
\node (IL12)  at (2.75,-1.5) {\tiny $(\partial \partial  S\cdot \partial\partial S)\cdot \partial S$};
\node (IR0)  at (5.5,-1) {\tiny $\partial \partial  S\cdot (\partial S\cdot \partial\partial S)$};
\node (IR1)  at (5.5,-2.34) {\tiny $\partial \partial  S\cdot (\partial\partial S\cdot \partial S)$};
\path[->,font=\footnotesize]
(L1) edge node[sloped,above,yshift=-0.05cm]{\scriptsize ${\tt c}\cdot{\it id}$} (IL12)
(IL12) edge node[above]{\scriptsize $\alpha$} (IR1)
(L1) edge node[above]{$\scriptsize\alpha^{\cdot}$} (M1)
(L1) edge node[sloped,above,yshift=-0.05cm]{\scriptsize $\alpha$} (IL0)
(IL0) edge node[above]{\scriptsize ${\it id}\cdot{\tt c}$} (IR0)
(IR0) edge node[above]{\scriptsize ${\alpha^{\cdot}}^{-1}$} (R2)
(R2) edge node[above]{\scriptsize ${\tt c}$} (IR1)
(IR1) edge node[above,sloped,yshift=-0.05cm]{\scriptsize ${\it id}\cdot \nu$} (R4)
(L1) edge node[left]{\scriptsize $(\varphi^{-1}\circ i_r)\cdot{\it id}$} (L2)
(L2) edge node[left]{\scriptsize $\partial({\tt ex}\cdot{\it id}) \cdot {\it id}$} (L3)
(L3) edge node[left]{\scriptsize $\partial{\nu}\cdot{\it id}$} (L4)
(L4) edge node[left]{\scriptsize ${\tt ex}\cdot{\it id}$} (L5)
(L5) edge node[below]{\scriptsize $\nu$} (BM)
(R1) edge node[right]{\scriptsize $({\tt ex}\cdot{\it id})\!\cdot\!{\it id}$} (R2)
(R2) edge node[right]{\scriptsize $\nu\cdot{\it id}$} (R3)
(R3) edge node[right]{\scriptsize ${\tt c}$} (R4)
(R4) edge node[right]{\scriptsize ${\tt ex}\cdot{\it id}$} (R5)
(R5) edge node[below]{\scriptsize $\nu$} (BM)
(M1) edge node[above]{\scriptsize ${\tt ex}\cdot {\tt c}$} (M2)
(L1) edge node[sloped,below]{\tiny $({\tt ex}\cdot{\it id})\cdot{\it id}$} (IL1)
(IL1) edge node[above]{\tiny $ (\varphi^{-1}\circ i_r)\!\cdot\!{\it id}$} (L3)
(IL4) edge node[above]{\tiny $\partial\nu\!\cdot\!{\it id}$} (L5)
(IL1) edge node[left]{\tiny $ ({\tt ex}\cdot{\tt ex})\!\cdot\!{\it id}$} (IL2)
(IL2) edge node[left]{\tiny ${\tt c}\!\cdot\!{\it id}$} (IL3)
(IL3) edge node[left]{\tiny $ (\varphi^{-1}\circ i_r)\!\cdot\!{\it id}$} (IL4)
(IL3) edge node[above]{\tiny $\alpha$} (IR3)
(IR1) edge node[right]{\tiny ${\tt ex}\cdot{\it id}$} (IR3)
(IR3) edge node[sloped,above]{\tiny ${\it id}\cdot\nu$} (R5)
(M2) edge node[above]{\scriptsize ${\alpha^{\cdot}}^{-1}$} (R1);
\end{tikzpicture}\\
\vspace{0.1cm}
{\small Diagram 8.}
\end{center}

\indent The last equality is $\rho'_2\circ (\rho'_1\cdot{\it id})=\rho'_1\circ(\partial\rho'_1\cdot{\it id})\circ(\varphi^{-1}\cdot{\it id})\circ(i_r\cdot{\it id})\circ\alpha^{-1}\circ({\tt ex}\cdot {\tt c})\circ{\alpha},$ and it follows from the second one since $({\tt ex}\cdot{\tt c})^{-1}={\tt ex}\cdot{\tt c}$. \\[0.1cm]
\indent \texttt{(CA2)} By similar analysis as we did for \texttt{(CA1)}, it can be shown that $\texttt{(CA2)}$ follows from the equalities  \vspace{-0.1cm} \begin{equation}\label{e1}\rho_1\circ(\partial\partial\eta_{\EuScript C}\cdot {\it id})=\partial\pi_1\circ\partial\lambda^{\blacktriangle}\circ(\varphi^{-1}\circ i_l) : \partial\partial E_2\cdot\partial S\rightarrow \partial S \vspace{-0.1cm}\end{equation} and   \begin{equation}\label{u2}\rho_2\circ(\partial\eta_{\EuScript C}\cdot {\it id})=\partial\pi_1\circ\partial\lambda^{\blacktriangle}\circ(\varphi^{-1}\circ i_r) : \partial E_2\cdot\partial\partial S\rightarrow \partial S.  \end{equation}
\indent In Diagram 9,  the inner triangle represents the equation \eqref{e1}. It  commutes by the commutations of the three diagrams that surround it  (easy to check) and from the commutation of the outer triangle, which represents the left triangle from the axiom \texttt{(OA2)}. The  equality \eqref{u2} is verified by a similar diagram, whose outer part will commute as the right triangle from  \texttt{(OA2)}.

\begin{center}
\begin{tikzpicture}[scale=1.5]
\node (A)  at (0,1.2) {\footnotesize $\partial\partial E_2\cdot \partial S$};
\node (G)  at (-2,1.9) {\footnotesize $\partial E_1\cdot \partial S$};
\node (B) at (2.8,1.2) {\footnotesize $\partial\partial S\cdot \partial S$};
\node (F) at (4.8,1.9) {\footnotesize $\partial\partial S\cdot \partial S$};
\node (C) at (0.57,0.25) {\footnotesize $\partial (\partial E_2\cdot\partial S)$};
\node (D) at (1.4,-1.45) {\footnotesize $\partial S$};
\node (D1) at (1.45,0.5) {\small $D_1$};
\node (E) at (0.95,-0.55) {\footnotesize $\partial S^{\bullet}$ };
\path[->,font=\scriptsize]
(A) edge node[above]{$\partial\partial \eta_{\EuScript C}\cdot{\it id}$} (B)
(A) edge node[right,xshift=0.1cm]{$\partial\epsilon_2\cdot{\it id}$} (G)
(A) edge node[right]{$\varphi^{-1}\circ i_l$} (C)
(C) edge node[right]{$\partial\lambda^{\blacktriangle}$} (E)
(B) edge node[right]{$\rho_1$} (D)
(F) edge node[right,yshift=-0.1cm]{$\nu$} (D)
(B) edge node[left,xshift=-0.1cm]{${\tt ex}\cdot{\it id}$} (F)
(G) edge node[above]{${\eta_{\EuScript O}}\cdot{\it id}$} (F)
(G) edge node[left]{${\lambda^{\star}}$} (D)
(E) edge node[right,yshift=0.1cm,xshift=-0.1cm]{$\partial\pi_1$} (D);
\end{tikzpicture}\\
\vspace{0.1cm}
{\small Diagram 9.}
\end{center}
{\textsc{[entries-only $\Rightarrow$ exchangeable-output]}}  Given a  cyclic operad  ${\EuScript C}=(S,\rho^{_S},{\eta_2}^{_S})$, we define  ${\EuScript O}_{\EuScript C}=(\partial S,{\nu_{\rho}}^{_{\partial S}},{\eta_1}^{_{\partial S}},{\tt ex}_S)$ by introducing ${\nu_{\rho}}^{_{\partial S}}:\partial S\star \partial S\rightarrow \partial S$ ($\nu$ for short) as  $$\nu:\partial\partial S\cdot\partial S\xrightarrow{\enspace{\tt ex}\cdot{\it id}\enspace}\partial\partial S\cdot\partial S\xrightarrow{\enspace i_l\enspace} \partial\partial S\cdot\partial S+\partial S\cdot \partial\partial S\xrightarrow{\enspace\varphi^{-1}\enspace}\partial({\partial S\cdot \partial S})\xrightarrow{\enspace\partial\rho\enspace}\partial S, $$ and   ${\eta_1}^{_{\partial S}}:E_1\rightarrow \partial S$ ($\eta_1$ for short) as   $\eta_1=\partial{\eta_2}^{_S}\circ\epsilon^{-1}_2.$ We now indicate how to verify the axioms.\\[0.1cm]
\indent \texttt{[OA1]} The verification of \texttt{[OA1]} for  ${\EuScript O}_{\EuScript C}$ uses   equalities $(\partial{\tt ex}\circ{\tt ex})^{3}={\it id}_{\partial\partial S}$ and $\partial\rho_{2}\circ{\partial\gamma}=\partial\rho_{1}$. The outer part of Diagram 10 represents the equality $\nu_{21}\circ\beta_1=\nu_{11}$ (once the definition of $\nu$ via $\rho$  is unfolded). The proof that it commutes uses\\[-0.7cm]
\begin{itemize}
\item the commutation of the diagram $E$, where $\psi=(\varphi^{-1}\circ i_l)\circ((\varphi^{-1}\circ i_l)\cdot{\it id}) \circ(({\tt ex}\cdot{\it id})\cdot{\it id})\circ ((\partial{\tt ex}\cdot{\it id})\cdot{\it id}),$ which follows  by the equality $\partial{\tt ex}\circ{\tt ex}\circ\partial{\tt ex}={\tt ex}\circ\partial{\tt ex}\circ{\tt ex}$,\\[-0.7cm]
\item the commutation of the diagram $G$, which represents the equality $\partial\rho_{21}\circ{\partial\gamma_1}=\partial\rho_{11}$, \\[-0.7cm]
\item the commutations of $F_1$ and $F_2$, which are simple ``renaming'' diagrams, and\\[-0.7cm]
\item the commutations of $R_1$ and $R_2$, which follow by the naturality of $\rho$.
\end{itemize}
The outer part of Diagram 11, which represents the equality $\nu_{22}\circ\beta_2=\nu_{12}$, commutes by \\[-0.7cm]
\begin{itemize}
\item the commutation of the diagram $E$,   where $\phi_1={\varphi^{-1}}\circ i_l\circ(\varphi^{-1}\circ i_l\circ({\tt ex}\cdot {\it id}))\cdot{\it id}$ and $\phi_2=\varphi^{-1}\circ i_l\circ({\tt ex}\cdot {\it id})$, which follows by the naturality of $\gamma_2$ (notice that $\gamma_2=({{\it id}\cdot({\tt ex}\cdot{\it id})})\circ\alpha$),\\[-0.7cm]
\item the commutation of $G$, which represents the equality $\partial\rho_{22}\circ{\partial\gamma_2}=\partial\rho_{12}$, and\\[-0.7cm]
\item the commutations of $F_1$, $F_2$, $R_1$ and $R_2$, which are of the same kind as in Diagram 10.\\[-0.7cm]
\end{itemize}
Notice that the equality $\nu_{23}\circ\beta_3=\nu_{13}$ also follows by the commutation of Diagram 11. \\[0.1cm]
\indent  \texttt{[OA2]} The commutation of the left triangle from \texttt{[OA2]} follows by replacing  $\nu$ in Diagram 9 with its definition via $\rho$ and by setting $\rho_1=\partial\rho\circ\varphi^{-1}\circ i_l$. The commutation of the right triangle follows analogously, relative to the diagram that arises in the proof of the equality \eqref{u2}. 
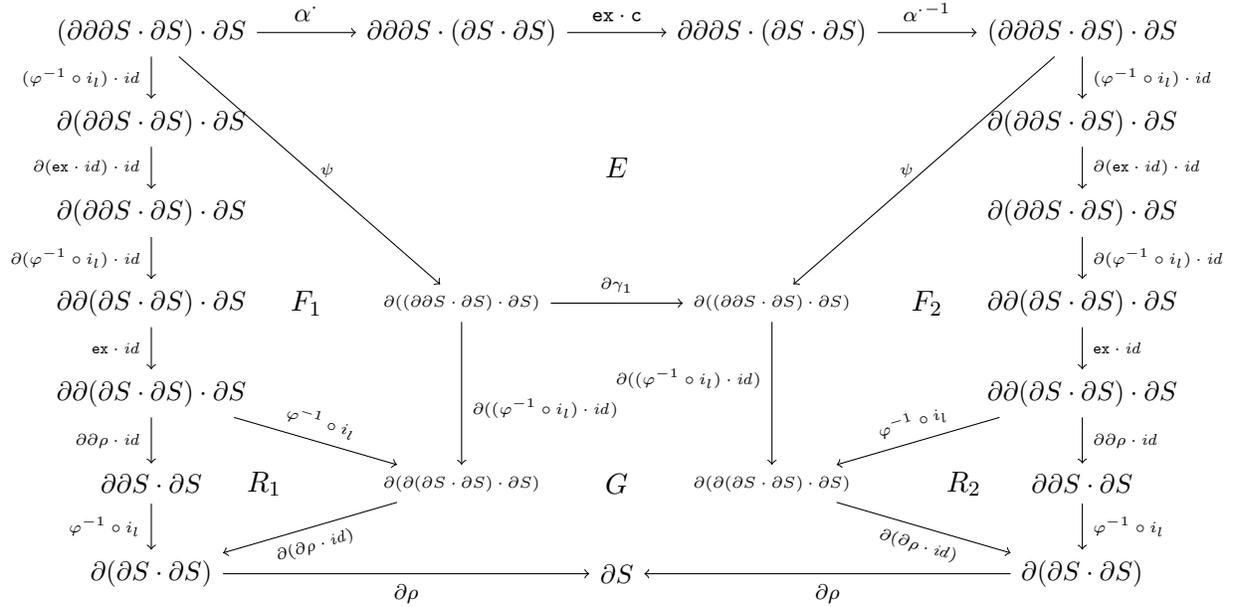
\begin{center}
\begin{tikzpicture}[scale=1.5]
\node (L1)  at (0,0) {\footnotesize ($\partial \partial \partial S\cdot \partial S)\cdot \partial S$};
\node (L15)  at (0,-4) {\footnotesize $\partial \partial S\cdot \partial S$};
\node (L12)  at (0,-4.8) {\footnotesize $\partial (\partial S\cdot \partial S)$};
\node (F)  at (2.75,-4) {\tiny $\partial(\partial(\partial S\cdot\partial S)\cdot\partial S)$};
\node (G)  at (5.5,-4) {\tiny $\partial(\partial(\partial S\cdot\partial S)\cdot\partial S)$};
\node (A)  at (2.75,-2.4) {\tiny $\partial((\partial\partial S\cdot\partial S)\cdot \partial S)$};
\node (B)  at (5.5,-2.4) {\tiny $\partial((\partial\partial S\cdot\partial S)\cdot \partial S)$};
\node (R12)  at (8.25,-4.8) {\footnotesize $\partial (\partial S\cdot \partial S)$};
\node (L25)  at (0,-3.2) {\footnotesize $\partial \partial (\partial S\cdot\partial S)\cdot \partial S$};
\node (L35)  at (0,-2.4) {\footnotesize $\partial \partial (\partial S\cdot\partial S)\cdot \partial S$};
\node (L45)  at (0,-1.6) {\footnotesize $\partial (\partial\partial S\cdot\partial S)\cdot \partial S$};
\node (L55)  at (0,-0.8) {\footnotesize $\partial (\partial\partial S\cdot\partial S)\cdot \partial S$};
\node (M1)  at (2.75,0) {\footnotesize $\partial \partial \partial S\cdot (\partial S\cdot \partial S)$};
\node (M2)  at (5.5,0) {\footnotesize $\partial \partial \partial S\cdot (\partial S\cdot \partial S)$};
\node (R1) at (8.25,0) {\footnotesize ($\partial \partial \partial S\cdot \partial S)\cdot \partial S$};
\node (R15)  at (8.25,-4) {\footnotesize $\partial \partial  S\cdot\partial S$};
\node (R25)  at (8.25,-3.2) {\footnotesize $\partial \partial (\partial S\cdot\partial S)\cdot \partial S$};
\node (R35)  at (8.25,-2.4) {\footnotesize $\partial \partial (\partial S\cdot\partial S)\cdot \partial S$};
\node (R45)  at (8.25,-1.6) {\footnotesize $\partial (\partial\partial S\cdot\partial S)\cdot \partial S$};
\node (R55)  at (8.25,-0.8) {\footnotesize $\partial (\partial\partial S\cdot\partial S)\cdot \partial S$};
\node (BM)  at (4.125,-4.8) {\footnotesize $\partial S$};
\node (E)  at (4.125,-1.2) {\small $E$};
\node (K)  at (4.125,-4) {\small $G$};
\node (F1)  at (1.375,-2.4) {\small $F_1$};
\node (F2)  at (6.875,-2.4) {\small $F_2$};
\node (Rg1)  at (1,-4) {\small $R_1$};
\node (Rg2)  at (7.2,-4) {\small $R_2$};
\path[->,font=\footnotesize]
(L1) edge node[above]{$\scriptsize\alpha^{\cdot}$} (M1)
(L1) edge node[left]{\tiny $(\varphi^{-1}\circ{i_l})\cdot{\it id}$} (L55)
(L1) edge node[right]{\tiny $\psi$} (A)
(R1) edge node[left]{\tiny $\psi$} (B)
(L25) edge node[left]{\tiny $\partial\partial\rho\cdot{\it id}$} (L15)
(L15) edge node[left]{\tiny $\varphi^{-1}\circ i_l$} (L12)
(A) edge node[right,yshift=-0.2cm]{\tiny $\partial((\varphi^{-1}\circ i_l)\cdot {\it id})$} (F)
(B) edge node[left,yshift=0.2cm]{\tiny $\partial((\varphi^{-1}\circ i_l)\cdot {\it id})$} (G)
(A) edge node[above]{\tiny $\partial\gamma_1$} (B)
(F) edge node[below,sloped]{\tiny $\partial(\partial\rho\cdot{\it id})$} (L12)
(G) edge node[below,sloped]{\tiny $\partial(\partial\rho\cdot{\it id})$} (R12)
(R15) edge node[right]{\tiny $\varphi^{-1}\circ i_l$} (R12)
(L25) edge node[above,sloped]{\tiny $\varphi^{-1}\circ i_l$} (F)
(L45) edge node[left]{\tiny $\partial(\varphi^{-1}\circ i_l)\cdot{\it id}$} (L35)
(L55) edge node[left]{\tiny $\partial({\tt ex}\cdot{\it id})\cdot{\it id}$} (L45)
(L12) edge node[below]{\scriptsize $\partial\rho$} (BM)
(R12) edge node[below]{\scriptsize $\partial\rho$} (BM)
(R1) edge node[right]{\tiny $(\varphi^{-1}\circ{i_l})\cdot{\it id}$} (R55)
(R25) edge node[right]{\tiny $\partial\partial\rho\cdot{\it id}$} (R15)
(R35) edge node[right]{\tiny ${\tt ex}\cdot {\it id}$} (R25)
(L35) edge node[left]{\tiny ${\tt ex}\cdot {\it id}$} (L25)
(R25) edge node[above,sloped]{\tiny $\varphi^{-1}\circ i_l$} (G)
(R45) edge node[right]{\tiny $\partial(\varphi^{-1}\circ i_l)\cdot{\it id}$} (R35)
(R55) edge node[right]{\tiny $\partial({\tt ex}\cdot{\it id})\cdot{\it id}$} (R45)
(M1) edge node[above]{\scriptsize ${\tt ex}\cdot {\tt c}$} (M2)
(M2) edge node[above]{\scriptsize ${\alpha^{\cdot}}^{-1}$} (R1);
\end{tikzpicture}
\vspace{0.1cm}
{\small Diagram 10.}
\end{center}
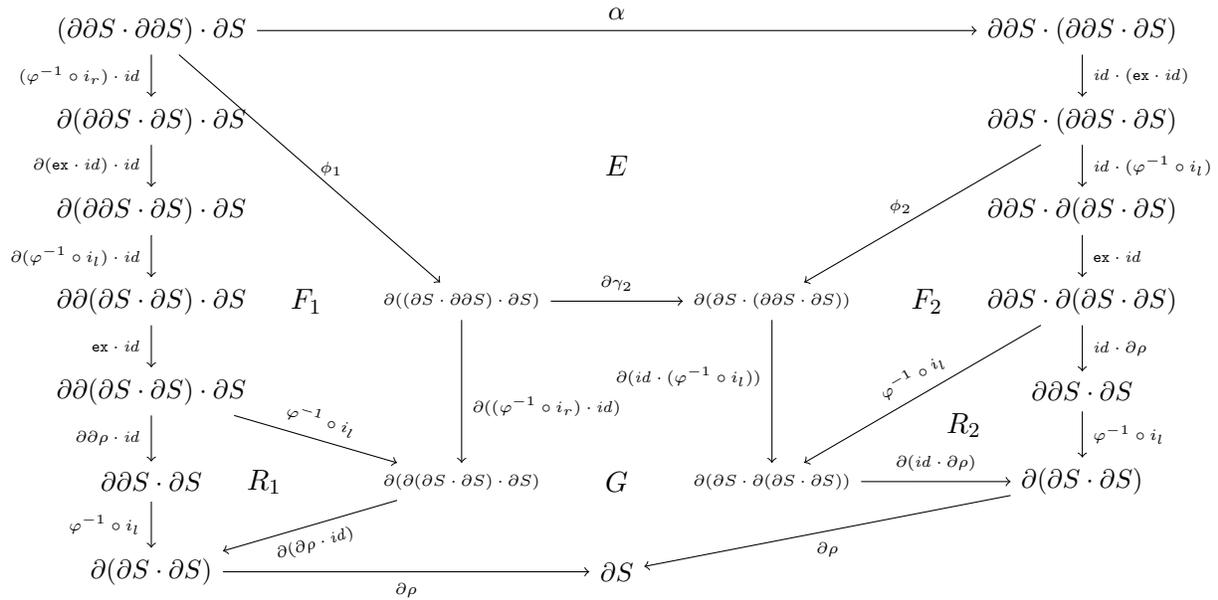
\vspace{-0.75cm}
\begin{center}
\begin{tikzpicture}[scale=1.5]
\node (L1)  at (0,0) {\footnotesize ($\partial \partial  S\cdot \partial\partial S)\cdot \partial S$};
\node (F)  at (2.75,-4) {\tiny $\partial(\partial(\partial S\cdot\partial S)\cdot \partial S)$};
\node (B)  at (5.5,-2.4) {\tiny $\partial(\partial S\cdot(\partial\partial S\cdot\partial S))$};
\node (G)  at (5.5,-4) {\tiny $\partial(\partial S\cdot \partial(\partial S\cdot\partial S))$};
\node (A)  at (2.75,-2.4) {\tiny $\partial((\partial S\cdot\partial\partial S)\cdot \partial S)$};
\node (L15)  at (0,-4) {\footnotesize $\partial \partial S\cdot \partial S$};
\node (L16)  at (0,-4.8) {\footnotesize $\partial (\partial S\cdot \partial S)$};
\node (L25)  at (0,-3.2) {\footnotesize $\partial \partial (\partial S\cdot\partial S)\cdot \partial S$};
\node (L35)  at (0,-2.4) {\footnotesize $\partial \partial (\partial S\cdot\partial S)\cdot \partial S$};
\node (L45)  at (0,-1.6) {\footnotesize $\partial (\partial\partial S\cdot\partial S)\cdot \partial S$};
\node (L55)  at (0,-0.8) {\footnotesize $\partial (\partial\partial S\cdot\partial S)\cdot \partial S$};
\node (R1) at (8.25,0) {\footnotesize $\partial \partial  S\cdot (\partial\partial S\cdot \partial S)$};
\node (R15)  at (8.25,-4) {\footnotesize $\partial (\partial  S\cdot\partial S)$};
\node (R25)  at (8.25,-3.2) {\footnotesize $\partial \partial  S\cdot\partial S$};
\node (R35)  at (8.25,-2.4) {\footnotesize $\partial\partial S\cdot\partial(\partial S\cdot \partial S)$};
\node (R45)  at (8.25,-1.6) {\footnotesize $\partial\partial S\cdot\partial(\partial S\cdot \partial S)$};
\node (R55)  at (8.25,-0.8) {\footnotesize $\partial \partial  S\cdot (\partial\partial S\cdot \partial S)$};
\node (BM)  at (4.125,-4.8) {\footnotesize $\partial S$};
\node (E)  at (4.125,-1.2) {\small $E$};
\node (K)  at (4.125,-4) {\small $G$};
\node (F1)  at (1.375,-2.4) {\small $F_1$};
\node (F2)  at (6.875,-2.4) {\small $F_2$};
\node (Rg1)  at (1,-4) {\small $R_1$};
\node (Rg2)  at (7.2,-3.5) {\small $R_2$};
\path[->,font=\footnotesize]
(F) edge node[below,sloped]{\tiny $\partial(\partial\rho\cdot{\it id})$} (L12)
(A) edge node[above]{\tiny $\partial\gamma_2$} (B)
(G) edge node[above]{\tiny $\partial({\it id}\cdot{\partial\rho})$} (R15)
(B) edge node[left,yshift=0.2cm]{\tiny $\partial({\it id}\cdot(\varphi^{-1}\circ i_l))$} (G)
(L1) edge node[right]{\tiny $\phi_1$} (A)
(R55) edge node[left,yshift=0.1cm]{\tiny $\phi_2$} (B)
(A) edge node[right,yshift=-0.2cm]{\tiny $\partial((\varphi^{-1}\circ i_r)\cdot {\it id})$} (F)
(L1) edge node[above]{$\scriptsize\alpha$} (R1)
(L1) edge node[left]{\tiny $(\varphi^{-1}\circ{i_r})\cdot{\it id}$} (L55)
(L25) edge node[left]{\tiny $\partial\partial\rho\cdot{\it id}$} (L15)
(L15) edge node[left]{\tiny $\varphi^{-1}\circ i_l$} (L16)
(R15) edge node[below]{\tiny $\partial\rho$} (BM)
(L16) edge node[below]{\tiny $\partial\rho$} (BM)
(L45) edge node[left]{\tiny $\partial(\varphi^{-1}\circ i_l)\cdot{\it id}$} (L35)
(L55) edge node[left]{\tiny $\partial({\tt ex}\cdot{\it id})\cdot{\it id}$} (L45)
(R1) edge node[right]{\tiny ${\it id}\cdot({\tt ex}\cdot{\it id})$} (R55)
(R25) edge node[right]{\tiny $\varphi^{-1}\circ i_l$} (R15)
(R35) edge node[right]{\tiny ${\it id}\cdot\partial\rho$} (R25)
(R35) edge node[above,sloped]{\tiny $\varphi^{-1}\circ i_l$} (G)
(L35) edge node[left]{\tiny ${\tt ex}\cdot {\it id}$} (L25)
(R45) edge node[right]{\tiny ${\tt ex}\cdot{\it id}$} (R35)
(L25) edge node[above,sloped]{\tiny $\varphi^{-1}\circ i_l$} (F)
(R55) edge node[right]{\tiny ${\it id}\cdot(\varphi^{-1}\circ i_l)$} (R45);
\end{tikzpicture}
\vspace{0.1cm}
{\small Diagram 11.}
\end{center}
\textsc{[the isomorphism of cylcic operads ${\EuScript C}$ and ${\EuScript C}_{{\EuScript O}_{\EuScript C}}$ (and ${\EuScript O}$ and ${\EuScript O}_{{\EuScript C}_{\EuScript O}}$)]}  As it was the case in the proof of Theorem \ref{2}, the isomorphism at the level of the underlying species exists by Lemma \ref{fff}.  The first isomorphism of cyclic operads follows from the equalities\footnotemark \footnotetext{These would be {\em isomorphisms} rather than {\em equalities} if we considered the sequence $S\rightarrow\partial S\rightarrow \int\partial S$ instead of $S\rightarrow\partial S\rightarrow S$.} $\partial{\eta_2}^{_S}=\partial{\eta_2}^{_S}\circ\epsilon^{-1}_2\circ\epsilon_2$ and $\partial{\rho}^{_S}=\rho'_{\nu_{\rho}^{_{\partial S}}}$. As for the second equality,
since $\rho'_{\nu_{\rho}^{_{\partial S}}}\circ \varphi^{-1}=[{\rho'_1}_{\nu_{\rho}^{_{\partial S}}},{\rho'_2}_{\nu_{\rho}^{_{\partial S}}}]$, it follows from equalities $\partial\rho^{_S}\circ\varphi^{-1}\circ i_l={\rho'_1}_{\nu_{\rho}^{_{\partial S}}}$   and  $\partial\rho^{_S}\circ\varphi^{-1}\circ i_r={\rho'_2}_{\nu_{\rho}^{_{\partial S}}}.$ We have \vspace{-0.2cm} {\small $$
{\rho'_1}_{\nu_{\rho}^{_{\partial S}}}=\nu_{\rho}^{_{\partial S}}\circ({\tt ex}\cdot{\it id}_{\partial S})=\partial\rho^{_S}\circ\varphi^{-1}\circ i_l\circ ({{\tt ex}}\cdot{\it id}_{\partial S})\circ({\tt ex}\cdot{\it id}_{\partial S})=\partial\rho^{_S}\circ\varphi^{-1}\circ i_l
\vspace{-0.3cm} $$}

\noindent and  {\small \vspace{-0.2cm}$$\begin{array}{rcl}
{\rho'_2}_{\nu_{\rho}^{_{\partial S}}}&=&\nu_{\rho}^{_{\partial S}}\circ({\tt ex}\cdot{\it id}_{\partial S})\circ{\tt c}\enspace=\enspace\partial\rho^{_S}\circ\varphi^{-1}\circ i_l\circ ({{\tt ex}}\cdot{\it id}_{\partial S})\circ({\tt ex}\cdot{\it id}_{\partial S})\circ{\tt  c}\\[0.1cm]
&=&\partial\rho^{_S}\circ\varphi^{-1}\circ i_l\circ{\tt  c}\enspace=\enspace\partial\rho^{_S}\circ\partial{\tt c}\circ{\varphi^{-1}}\circ  i_r\enspace=\enspace\partial\rho^{_S}\circ{\varphi^{-1}}\circ  i_r\,.
\end{array}\vspace{-0.2cm}$$}

\indent The second isomorphism follows from the equalities\footnote{Like before, these would be {\em isomorphisms}  if we considered the sequence $S\rightarrow\int S\rightarrow \partial\int S$.}  $\partial \eta_{1}^{_{\partial  S}}=\partial(\eta_1^{_{\partial S}}\circ\epsilon_2)\circ\epsilon_2^{-1}$ and {\small $$
\begin{array}{rcl}
\nu_{\rho_{\nu}}&=&\partial\rho_{\nu}\circ\varphi^{-1}\circ i_l\circ({\tt ex}\cdot{\it id})\\[0.1cm]
&=&[\nu\circ({\tt ex}\cdot{\it id}_{\partial S}),\nu\circ({\tt ex}\cdot{\it id}_{\partial S})\circ{\tt c}]\circ\varphi\circ\varphi^{-1}\circ i_l\circ({\tt ex}\cdot{\it id})\\[0.1cm]
&=&\nu\circ({\tt ex}\cdot{\it id}_{\partial S})\circ({\tt ex}\cdot{\it id})\enspace=\enspace\nu .
\end{array}\vspace{-0.3cm}$$}
\end{proof}
\section*{Conclusion}
Given a category ${\bf C}$ equipped with a bifunctor $\diamond:{\bf C}\times{\bf C}\rightarrow {\bf C}$ that does not bear a monoidal structure, a question of finding the ``minimal" associativity-like and unit-like isomorphisms can be asked, which  leads to categorifications of monoid-like algebraic structures in a way analogous to the one illustrated in Table 1 in Introduction. If such isomorphisms (and unit-like objects) are established, we could say that ${\bf C}$ (together with the additional structure) is a monoidal-{\em like} category and define in a natural way a monoid-like object in ${\bf C}$ - this is what the microcosm principle is about. In this paper we exhibited one such monoidal-like category: $({\bf Spec},\blacktriangle,E_2)$, which allowed us to deliver the algebraic definition of entries-only cyclic operads (Definition \ref{copeo}): 
\begin{center}
{\em A cyclic operad is a monoid-like object in the monoidal-like category $({\bf Spec},\blacktriangle,E_2)$,}
\end{center}
represented more explicitly in Table 4.
\begin{center}  \begin{savenotes}{\small 
  \begin{tabular}{rcc}  
    \toprule
       &  \textsc{Monoidal-like category}  ${\bf Spec}$  & \textsc{Monoid-like object} $S\in {\bf Spec}$ \\
    \midrule
    {\small{\textsc{product\enspace}  }}     &{\small{$\blacktriangle:{\bf Spec}\times{\bf Spec}\rightarrow {\bf Spec}$}}    & {\small{$\rho: S\blacktriangle S\rightarrow S$}}      \\[0.1cm]
    {\small{\textsc{unit\,\,\,\,\,}}}  &{\small{$E_2\in {\bf Spec}$}}    & {\small{$\eta_2:E_2\rightarrow S$}}      \\[0.1cm]
{\small \begin{tikzpicture}[scale=1.5]
\node (A)  at (0,0.38) {\textsc{associativity}\footnotemark \footnotetext{Actually, the ``minimal" associativity-like isomorphism.}};
\node (B)  at (0,0) {};
\end{tikzpicture}  }  
 &{\scriptsize
\begin{tikzpicture}[scale=1.5]
\node (A)  at (0,1) {$\gamma_{S,T,U}:(S\blacktriangle T)\blacktriangle U+T\blacktriangle(S\blacktriangle U)+(T\blacktriangle U)\blacktriangle S$};
\node (B)  at (0.275,0.7) {$\rightarrow S\blacktriangle (T\blacktriangle U)+(S\blacktriangle U)\blacktriangle T+U\blacktriangle (S\blacktriangle T)$};
\end{tikzpicture}  }  
     & {\small \begin{tikzpicture}[scale=1.5]
\node (A)  at (0,0.42) {\texttt{(CA1)}};
\node (B)  at (0,0) {};
\end{tikzpicture}  }      \\[0.065cm]
 {\small{\textsc{left unit}\footnote{Analogously, the ``minimal" unit-like isomorphism.} }}  
 &{\small{$\lambda^{\blacktriangle}_S:E_2\blacktriangle S\rightarrow S^{\bullet}$ }}  & {\small{\texttt{(CA2)} }}      \\[0.065cm]
\midrule
 {\small{\textsc{right unit}\footnotemark[9] }}
 &{\small{$\kappa^{\blacktriangle}_S:S\blacktriangle E_2\rightarrow S^{\bullet}$ }}  & {\small{Corolary \ref{run}}}   \\
    \bottomrule
  \end{tabular}

\vspace{0.25cm}
Table 4. A cyclic operad defined internally to the monoidal-like category of species. }
\end{savenotes}
\end{center}

\indent We also introduced the algebraic definition of exchangeable-output cyclic operads (Definition \ref{dddd}),
 by first upgrading the structure $({\bf Spec},\star)$, exhibited by Fiore, into the monoidal-like category $({\bf Spec},\star,E_1)$, and then by endowing the monoid-like objects of this category, i.e. operads (see Table 5), with a natural transformation that accounts for the ``input-output interchange".
\begin{center} \begin{savenotes} {\footnotesize
  \begin{tabular}{rccc}  
    \toprule
       &  \textsc{Monoidal-like category}  ${\bf Spec}$  & &\textsc{Monoid-like object} $S\in {\bf Spec}$ \\
    \midrule
    {\small{\textsc{product\enspace}  }}     &{\small{$\star:{\bf Spec}\times{\bf Spec}\rightarrow {\bf Spec}$}}    && {\small{$\nu: S\star S\rightarrow S$}}      \\[0.1cm]
    {\small{\textsc{unit\,\,\,\,\,}}}  &{\small{$E_1\in {\bf Spec}$}}    && {\small{$\eta_1:E_1\rightarrow S$}}      \\[0.1cm]
{\small \begin{tikzpicture}[scale=1.5]
\node (A)  at (0,0.38) {\textsc{associativity}\footnote{Actually, the ``minimal" associativity-like isomorphism.}  };
\node (B)  at (0,0) {};
\end{tikzpicture}  }  
 &{\footnotesize
\begin{tikzpicture}[scale=1.5]
\node (A)  at (0,1) {$\beta_{S,T,U}:(S\star T)\star U+S\star(U\star T)$};
\node (B)  at (0.275,0.7) {$\rightarrow S\star (T\star U)+(S\star U)\star T$};
\end{tikzpicture}  }  
     && {\small \begin{tikzpicture}[scale=1.5]
\node (A)  at (0,0.42) {\texttt{(OA1)}};
\node (B)  at (0,0) {};
\end{tikzpicture}  }      \\[-0.25cm]
{\small\begin{tikzpicture}
\node (A)  at (0,0.5) {};
\node (F)  at (1.14,0.53) {  {\small{\textsc{left unit} }}  };
\node (C) at (1.05,-0.05) {{\small{\textsc{right unit} }}  };
\node (D) at (0,0) {};
\node (E) at (1.2,-0.5) {};
\end{tikzpicture}  }
&{\small\begin{tikzpicture}[scale=1.5]
\node (A)  at (0,0.5) {};
\node (F)  at (1.05,0.2) { {\small{$\lambda^{\star}_S:E_1\star S\rightarrow S$}}};
\node (C) at (1.05,-0.2) {{\small{$\rho^{\star}_S:S\star E_1\rightarrow S$}}  };
\node (D) at (0,0) {};
\node (E) at (1.2,-0.5) {};
\end{tikzpicture}  } && {\small\begin{tikzpicture}[scale=1.5]
\node (A)  at (0.1,0.5) {\texttt{(OA2)}};
\node (B) at (0.5,0) {};
\node (C) at (0.5,0.5) {};
\end{tikzpicture}}    
  \\[-0.5cm]
    \bottomrule
  \end{tabular} 
\vspace{0.25cm}

Table 5. An operad defined internally to the monoidal-like category of species}
\end{savenotes}\end{center}

\indent The main correspondence that we established is the equivalence between these two algebraic definitions, which consolidates the equivalence between the two points of view on cyclic operads. \\[0.1cm]
\indent Future work will involve establishing the notion of {\em weak Cat-cyclic-operad}, i.e. a cyclic operad enriched over the category $\bf Cat$ of small categories, by replacing in Definition \ref{entriesonly} the category $\bf Set$ with $\bf Cat$ and the equations given by the axioms with isomorphisms in $\bf Cat$. We presume that the task of formulating the coherence conditions for these isomorphisms  represents a further categorification of the notion of monoid-like object in $({\bf Spec},\blacktriangle,E_2)$. This work is motivated by the article \cite{wco} of Do\v sen and Petri\' c, where they introduced the notion of {\em weak Cat-operad}.
\refs
\bibitem[{BD97}]{micro} J. C. Baez, J. Dolan, Higher-Dimensional Algebra III: n-Categories and the Algebra of Opetopes, {\em Adv. Math.}, {\bf 135} 145-206, 1998. 
\bibitem[BLL08]{species} F. Bergeron, G. Labelle, P. Leroux, Introduction to the Theory of Species of Structures, Universit\' e du Qu\' ebec \`a Montr\' eal, Montreal, 2008. 
\bibitem[CO16]{cyc1} P.-L. Curien, J. Obradovic, A formal language for cyclic operads, \href{http://arxiv.org/abs/1602.07502}{arXiv:1602.07502}, 2016. 
\bibitem[DP15]{wco} K. Do\v sen, Z. Petri\' c, Weak Cat-operads, {\em Log. Methods in Comp. Science}, {\bf 11}(1:10)  1-23, 2015. 
\bibitem[F14]{fiore} M. Fiore, Lie Structure and Composition, CT2014, University of Cambridge, \url{http://www.cl.cam.ac.uk/~mpf23/talks/CT2014.pdf}, 2014. 
\bibitem[GK95]{Getzler:1994pn}
E. Getzler, M. Kapranov, Cyclic operads and cyclic homology, {\em Geom., Top., and Phys. for Raoul Bott (S.-T. Yau, ed.), Conf. Proc. Lect. Notes. Geom. Topol., International Press}, {\bf 4} 167-201,   1995. 
\bibitem[J81]{joyal} A. Joyal, Une th\' eorie combinatoire des s\' eries formelles, {\em Advances in Mathematics}, {\bf 42} 1-82, 1981. 
\bibitem[K05]{kelly} G. M. Kelly, On the operads of J. P. May, {\em Repr. Theory Appl. Categ.}, {\bf 13} 1-13, 2005. 
\bibitem[Lam15]{l15} F. Lamarche, private communication, 2015.
\bibitem[Mar96]{mmm} M. Markl, Models for operads, {\em Comm. Algebra}, {\bf 24}(4) 1471-1500, 1996. 
\bibitem[Mar08]{opsprops} M. Markl, Operads and PROPs, {\em Elsevier, Handbook for Algebra}
Vol. 5 (2008), 87-140. 
\bibitem[Mar15]{modular} M. Markl, Modular envelopes, OSFT and nonsymmetric (non-$\Sigma$) modular operads, {\em J. Noncommut. Geom.} 10, 775-809, 2016.
\bibitem[May72]{GILS} J. P. May, The geometry of iterated loop spaces, {\em Lectures Notes in Mathematics},
{\bf 271}, Springer-Verlag, Berlin, 1972.
\endrefs

\end{document}